 \documentclass[11pt,reqno]{amsart}
\usepackage[english]{babel}
\usepackage[letterpaper,top=2cm,bottom=2cm,left=3cm,right=3cm,marginparwidth=1.75cm]{geometry}

\usepackage{multirow}        
\usepackage{multicol}        
\usepackage{subfigure}        
\usepackage{graphicx}
\usepackage{amsfonts}
\usepackage{bbold}
\usepackage{diagbox}
\usepackage[colorlinks=true, allcolors=blue]{hyperref}
\usepackage{mathrsfs}
\usepackage{indentfirst}
\usepackage{amssymb,amsmath,amsthm}
\usepackage{epsfig}
\usepackage{appendix}
\usepackage{lineno}
\usepackage{enumerate}
\usepackage{comment}
\usepackage{subeqnarray}
\usepackage{cases}
\usepackage{stmaryrd}
\usepackage{color}
\usepackage{mathrsfs}
\usepackage{tikz}
\usepackage[update,prepend]{epstopdf}
\usepackage{booktabs}
\usepackage[linesnumbered,ruled,lined]{algorithm2e}
\usepackage{epsfig,graphics,bm,bbm}
\usepackage[numbers,sort&compress]{natbib}
\usepackage{cellspace} 
\numberwithin{equation}{section}
\newcommand{\bq}{\begin{equation}}
\newcommand{\eq}{\end{equation}}
\def\beq{\begin{equation*}}
\def\eeq{\end{equation*}}
\def\br{\begin{eqnarray}}
\def\er{\end{eqnarray}}
\def\brr{\bq\begin{array}{r@{}l}}
\def\err{\end{array}\eq}
\def\bry{\beq\begin{array}{r@{}l}}
\def\ery{\end{array}\eeq}
\def\dt{\Delta t}
\def\x{{\boldsymbol x}}
\def\n{{\boldsymbol n}}

\def\R{{\mathbb R}}
\def\Z{{\mathbb Z}}

\newcommand{\dps}{\displaystyle}
\newtheorem{theorem}{\indent Theorem}[section]
\newtheorem{lemma}{\indent Lemma}[section]

\newtheorem{remark}{\indent Remark}[section]

\newcommand{\errN}{\mathcal N}
\newcommand{\errS}{\mathcal S}
\newcommand{\errK}{\mathcal K}
\newcommand{\errJ}{\mathcal J}
\newcommand{\ba}{\begin{array}}\newcommand{\ea}{\end{array}}
\newcommand{\be}{\begin{eqnarray}}\newcommand{\ee}{\end{eqnarray}}
\newcommand{\bex}{\begin{eqnarray*}}
\newcommand{\eex}{\end{eqnarray*}}

\newcommand{\myitem}[1]{%
	\par\begingroup
	\leftskip=3em
	\noindent
	\llap{\makebox[1.2em][l]{\large$\bullet$}}#1\par
	\endgroup
}
\date{\today}

\begin{document}
\graphicspath{{figures/},}

\title[Decoupled high-order schemes for dendritic crystal growth]
{A Novel Decoupled, Energy-Stable Runge-Kutta Framework of Arbitrary Order for the Anisotropic Phase-Field Dendritic Crystal Growth Model
}
\author{Weiwen Wang$^{1}$
\quad
Chuanju Xu$^{2}$
\quad
Shaoqin Zheng$^{3,*}$}
\thanks{\hskip -12pt
${}^{1}$School of Mathematics and Statistics, Guangdong University of Technology, Guangdong, Guangzhou, 510520, China.\\
${}^{2}$School of Mathematical Sciences, Xiamen university and Fujian Provincial Key Laboratory on Mathematical Modeling \& High Performance Scientific Computing, Xiamen University, Fujian 361005, China. \\
${}^{3}$Institute of Natural Sciences, Shanghai Jiao Tong University, Shanghai, 200240, China.\\
${}^{*}$Corresponding author.\\
The work of W.~Wang is supported by Guangdong Basic and Applied Basic Research Foundation (Grant No.~2025A1515110112). This work of C.~Xu is supported by the National Natural Science Foundation of China (Grants No.~12371408).\\
Email:  wwwang@gdut.edu.cn (W.~Wang),  cjxu@xmu.edu.cn (C.~Xu), sqzheng@sjtu.edu.cn (S.~Zheng)}

\keywords {Dendritic crystal growth, Auxiliary variable, Algebraically stable Runge-Kutta method, Matrix diagonalization, Temporal convergence analysis, Energy dissipation}
\subjclass[2020]{65M12, 65L06, 65M70, 74N20, 74N05}

\date {\today}
\begin{abstract}
This paper introduces a novel and general framework for constructing arbitrary-order, energy-stable, and decoupled schemes for the anisotropic phase-field dendritic crystal growth model. The proposed framework combines a time-dependent auxiliary variable, an algebraically stable Runge-Kutta method, and a matrix diagonalization technique. This unique combination yields a scheme that is proven to be unconditionally energy-dissipative while requiring only the solution of independent, constant-coefficient elliptic equations at each time step, which can be solved in parallel. The auxiliary variable is recovered from a uniquely solvable $q\times q$ algebraic system. We provide a rigorous proof of $q$th-order convergence in time for a $q$-stage Runge-Kutta method within a fixed Fourier-Galerkin space. Numerical experiments in 2D and 3D not only confirm the theoretical convergence and energy stability but also demonstrate the efficiency of the decoupled approach and provide a detailed investigation of key parameters' effects on dendritic morphology.

\end{abstract}
\maketitle

\section{Introduction}\label{sec1}
\setcounter{equation}{0}
Efficient numerical simulation of phase-field models remains a challenging task in scientific computing. The difficulties mainly stem from nonlinear free-energy terms, multiscale interfacial dynamics, and strong couplings among physical variables. Preserving the energy dissipation property is important for capturing the long-time behavior of phase-field models, while high-order time discretizations are useful for improving accuracy and reducing computational cost. Therefore, the development of numerical methods that combine high-order accuracy, energy stability, and computational efficiency remains an important research topic.

However, high-order accuracy and energy stability do not necessarily ensure computational efficiency. In the original implementation of the multi-stage Runge-Kutta scheme considered here, the stage variables are coupled through an elliptic system that must be solved at each time step. This repeated solution can be computationally expensive, particularly for fine spatial discretizations and multidimensional simulations.

Among various phase-field models, the anisotropic dendritic crystal growth model has attracted considerable attention. Dendritic patterns are commonly observed in snowflake formation, frost growth, metal solidification, and crystallization in supersaturated solutions \cite{honjo1982quantitative,glicksman1976dendritic,dougherty1988steady}. In metallic materials, dendritic structures are an important feature of the microstructure formed during solidification. Such patterns arise from the interaction between thermal diffusion and anisotropic interfacial effects during crystal growth and solidification processes. To describe the evolution of such interfaces, the phase-field method \cite{halperin1974renormalization,kobayashi1993modeling,1985Diffuse} introduces a continuous function $\phi$ in both time and space to distinguish the solid and liquid phases. The function takes distinct values in each phase and varies smoothly across the interfacial region. As a result, the interface is represented implicitly through the variation of $\phi$, eliminating the need for explicit interface tracking.
The governing equations of the phase-field model are typically derived from a free-energy functional through a variational principle, which ensures thermodynamic consistency of the system \cite{marinozzi1996phase,1998Universal,mullis2003study,demange2017phase}. The resulting model consists of an anisotropic Allen-Cahn equation coupled with a heat transfer equation through a latent heat term \cite{karma1998}.
Numerically, the small interfacial width introduces severe stiffness into the phase equation, anisotropic effects may lead to large oscillations, and the latent heat term creates a strong nonlinear coupling between the phase field and temperature. These features make the construction of robust high-order solvers particularly challenging.

To address these challenges, a variety of numerical approaches have been developed over the past two decades. Nevertheless, many existing schemes are nonlinear and require iterative solvers, while others fail to preserve the energy dissipation property \cite{li2011fast,provatas1998efficient,jeong2001phase}. This has motivated the introduction of the invariant energy quadratization (IEQ) and scalar auxiliary variable (SAV) approaches for constructing linear energy-stable schemes. Both frameworks introduce auxiliary variables to handle the nonlinear free-energy terms, allowing the construction of linear schemes that preserve the energy dissipation property. Several IEQ-based schemes have been proposed for anisotropic dendritic growth models \cite{zhao2017numerical,Yang2019}. These methods transform the nonlinear free-energy terms into quadratic forms and enable the construction of linear schemes. Subsequent work introduced stabilization and decoupling techniques to improve energy stability and reduce computational cost \cite{2019Zhang,2020Zhang}. Nevertheless, some formulations still require solving coupled systems with variable coefficients, while fully-decoupled versions are often restricted to low-order temporal discretizations.
The SAV approach has also been applied to anisotropic dendritic growth models \cite{yang2020efficient}, providing an alternative framework for constructing linear energy-stable schemes. Subsequent developments introduced multi-auxiliary-variable strategies, improved treatments of the coupling terms, and generalized auxiliary variable techniques with relaxation \cite{yang2021fully,Lmh2022,guo2024efficient}. These developments further improve computational efficiency and reduce the number of equations that need to be solved at each time step.

Despite these advances, many earlier energy-stable schemes were restricted to first- and second-order temporal accuracy. For long-time simulations, low-order methods may require small time steps and consequently lead to high computational cost. Wang and Xu \cite{wang2025class} developed arbitrarily
high-order energy-stable time-stepping schemes based on a time-dependent auxiliary variable and algebraically stable Runge-Kutta methods. 
However, their implementation requires solving a coupled elliptic system at each time step, which increases the computational cost, particularly at high spatial resolutions and in multidimensional simulations. In addition, a rigorous global temporal error estimate for this scheme has not yet been established.
More recently, Zhou and Guo \cite{zhou2026high} developed high-order, fully decoupled, and energy-stable RK-SAV/LDG schemes for the anisotropic phase-field dendritic growth model. Their approach combines a linear-implicit Runge-Kutta method with a local discontinuous Galerkin spatial discretization, leading to a sequence of constant-coefficient elliptic problems. 

Motivated by these developments, we focus on improving the computational efficiency of the high-order Runge-Kutta scheme in \cite{wang2025class} without modifying its underlying time discretization. By applying a matrix diagonalization technique to the coupled stage system, we transform it into independent elliptic equations with constant coefficients, which can be solved separately or in parallel. The auxiliary variables
are then determined from a uniquely solvable $q\times q$ algebraic system. Since the transformation is algebraically equivalent to the original scheme, the unconditional energy stability and temporal accuracy are preserved. Furthermore, we establish a rigorous global temporal error estimate and prove $q$th-order convergence for the $q$-stage method in a fixed Fourier-Galerkin space.

The main contributions of this work are summarized as follows:
\vspace{0.6em}

\myitem{A matrix diagonalization technique is developed for the arbitrary-order Runge-Kutta scheme in \cite{wang2025class}. The coupled system is transformed into independent elliptic equations with constant coefficients, thereby reducing the computational cost.}

\myitem{The proposed decoupled implementation retains the unconditional discrete energy dissipation law of the original algebraically stable Runge-Kutta scheme.}

\myitem{A rigorous temporal convergence analysis is established for the proposed Runge-Kutta scheme. For a fixed Fourier-Galerkin space, we prove $q$th-order convergence in time for the scheme based on a $q$-stage Gauss method, providing a theoretical justification for its temporal accuracy.}

\myitem{The temporal convergence and discrete energy dissipation are verified by two- and three-dimensional tests. Comparisons with the corresponding coupled scheme demonstrate the computational efficiency of the matrix diagonalization technique, with the decoupled scheme achieving speedups ranging from 3.6 to 18.5 in the tested cases. Comparisons among schemes of different orders illustrate the advantages of high-order temporal discretizations.}

\myitem{Numerical simulations are further performed to examine how the anisotropy coefficients, latent heat parameters, rotation angles, and initial nucleus shapes affect dendritic morphology.}

\vspace{0.6em}
The rest of this paper is organized as follows. Section \ref{sec2} describes the anisotropic phase-field model for dendritic crystal growth that satisfies the energy dissipation law. In Section \ref{sec3}, we introduce the auxiliary-variable Runge-Kutta discretization and use matrix diagonalization to decouple the resulting system. We further establish the discrete energy dissipation law and derive a temporal error estimate in a fixed Fourier-Galerkin space. The fully discrete scheme is given at the end of this section. Section \ref{sec4} presents accuracy tests and numerical simulations in two and three dimensions to validate the theoretical results and demonstrate the performance of the proposed schemes. This paper ends with some concluding remarks in Section \ref{sec5}. The corresponding Butcher tableaux for the Runge-Kutta methods satisfying the required conditions are provided in the Appendix.

\section{An anisotropic phase-field model for dendritic crystal growth}\label{sec2}
\setcounter{equation}{0}
In this section, we begin with the anisotropic phase-field model for dendritic crystal growth considered in \cite{karma1998}. Let $\Omega\subset\mathbb{R}^d$, $d=2,3$, be an open, bounded rectangular domain. We introduce a continuous scalar phase-field function $\phi=\phi(\x,t)$ to distinguish the solid and liquid phases:
\begin{equation*}\label{scalar_phi}
\phi(\x,t) = \begin{cases} 1,\qquad \mbox{solid},\\ -1,\quad \mbox{liquid}. \end{cases}
\end{equation*}

The dynamics of the model are governed by the total free energy, which is expressed in terms of the phase variable $\phi$ and the scaled temperature $T=T(\x,t)$:
\begin{equation}\label{aenergy}
E(\phi,T)=\displaystyle{\int_{\Omega}\Big(\dfrac{1}{2}\kappa^2(\nabla\phi)|\nabla\phi|^2+\frac{1}{\varepsilon^2}F(\phi)\Big)d\x+\int_{\Omega}\dfrac{\lambda}{2\varepsilon K} T^2\,d\x}.
\end{equation}
In this expression, the three terms represent the anisotropic interfacial energy, the double-well potential responsible for phase separation, and the thermal energy. The positive parameters $\varepsilon$, $\lambda$, and $K$ are associated with the interfacial thickness, the kinetic coefficient, and the latent heat effect during the phase transition, respectively. 
The Ginzburg-Landau double-well potential is given by $F(\phi)=\frac{1}{4}(\phi^2-1)^2$, and the anisotropy is described by $\kappa(\nabla\phi)$, which depends on the interface-normal direction $\n=-\frac{\nabla \phi}{|\nabla \phi|}$.
For two-dimensional problems, $\kappa(\nabla \phi)$ typically takes the following form \cite{li2011fast,meca2013phase,karma1996phase}:
\begin{equation*}\label{kappa}
\kappa(\nabla \phi)=1+\epsilon_{4} \cos (m (\theta-\theta_0)),
\end{equation*}
where $m$ is a positive integer denoting the fold number of the anisotropy, and $\epsilon_4\in[0,1)$ measures its strength. The angle $\theta=\arctan\left(\frac{\phi_y}{\phi_x}\right)$ describes the orientation of $\nabla\phi$, while $\theta_0$ is the prescribed rotation angle. For the two- and three-dimensional cases considered below, $\kappa(\nabla\phi)$ can be written explicitly as follows:

\begin{equation*}
\begin{aligned} 
\kappa(\nabla \phi)=
\begin{cases} 1 - 3\epsilon_4 + 4\epsilon_4\dfrac{\phi_x^4 + \phi_y^4}{|\nabla \phi|^{4}}, &\mbox{for 2D},\ m=4,\ \theta_0=0^{\circ};\\[10pt] 
1- \epsilon_{4}\dfrac{\phi_{y}^{2}-\phi_{x}^{2}}{|\nabla \phi|^{2}}\big(8\dfrac{\phi_{x}^{4}+\phi_{y}^{4}}{|\nabla \phi|^{4}} -7\big), &\mbox{for 2D},\ m=6,\ \theta_0=0^{\circ};\\[10pt] 
1+ \epsilon_{4}\dfrac{\phi_{y}^{2}-\phi_{x}^{2}}{|\nabla \phi|^{2}}\big(8\dfrac{\phi_{x}^{4}+\phi_{y}^{4}}{|\nabla \phi|^{4}} -7\big), &\mbox{for 2D},\ m=6,\ \theta_0=90^{\circ};\\[10pt] 
1 - 3\epsilon_4 + 4\epsilon_4\dfrac{\phi_{x}^{4}+\phi_{y}^{4}+\phi_{z}^{4}}{|\nabla \phi|^{4}}, &\mbox{for 3D},\ m=4,\ \theta_0=0^{\circ}. 
\end{cases} 
\end{aligned}
\end{equation*}

Applying the variational formulation to the above free energy \eqref{aenergy} yields the following governing equations for dendritic crystal growth:
\begin{subequations}\label{aeq}
\begin{align}
\nonumber
\tau \phi_{t} &=-\dfrac{\delta E}{\delta \phi}-\dfrac{\lambda}{\varepsilon} h'(\phi) T \\
&=\nabla \cdot\left(\kappa^{2}(\nabla \phi) \nabla \phi+\kappa(\nabla \phi)|\nabla \phi|^{2} \zeta(\phi)\right)-\dfrac{1}{\varepsilon^{2}} F'(\phi)-\dfrac{\lambda}{\varepsilon} h'(\phi) T, \label{aeq_1}\\
T_{t} &=D \Delta T+K h'(\phi) \phi_{t}. \label{aeq_2}
\end{align}
\end{subequations}
The first equation governs the evolution of the phase field, while the second describes heat diffusion during the phase transition. The quantities appearing in the above system are specified as follows. The mobility parameter $\tau>0$ is taken to be a constant, and $D>0$ denotes the thermal diffusivity. The term $\delta E/\delta\phi$ represents the variational derivative of the total free energy with respect to $\phi$. For the latent heat coupling, we use $h(\phi)=\frac{1}{5}\phi^{5}-\frac{2}{3}\phi^{3}+\phi$. Since $h'(\pm1)=0$, the latent heat contribution is concentrated near the liquid-solid interface. The quantity $\zeta(\phi)$ denotes the derivative of $\kappa(\nabla\phi)$ with respect to $\nabla\phi$. In two dimensions, it takes the following form \cite{kobayashi1993modeling}:
\begin{equation*}
\zeta(\phi)=\dfrac{\delta \kappa(\nabla \phi)}{\delta (\nabla \phi)}=
-m \epsilon_4 \dfrac{\sin \left(m\left(\theta-\theta_0\right)\right)}{|\nabla \phi|^2}\left(-\phi_y,\phi_x\right).
\end{equation*}
For the above choices of $\kappa(\nabla\phi)$, the corresponding expressions of $\zeta(\phi)$ are given by:
\begin{equation*}
\begin{aligned} \zeta(\phi)=\begin{cases} \dfrac{16\epsilon_4}{|\nabla\phi|^6}\left(\phi_x^3\phi_y^2-\phi_x\phi_y^4,\ \phi_x^2\phi_y^3-\phi_x^4\phi_y\right), &\hspace{-30mm}\mbox{for 2D},\ m=4,\ \theta_0=0^{\circ};\\[10pt] \dfrac{12\epsilon_4}{|\nabla\phi|^8}\left(3\phi_x^2\phi_y-\phi_y^3\right)\left(\phi_x^3-3\phi_x\phi_y^2\right)\left(\phi_y,-\phi_x\right), &\hspace{-30mm}\mbox{for 2D},\ m=6,\ \theta_0=0^{\circ};\\[10pt] \dfrac{12\epsilon_4}{|\nabla\phi|^8}\left(3\phi_x^2\phi_y-\phi_y^3\right)\left(\phi_x^3-3\phi_x\phi_y^2\right)\left(-\phi_y,\phi_x\right), &\hspace{-30mm}\mbox{for 2D},\ m=6,\ \theta_0=90^{\circ};\\[10pt] \dfrac{16\epsilon_{4}}{|\nabla \phi|^{6}}\left(\phi_x^3\phi_y^2 + \phi_x^3\phi_z^2 - \phi_x\phi_y^4 - \phi_x\phi_z^4,\ \phi_x^2\phi_y^3 + \phi_y^3\phi_z^2 - \phi_x^4\phi_y - \phi_y\phi_z^4\right.,\\[10pt] \hspace{16mm} \left.\phi_x^2\phi_z^3 + \phi_y^2\phi_z^3 - \phi_x^4\phi_z - \phi_y^4\phi_z\right), &\hspace{-30mm}\mbox{for 3D},\ m=4,\ \theta_0=0^{\circ}. \end{cases} \end{aligned}
\end{equation*}

For simplicity, periodic boundary conditions are imposed throughout this work. Taking the $L^2$-inner product of \eqref{aeq_1} and \eqref{aeq_2} with $\phi_t$ and $\frac{\lambda}{\varepsilon K}T$, respectively, and using integration by parts, we obtain the following energy dissipation law:
\bq\label{a_energy1}
\dfrac{d}{d t} E(\phi, T)=-\tau\left\| \phi_{t}\right\|^{2}-\frac{\lambda D}{\varepsilon K}\|\nabla T\|^{2} \leq 0,\nonumber
\eq
where $\|\cdot\|$ denotes the standard $L^2$-norm. Hence, the total free energy decays in time during the dendritic growth process.

\section{Decoupled energy-stable Runge-Kutta schemes}\label{sec3}
\setcounter{equation}{0}
In this section, we first introduce an equivalent auxiliary-variable system for \eqref{aeq} and present a first-order IMEX scheme. We then extend the time discretization to arbitrary order using algebraically stable Runge-Kutta methods and establish the discrete energy dissipation law. To reduce the computational cost, a matrix diagonalization technique is applied to transform the coupled elliptic system at each time step into independent elliptic equations with constant coefficients. For a fixed Fourier-Galerkin space, we further establish $q$th-order convergence in time. Finally, we use a Fourier spectral method for the spatial discretization of the proposed schemes.

\subsection{Auxiliary variable reformulation}\label{sec31}
We introduce
\begin{equation}\label{eq:E1}
E_1(\phi)=\int_{\Omega}\left(\frac{1}{2}\left(\kappa^2(\nabla \phi)-C_1\right)|\nabla \phi|^2+\frac{1}{\varepsilon^2}\left(F(\phi)-\frac{C_2}{2} \phi^2\right)+C_0\right)d \x,
\end{equation}
where $C_0,\ C_1$ and $C_2$ are positive constants with $C_1<(1-\epsilon_4)^2$. Among these constants, $C_0$ is chosen sufficiently large to ensure that $E_1(\phi)$ positive, 
while $C_1$ and $C_2$ play a role in the subsequent stability analysis, particularly in establishing the $H^1$-stability of the phase-field approximation.
Since $\kappa^2(\nabla\phi)\ge(1-\epsilon_4)^2$, the first term in \eqref{eq:E1} is nonnegative. 
Moreover, $F(\phi)-\frac{C_2}{2}\phi^2$ is bounded from below because $F(\phi)$ is a quartic polynomial with a positive leading coefficient. 
Therefore, there exists $C_0>0$ such that $E_1(\phi)>0$. 

Then we introduce the time-dependent auxiliary variable
\begin{equation}\label{av}
U(t)=\sqrt{E_1(\phi)},
\end{equation}
which allows the original system \eqref{aeq} to be written in the following equivalent form:
\begin{subequations}\label{aeq_sav}
\begin{align}
&\phi_t=\dfrac{1}{\tau}\mu,\label{aeq_sav_1}\\
&\mu=C_1\Delta\phi-\dfrac{C_2}{\varepsilon^2}\phi-UG(\phi)-UH(\phi,T),\label{aeq_sav_2}\\
&U_t=\dfrac{1}{2}\int_{\Omega}G(\phi)\phi_t,d \x,\label{aeq_sav_3}\\[3pt]
&T_t=D\Delta T+ UP(\phi,\mu),\label{aeq_sav_4}
\end{align}
\end{subequations}
where
$$G(\phi)=\dfrac{g(\phi)}{\sqrt{E_1(\phi)}},\ H(\phi,T)=\dfrac{\lambda h'(\phi) T}{\varepsilon\sqrt{E_1(\phi)}},\ P(\phi,\mu)=\dfrac{Kh'(\phi)\mu}{\tau\sqrt{E_1(\phi)}},$$
$$g(\phi)=-\nabla \cdot\left(\left(\kappa^2(\nabla \phi)-C_1\right) \nabla \phi+\kappa(\nabla \phi)|\nabla \phi|^2 \zeta(\phi)\right)+\frac{1}{\varepsilon^2}\left(F'(\phi)-C_2 \phi\right).$$
The associated initial conditions read:
\begin{equation}\label{aeq_sav_ini}
	\phi|_{t=0}=\phi_0,\ T|_{t=0}=T_0,\ U|_{t=0}=\sqrt{E_1(\phi_0)}.
\end{equation}

Equation \eqref{aeq_sav_3} is derived by differentiating \eqref{av}. Since $g(\phi)=\delta_\phi E_1(\phi)$, with $\delta_\phi$ denoting the variational derivative with respect to $\phi$, we have
$$U_t=\frac{1}{2\sqrt{E_1(\phi)}}\frac{d}{dt}E_1(\phi)=\frac{1}{2\sqrt{E_1(\phi)}}\int_{\Omega}g(\phi)\phi_t\,dx=\frac12\int_{\Omega}G(\phi)\phi_t\,dx.$$
Combined with the initial datum for $U$ given in \eqref{aeq_sav_ini}, \eqref{aeq_sav_3} guarantees consistency between $U$ and its definition \eqref{av}.
Consequently, the reformulated system \eqref{aeq_sav} is strictly equivalent to the original system \eqref{aeq} at the continuous level.

The energy associated to the reformulated system \eqref{aeq_sav} is defined by
\begin{equation*}\label{asav_E} 
\widetilde E(\phi,T,U)=\displaystyle{\int_{\Omega}\Big(\frac{C_1}{2}|\nabla \phi|^2+\frac{C_2}{2 \varepsilon^2} \phi^2+\frac{\lambda}{2 \varepsilon K} T^2-C_0\Big)d\x}+U^2.
\end{equation*}
Since $U^2=E_1(\phi)$ at the continuous level, $\widetilde E(\phi,T,U)$ coincides with the original free energy $E(\phi,T)$.  Taking the $L^2$-inner products of \eqref{aeq_sav_1}, \eqref{aeq_sav_2}, and \eqref{aeq_sav_4} with $\tau\phi_t$, $\phi_t$, and $\frac{\lambda}{\varepsilon K}T$, respectively, multiplying \eqref{aeq_sav_3} by $2U$, and then summing the resulting identities and applying integration by parts, we obtain the following energy dissipation law:
\begin{equation*}\label{asav_energy}
\dfrac{d}{d t} \widetilde E(\phi,T,U)=-\tau\left\| \phi_{t}\right\|^{2}-\frac{\lambda D}{\varepsilon K}\|\nabla T\|^{2} \leq 0.
\end{equation*}
Thus, the reformulated system preserves the energy dissipation structure of the original model. The schemes developed below are designed to preserve this property at the discrete level.

\subsection{Temporal discretization}\label{sec32}
In this subsection, we construct temporal discretizations for the reformulated system \eqref{aeq_sav}-\eqref{aeq_sav_ini}. We first present a first-order IMEX scheme and then develop arbitrary-order time-stepping schemes based on algebraically stable Runge-Kutta methods.

\subsubsection{A first-order IMEX scheme}\label{sec321}
Let $t_n$ $(n=0,\ldots,N_t)$ be a uniform partition of $[0,t_{\mathrm{end}}]$ with step size $\dt =t_{n+1}-t_n=t_{\mathrm{end}}/N_t$, where $N_t$ is a positive integer. A first-order IMEX discretization of \eqref{aeq_sav} is given by
\begin{equation}\label{ark_1st}
\left\{\begin{array}{l} \dps \phi_{n+1}=\phi_{n}+\dfrac{\dt}{\tau}\mu_{n+1}, \\[4pt] \dps \mu_{n+1}=C_1\Delta\phi_{n+1}-\dfrac{C_2}{\varepsilon^2}\phi_{n+1}-U_{n+1}G(\phi_n)-U_{n+1}H(\phi_n,T_n), \\[4pt] \dps U_{n+1}=U_{n}+\dfrac{\dt}{2\tau}\big(G(\phi_{n}), \mu_{n+1}\big)+\dfrac{\dt}{2\tau}\big(H(\phi_{n}, T_{n}),\mu_{n+1}\big)-\dfrac{\dt\lambda}{2\varepsilon K}\big(P(\phi_{n}, \mu_{n}),T_{n+1}\big), \\[12pt] \dps T_{n+1}=T_{n}+\dt D\Delta T_{n+1}+\dt U_{n+1}P(\phi_n,\mu_{n}). \end{array}\right.
\end{equation}

In the above scheme, $G$, $H$, and $P$ are evaluated explicitly, while the remaining terms are treated implicitly. The last two terms in the third equation of \eqref{ark_1st} are introduced to obtain the discrete energy dissipation law. At the continuous level, these two terms cancel each other, while their discrete approximations differ by $O(\dt)$:
\begin{equation*}
\begin{aligned} 
&\hspace{8mm}\dfrac{1}{2\tau}\big(H(\phi_{n}, T_{n}), \mu_{n+1}\big)-\dfrac{\lambda}{2\varepsilon K}\big(P(\phi_{n}, \mu_{n}),T_{n+1}\big)\\ &=\dfrac{\lambda}{2\varepsilon\tau \sqrt{E_1(\phi_{n})}}\big((h'(\phi_{n})T_{n},\mu_{n+1})-(h'(\phi_{n})\mu_{n},T_{n+1}) \big)\\ &=\dfrac{\lambda}{2\varepsilon\tau \sqrt{E_1(\phi_{n})}}\big((h'(\phi_{n})T_{n},\mu_{n+1})-(h'(\phi_{n})T_{n},\mu_{n})+(h'(\phi_{n})\mu_{n},T_{n})-(h'(\phi_{n})\mu_{n},T_{n+1}) \big)\\[5pt] 
&=O(\dt). 
\end{aligned}
\end{equation*}

The following theorem shows that the scheme preserves a discrete energy dissipation law.

\begin{theorem}\label{th31}
The first-order IMEX scheme \eqref{ark_1st} satisfies the following discrete energy dissipation law:
\begin{equation*}
\mathcal{E}_{n+1}-\mathcal{E}_{n} \leq 0, \ n=0,1,\ldots,
\end{equation*}
where the discrete energy at $t_n$ is defined as
\begin{equation*}
\mathcal{E}_n=\frac{C_1}{2}\left\|\nabla \phi_n\right\|^2+\frac{C_2}{2 \varepsilon^2}\left\|\phi_n\right\|^2+\frac{\lambda}{2 \varepsilon K}\left\|T_n\right\|^2+\left|U_{n}\right|^2.
\end{equation*}
\end{theorem}

\begin{proof}
By taking the inner product of the first two equations in \eqref{ark_1st} with $\tau\frac{\phi_{n+1}-\phi_{n}}{\dt^2}$ and $\frac{\phi_{n+1}-\phi_{n}}{\dt}$ respectively, and then summing the resulting equations yields
\begin{equation}\label{th31_e1}
\begin{aligned} &\hspace{8mm}\tau\|\dfrac{\phi_{n+1}-\phi_{n}}{\dt}\|^2+\dfrac{U_{n+1}}{\dt}\big(G(\phi_n)+H(\phi_n,T_n),\phi_{n+1}-\phi_n\big)\\[2pt] &+\dfrac{C_1}{2\dt}(\|\nabla\phi_{n+1}\|^2-\|\nabla\phi_{n}\|^2+\|\nabla\phi_{n+1}-\nabla\phi_{n}\|^2)\\[4pt] &+\dfrac{C_2}{2\varepsilon^2\dt}(\|\phi_{n+1}\|^2-\|\phi_{n}\|^2+\|\phi_{n+1}-\phi_{n}\|^2)=0, \end{aligned}
\end{equation}
where we have used the following identity:
\begin{equation}\label{q1}
2(a_{k+1},a_{k+1}-a_{k})=|a_{k+1}|^{2}-|a_{k}|^{2}+|a_{k+1}-a_{k}|^{2}.
\end{equation}
Eliminating $\mu_{n+1}$ from the third equation by means of the first in \eqref{ark_1st}, and subsequently multiplying the resulting expression by $\frac{2U_{n+1}}{\dt}$, we obtain
\begin{equation}\label{th31_e2}
\dfrac{2U_{n+1}(U_{n+1}-U_{n})}{\dt}=\dfrac{U_{n+1}}{\dt}\big(G(\phi_n)+H(\phi_n,T_n),\phi_{n+1}-\phi_n\big)-\dfrac{\lambda U_{n+1}}{\varepsilon K}\big(P(\phi_{n}, \mu_{n}),T_{n+1}\big).
\end{equation}
Furthermore, taking the inner product of the fourth equation in \eqref{ark_1st} with $\frac{\lambda}{\varepsilon K \dt}T_{n+1}$ gives
\begin{equation}\label{th31_e3}
\big(\dfrac{T_{n+1}-T_{n}}{\dt},\dfrac{\lambda}{\varepsilon K}T_{n+1}\big)=-\dfrac{\lambda D}{\varepsilon K}\|\nabla T_{n+1}\|^{2}+\dfrac{\lambda U_{n+1}}{\varepsilon K}\big(P(\phi_{n}, \mu_{n}),T_{n+1}\big).
\end{equation}
Finally, combining \eqref{th31_e1}, \eqref{th31_e2} and \eqref{th31_e3}, and applying \eqref{q1} once again, we obtain the desired result by dropping some nonnegative terms. This completes the proof.

\end{proof}

\begin{remark}\label{re31}
An advantage of the above first-order IMEX scheme is that the coupled system can be reduced to several independent linear elliptic problems with constant coefficients and a scalar equation for $U_{n+1}$. More precisely, the first two equations together with the fourth equation in \eqref{ark_1st} allow $\phi_{n+1}$ and $T_{n+1}$ to be expressed in terms of $U_{n+1}$, while $\mu_{n+1}$ is then obtained from the first equation. Substituting these relations into the third equation yields a scalar equation for $U_{n+1}$. Once $U_{n+1}$ is determined, $\phi_{n+1}$, $\mu_{n+1}$, and $T_{n+1}$ can be recovered directly.
\end{remark}

\subsubsection{Arbitrary-order Runge-Kutta schemes}\label{sec322}
We now extend the first-order IMEX scheme to arbitrary order using algebraically stable Runge-Kutta methods. For completeness, we briefly recall the general Runge-Kutta framework  and the assumptions used below. Consider the autonomous ODE:
\begin{equation}\label{ODE}
\left\{\begin{array}{l} 
\boldsymbol{u}_t=\boldsymbol{F}(\boldsymbol{u}), \ \boldsymbol{F}: \mathbb{R}^N \to \mathbb{R}^N,\\ 
\boldsymbol{u}|_{t=0} = \boldsymbol{u}_0.
\end{array}\right.
\end{equation}
A $q$-stage Runge-Kutta method applied to \eqref{ODE} is given by
\begin{equation*}
\left\{\begin{array}{l} 
\dps \boldsymbol{u}_{ni}=\boldsymbol{u}_{n}+\dt \sum_{j=1}^{q}  a_{i j}  \boldsymbol{F}(\boldsymbol{u}_{nj}), \quad i=1, \ldots, q, \\[10pt] 
\dps \boldsymbol{u}_{n+1} = \boldsymbol{u}_n+\dt \sum_{i=1}^{q}  b_{i} \boldsymbol{F}(\boldsymbol{u}_{ni}), 
\end{array}\right.
\end{equation*}
where $\boldsymbol{u}_{ni}$ denotes the internal stage value at $t_{ni}=t_{n}+ c_{i} \dt, \ i=1, \ldots, q,\ n=0,\ldots,N_t-1$. The method is represented by the standard Butcher tableau:
\begin{equation*}\label{tableau}
\begin{array}{c|c}
	A & \boldsymbol c \\ \hline
	\vspace{-0.25cm} \\
	\boldsymbol b^T &
\end{array}=
\begin{array}{ccc|c}
	 a_{11} & \ldots &  a_{1 q} &  c_1 \\
	\vdots & & \vdots & \vdots \\
	 a_{q 1} & \ldots &  a_{q q} &  c_q \\
	\hline  b_1 & \ldots &  b_q &
\end{array}.
\end{equation*}
Here $A = ( a_{ij}) \in \mathbb{R}^{q \times q}$. The vectors $\boldsymbol b^T = ( b_1,  b_2, \ldots,  b_q)$ and $\boldsymbol c^T = ( c_1,  c_2, \ldots,  c_q)$ denote the usual Runge-Kutta coefficients, where $ c_i = \sum_{j=1}^q  a_{ij}, \ i=1, \ldots, q.$

For the schemes developed below, we assume that $A$ is invertible and that the nodes $ c_i$, $i=1,\ldots,q$, are distinct. In addition, the Runge-Kutta method is assumed to be algebraically stable, namely, $ b_i>0,\ i=1,\ldots,q$, and the symmetric matrix $\Gamma=(\gamma_{ij})\in\mathbb{R}^{q\times q}$ defined by
\begin{equation*}\label{rk_gamma}
\gamma_{ij}= b_i a_{ij}+ b_j a_{ji}- b_i b_j,\ i,j=1,\ldots,q,
\end{equation*}
is positive semidefinite. We further assume that the Runge-Kutta method has order $p\geq q$ and stage order at least $q$, and that its coefficients satisfy
\begin{equation*}\label{rk_Bp}
\sum_{i=1}^q b_i c_i^{\ell-1}=\frac{1}{\ell},\ \ell=1,\ldots,p,
\end{equation*}
and
\begin{equation*}\label{rk_Cq}
\sum_{j=1}^q a_{ij} c_j^{\ell-1}=\frac{ c_i^\ell}{\ell},\ \ell=1,\ldots,q,\ i=1,\ldots,q.
\end{equation*}
The algebraic stability conditions are used to establish the discrete energy dissipation law, while the order conditions above are required in the convergence analysis \cite{butcher1964implicit,ehle1969pade,tang2022arbitrarily}. In Section \ref{sec4}, we use the $q$-stage Gauss Runge-Kutta methods with $q=2,3,4,5$. They satisfy all the above assumptions, and the corresponding Butcher tableaux are provided in the Appendix.

To treat the nonlinear terms explicitly at the Runge-Kutta stages, we introduce an interpolation operator based on the data from the previous time step. Let $\mathbb{P}^{q-1}$ denote the space of polynomials of degree at most $q-1$. For $n\geq1$ and a function $v$, we define $\Pi_{n-1}v\in\mathbb{P}^{q-1}$ by
$$
\big(\Pi_{n-1} v\big)(t_{n-1, i})=v(t_{n-1, i}), \ i=1, \ldots, q.
$$
Since the nodes $ c_i$ are distinct, the interpolation polynomial $\Pi_{n-1}v$ is uniquely determined. For the discrete stage values $v_{n-1,i}$, $i=1,\ldots,q$, we similarly let $\Pi_{n-1}v_{n-1}\in\mathbb{P}^{q-1}$ be the polynomial determined by
$$\big(\Pi_{n-1} v_{n-1}\big)(t_{n-1, i})=v_{n-1, i}, \ i=1, \ldots, q.$$
For simplicity, we use the abbreviation
$$\Pi_{n-1}v_{ni}:=(\Pi_{n-1}v_{n-1})(t_{ni}),\ i=1,\ldots,q.$$

Using the interpolation operator defined above, we construct the following $q$-stage Runge-Kutta scheme for the reformulated system \eqref{aeq_sav}. Given $\phi_n$, $T_n$, $U_n$, and the stage values $\phi_{n-1,i}$, $\mu_{n-1,i}$, and $T_{n-1,i}$ from the previous time step, $i=1,\ldots,q$, the internal stages are determined by
\begin{equation}\label{ark_phi}
\hspace{9mm}\begin{cases} 
\mu_{n i}=C_1\Delta \phi_{n i}-\dfrac{C_2}{\varepsilon^2}\phi_{n i}-U_{n i} G(\Pi_{n-1}  \phi_{n i})-U_{n i} H(\Pi_{n-1} \phi_{n i}, \Pi_{n-1} T_{n i}),\\[2pt] \dps\phi_{n i}=\phi_{n}+\dfrac{\dt}{\tau } \sum_{j=1}^{q}  a_{i j} \mu_{n j}, \end{cases}
\end{equation}

\begin{equation}\label{ark_r}
\hspace{-14mm}\begin{cases} 
\dot{U}_{n i}=\dfrac{1}{2\tau}\big(G(\Pi_{n-1}\phi_{n i}), \mu_{n i}\big)+\dfrac{1}{2\tau}\big(H(\Pi_{n-1} \phi_{n i}, \Pi_{n-1} T_{n i}),\mu_{n i}\big)\\[10pt] 
\hspace{9.5mm}-\dfrac{\lambda}{2\varepsilon K}\big(P(\Pi_{n-1} \phi_{n i}, \Pi_{n-1} \mu_{n i}),T_{n i}\big), \\[2pt] 
\dps U_{n i}=U_{n}+\dt \sum_{j=1}^{q}  a_{i j} \dot{U}_{n j}, 
\end{cases}
\end{equation}

\begin{equation}\label{ark_u}
\hspace{-47.5mm}\begin{cases} 
\dot{T}_{n i}=D\Delta T_{n i}+U_{n i} P(\Pi_{n-1} \phi_{n i}, \Pi_{n-1} \mu_{n i}),\\[2pt] 
\dps T_{n i}=T_{n}+\dt \sum_{j=1}^{q}  a_{i j} \dot{T}_{n j}. 
\end{cases}
\end{equation}
The numerical solution at $t_{n+1}$ is then updated by
\begin{equation}\label{ark_n+1}
\hspace{-72.5mm}\left\{\begin{array}{l} 
\dps \phi_{n+1}=\phi_{n}+\dfrac{\dt}{\tau} \sum_{i=1}^{q}  b_{i} \mu_{n i}, \\ 
\dps U_{n+1}=U_{n}+\dt \sum_{i=1}^{q}  b_{i} \dot{U}_{n i}, \\ 
\dps T_{n+1}=T_{n}+\dt \sum_{i=1}^{q}  b_{i} \dot{T}_{n i}. 
\end{array}\right.
\end{equation}

The scheme \eqref{ark_phi}-\eqref{ark_n+1} coincides with that studied in \cite{wang2025class}. Subject to the assumptions specified above, 
the scheme admits a unique solution and obeys the discrete energy dissipation law
\begin{equation*}
	\mathcal{E}_{n+1}-\mathcal{E}_{n}\leq 0, \qquad n=1,\ldots,N_t-1,
\end{equation*}
where
\begin{equation*}
	\mathcal{E}_n =\frac{C_1}{2}\left\|\nabla\phi_n\right\|^2 +\frac{C_2}{2\varepsilon^2}\left\|\phi_n\right\|^2 +\frac{\lambda}{2\varepsilon K}\left\|T_n\right\|^2 +\left|U_n\right|^2.
\end{equation*}
It follows that the scheme is unconditionally energy stable.
	
It should be noted that the $q$-stage scheme requires the stage values from the previous time step and therefore needs to be initialized. The required values $\phi_{0i}$, $\mu_{0i}$, and $T_{0i}$, $i=1,\ldots,q$, together with $\phi_1$, $T_1$, and $U_1$, can be computed by a suitable numerical scheme with a sufficiently small time step, such as the first-order IMEX scheme \eqref{ark_1st}. 

After this initialization, the $q$-stage scheme \eqref{ark_phi}-\eqref{ark_n+1} is used for the subsequent time steps with $n\geq1$. Its direct implementation requires solving a coupled multistage elliptic system at each time step. To reduce the computational cost, we employ a matrix diagonalization technique, which transforms the coupled system into independent elliptic problems with constant coefficients and a $q\times q$ algebraic system for the auxiliary variables.
	
For this purpose, we introduce the following vector notation for the internal stage variables. Let $\Phi_{n}=\left(\phi_{n 1}, \ldots, \phi_{n q}\right)^T$, $\mathcal{M}_n =(\mu_{n1},\dots,\mu_{nq})^T$, $\mathcal{U}_{n}=\left(U_{n 1}, \ldots, U_{n q}\right)^T$, $\dot{\mathcal{U}}_{n}=\left(\dot{U}_{n 1}, \ldots, \dot{U}_{n q}\right)^T$, $\mathcal{T}_{n}=\left(T_{n 1}, \ldots, T_{n q}\right)^T$, $\dot{\mathcal{T}}_{n}=\left(\dot{T}_{n 1}, \ldots, \dot{T}_{n q}\right)^T$, $\mathbb{1}=(1, \ldots, 1)^T$, and let $I$ denote the $q \times q$ identity matrix. Thus, \eqref{ark_phi} can be rewritten as follows
\begin{equation}\label{de_phi1}
\begin{cases} 
\mathcal{M}_n=C_1\Delta \Phi_{n}-\dfrac{C_2}{\varepsilon^2}\Phi_{n}-B_1\mathcal{U}_{n},\\ \Phi_{n}=\phi_n\mathbb{1}+\dfrac{\dt}{\tau }A\mathcal{M}_n, 
\end{cases}
\end{equation}
where
\begin{equation*}
B_{1}=\operatorname{diag}\left(b_{n 1}, \ldots, b_{n q}\right) \  \mbox { with }\ b_{n i}=G(\Pi_{n-1} \phi_{n i})+H(\Pi_{n-1} \phi_{n i}, \Pi_{n-1} T_{n i}).
\end{equation*}
Substituting the first equation of \eqref{de_phi1} into the second equation yields
\begin{equation}\label{de_phi2}
\left(A^{-1}-\dfrac{\dt C_1}{\tau }\Delta I+\dfrac {\dt C_2}{\tau\varepsilon^2}I\right)\Phi_{n}=\phi_nA^{-1}\mathbb{1}-\dfrac{\dt}{\tau }B_1\mathcal{U}_{n}.
\end{equation}
For the Runge-Kutta methods considered in this work, $A^{-1}$ is diagonalizable. Let $\Lambda=\operatorname{diag}(\lambda_1,\ldots,\lambda_q)$ be the diagonal matrix of the eigenvalues of $A^{-1}$, and let $V$ be the matrix formed by the corresponding eigenvectors. Then we have
\begin{equation}\label{EG}
	A^{-1} V= V\Lambda.
\end{equation}
Using the decomposition \eqref{EG}, \eqref{de_phi2} can be rewritten as
\begin{equation*}
\left(\Lambda -\frac{\dt C_1}{\tau }\Delta I+\frac {\dt C_2}{\tau\varepsilon^2} I\right) V^{-1}\Phi_{n}= V^{-1}\phi_nA^{-1}\mathbb{1}-\dfrac{\dt}{\tau } V^{-1}B_1\mathcal{U}_{n}. 
\end{equation*}
The diagonal structure of $\Lambda$ decouples the transformed system into $q$ independent elliptic equations with constant coefficients. Therefore,
\begin{equation}\label{de_phi}
\Phi_{n}=g_n+B_n\mathcal{U}_{n},
\end{equation}
where 
\begin{subequations}\label{g_B}
\begin{align}
	g_n&= V\left(\Lambda -\frac{\dt C_1}{\tau }\Delta I+\frac {\dt C_2}{\tau\varepsilon^2}I\right)^{-1} V^{-1}\phi_nA^{-1}\mathbb{1}, \label{g_B_1}\\
	B_n&=- V\left(\Lambda -\frac{\dt C_1}{\tau }\Delta I+\frac {\dt C_2}{\tau\varepsilon^2}I\right)^{-1}\frac{\dt}{\tau } V^{-1}B_1. \label{g_B_2}
\end{align}
\end{subequations}
Since $A^{-1}$ is real, its nonreal eigenvalues appear in complex-conjugate pairs. Hence, the corresponding elliptic equations can be solved pairwise, and the real stage variables are then recovered from the transformed solutions.

Similarly, the temperature stage equations in \eqref{ark_u} can be written as
\begin{equation}\label{de_u1}
(A^{-1}-\dt D\Delta I)\mathcal{T}_{n}=T_{n}A^{-1}\mathbb{1}+\dt P_1\mathcal{U}_{n},
\end{equation}
where
\begin{equation*}
P_{1}=\operatorname{diag}\left(p_{n 1}, \ldots, p_{n q}\right) \  \mbox { with }\ p_{n i}=P(\Pi_{n-1} \phi_{n i}, \Pi_{n-1} \mu_{n i}).
\end{equation*}
Using the decomposition \eqref{EG}, equation \eqref{de_u1} gives
\begin{equation}\label{de_u}
\mathcal{T}_{n}=h_n+P_n\mathcal{U}_{n},
\end{equation}
where 
\begin{subequations}\label{l_W}
	\begin{align}
		h_n&= V(\Lambda -\dt D\Delta I)^{-1} V^{-1}T_{n}A^{-1}\mathbb{1}, \label{l_W_1}\\
		P_n&= V(\Lambda -\dt D\Delta I)^{-1}\dt  V^{-1}P_1. \label{l_W_2}
	\end{align}
\end{subequations}

Substituting \eqref{de_phi} and \eqref{de_u} into the second equations of \eqref{de_phi1} and \eqref{ark_u} respectively, we obtain
\begin{equation}\label{de_sigma}
\mathcal{M}_n=\tau\dt^{-1}A^{-1}(g_n-\phi_n\mathbb{1})+\tau\dt^{-1}A^{-1}B_n\mathcal{U}_{n},
\end{equation}
\begin{equation}\label{de_dotu}
\dot{\mathcal{T}}_{n}=\dt^{-1}A^{-1}(h_n-T_{n}\mathbb{1})+\dt^{-1}A^{-1}P_n\mathcal{U}_{n}.
\end{equation}
Then the first relation of \eqref{ark_r} yields
\begin{equation}\label{de_dotr}
\dot{\mathcal{U}}_{n}=\dt^{-1}f_n+\dt^{-1}F_n\mathcal{U}_{n}+s_n+S_n\mathcal{U}_{n},
\end{equation}
where 
$$f_n=\dfrac{1}{2\tau}\Big(\big(b_{n1},\left(\tau A^{-1}(g_n-\phi_n\mathbb{1})\right)_1\big),\ldots,\big(b_{nq},\left(\tau A^{-1}(g_n-\phi_n\mathbb{1})\right)_q\big)\Big)^T,$$
$$s_n=-\dfrac{\lambda}{2\varepsilon K}\Big(\left(p_{n1},(h_n)_1\right),\ldots,\left(p_{nq},(h_n)_q\right)\Big)^T,$$
and
$$F_n=\dfrac{1}{2\tau}(F_{ij}),\ F_{ij}=\big(b_{ni},(\tau A^{-1}B_n)_{ij}\big),$$
$$S_n=-\dfrac{\lambda}{2\varepsilon K}(S_{ij}),\ S_{ij}=\big(p_{ni},(P_n)_{ij}\big).$$
The second equation of \eqref{ark_r} gives
\begin{equation*}\label{de_r1}
\mathcal{U}_{n}=U_{n}\mathbb{1}+\dt A\dot{\mathcal{U}}_{n}.
\end{equation*}
Substituting \eqref{de_dotr} into this relation yields
\begin{equation*}\label{de_r2}
\mathcal{U}_{n}=U_{n}\mathbb{1}+A f_n+A F_n\mathcal{U}_{n}+\dt A s_n+\dt A S_n\mathcal{U}_{n}.
\end{equation*}
Define
$$L_n:=I-A F_n-\dt A S_n.$$
Since the equation for $\mathcal{U}_{n}$ is obtained by algebraic elimination, the unique solvability implies that $L_n$ is nonsingular. Therefore,
\begin{equation}\label{de_r3}
\mathcal{U}_{n}=L_n^{-1}(U_{n}\mathbb{1}+A f_n+\dt A s_n).
\end{equation}

With $\mathcal{U}_{n}$ computed by this expression, $\Phi_n, \mathcal{T}_{n}, \mathcal{M}_n, \dot{\mathcal{T}}_{n}, \dot{\mathcal{U}}_{n}$ can be obtained directly, and then $\phi_{n+1}, T_{n+1}$ and $U_{n+1}$ can be updated by \eqref{ark_n+1}.

The resulting decoupled implementation is summarized as follows:

\begin{quote}
\begin{enumerate}
\renewcommand{\labelenumi}{\textbf{Step \arabic{enumi}.}}
	\item Compute $g_n, B_n, h_n$ and $P_n$ from \eqref{g_B} and \eqref{l_W}. The corresponding elliptic problems are independent and can be solved in parallel.
	\item Compute $f_n, s_n, F_n$ and $S_n$ from \eqref{de_dotr}.
	\item Determine $\mathcal{U}_n$ by solving the $q\times q$ algebraic system \eqref{de_r3}.
	\item Obtain $\Phi_n, \mathcal{T}_{n}, \mathcal{M}_n, \dot{\mathcal{T}}_{n}$ and $\dot{\mathcal{U}}_{n}$ from \eqref{de_phi}, \eqref{de_u}, \eqref{de_sigma}, \eqref{de_dotu} and \eqref{de_dotr}, respectively.
	\item Update $\phi_{n+1}, T_{n+1}$ and $U_{n+1}$ by \eqref{ark_n+1}.
	\end{enumerate}
\end{quote}

The matrix diagonalization employed above is an algebraically equivalent transformation of the Runge-Kutta scheme and introduces no additional approximation. Hence, the resulting decoupled implementation retains the unique solvability and unconditional discrete energy dissipation law of the original coupled scheme.

\begin{remark}\label{re32}
	The extrapolation used in the scheme does not reduce the temporal order of accuracy. Indeed, the extrapolation based on the $q$ stage values has an error of order $O(\dt^q)$. Since the extrapolation error is multiplied by $\dt$ in the Runge-Kutta equations, the resulting local residuals are of order $O(\dt^{q+1})$. The $q$th-order temporal convergence for a fixed Fourier-Galerkin space will be proved in Subsection~\ref{sec33}.
\end{remark}

\subsection{Temporal convergence analysis}\label{sec33}
In this subsection, we establish a rigorous temporal error estimate for the Runge-Kutta scheme proposed in Subsection~\ref{sec32}. The main difficulty arises from the nonlinear coupling and the gradient-dependent anisotropic terms. Their error estimates require control of higher-order spatial norms beyond those provided directly by the energy estimate. A direct analysis in the continuous spatial setting would therefore require additional spatial regularity assumptions. To focus on the temporal discretization error, we carry out the analysis in a fixed Fourier-Galerkin space. This setting is consistent with the Fourier spectral discretization used in the numerical computations and separates the temporal discretization error from the spatial approximation error. In the fixed finite-dimensional space, the spatial norms arising in the nonlinear error estimates can also be controlled by the norm-equivalence estimates used below.

\subsubsection{Preliminaries}\label{331}
Let $X_M$ be the Fourier-Galerkin space and let
$\mathcal Q_M:L^2(\Omega)\rightarrow X_M$ denote the $L^2$-orthogonal
projection. Using the properties of the Fourier basis functions, for any fixed $M$ and any fixed $s\geq1$, we have
\begin{equation}\label{eq:inverse_equivalence}
\|v_M\|_{H^s} \le C_M\|v_M\|, \ \|v_M\|_{L^\infty} \le C_M\|v_M\|_{H^1}, \ v_M\in X_M.
\end{equation}

Define
$$F_2(\boldsymbol p) = \frac12|\kappa(\boldsymbol p)\boldsymbol p|^2, \ \mathcal A(\boldsymbol p) = \nabla_{\boldsymbol p}F_2(\boldsymbol p).$$
By \cite[Lemma~3.1(ii)]{li2026efficient}, $F_2\in C^{1,1}(\mathbb R^d)$. Consequently, there exists $L_A>0$ such that
\begin{equation}\label{eq:A_global_lipschitz}
|\mathcal A(\boldsymbol p)-\mathcal A(\boldsymbol q)| \le L_A|\boldsymbol p-\boldsymbol q|, \ \boldsymbol p,\boldsymbol q\in\mathbb R^d.
\end{equation}

For $v_M\in X_M$, define $g_M(v_M)\in X_M$ variationally by
\begin{equation}\label{eq:gN_definition}
\begin{aligned} 
(g_M(v_M),\chi_M) ={}& \left( \mathcal A(\nabla v_M)-C_1\nabla v_M, \nabla\chi_M \right) \\
&+ \frac1{\varepsilon^2} \left( F'(v_M)-C_2v_M,\chi_M \right), \ \chi_M\in X_M. 
\end{aligned}
\end{equation}
Furthermore, define
\begin{equation*}\label{eq:GN_HN_PN_definition}
\begin{aligned} 
G_M(v_M) &= \frac{g_M(v_M)} {\sqrt{E_{1,M}(v_M)}}, \\ 
H_M(v_M,z_M) &= \mathcal Q_M \left[ \frac{\lambda h'(v_M)z_M} {\varepsilon\sqrt{E_{1,M}(v_M)}} \right], \\ 
P_M(v_M,m_M) &= \mathcal Q_M \left[ \frac{Kh'(v_M)m_M} {\tau\sqrt{E_{1,M}(v_M)}} \right], 
\end{aligned}
\end{equation*}
where $E_{1,M}$ denotes the restriction of $E_1$ to $X_M$. As established in Subsection~\ref{sec31}, the constant $C_0$ can be chosen sufficiently large such that $E_1$ is uniformly bounded away from zero. Therefore, its restriction to $X_M$ satisfies
\begin{equation}\label{eq:E1_positive}
E_{1,M}(v_M)\ge c_*>0.
\end{equation}

To control the nonlinear terms in the subsequent temporal error analysis, we first establish the following lemma.

\begin{lemma}\label{lem31}
Let $M$ be fixed and let $\mathcal B_M$ be any bounded subset of
$X_M\times X_M\times X_M$. Then there exists $C_M>0$ such that, for any $(v_M,z_M,m_M),\ (w_M,\zeta_M,\eta_M)\in\mathcal B_M$, the following inequalities hold:
\begin{equation}\label{eq:GN_lipschitz}
\|G_M(v_M)-G_M(w_M)\|_{H^1}\le C_M\|v_M-w_M\|_{H^1},
\end{equation}
\begin{equation}\label{eq:HN_lipschitz}
\|H_M(v_M,z_M)-H_M(w_M,\zeta_M)\|_{H^1}\le C_M\left( \|v_M-w_M\|_{H^1}+\|z_M-\zeta_M\|_{H^1}\right),
\end{equation}
\begin{equation}\label{eq:PN_lipschitz}
\|P_M(v_M,m_M)-P_M(w_M,\eta_M)\|_{H^1}\le C_M\left( \|v_M-w_M\|_{H^1}+\|m_M-\eta_M\|_{H^1}\right),
\end{equation}
and
\begin{equation}\label{eq:operator_uniform_bound}
\|G_M(v_M)\|_{H^1}+\|H_M(v_M,z_M)\|_{H^1}+\|P_M(v_M,m_M)\|_{H^1}\leq C_M.
\end{equation}
\end{lemma}

\begin{proof}
We begin by proving \eqref{eq:GN_lipschitz}. From \eqref{eq:gN_definition}, we have
\begin{equation*}
\begin{aligned}
&(g_M(v_M)-g_M(w_M),\chi_M)\notag\\
={}&\left(\mathcal A(\nabla v_M)-\mathcal A(\nabla w_M),\nabla\chi_M\right)-C_1\left(\nabla(v_M-w_M),\nabla\chi_M\right)\notag\\
&+\frac1{\varepsilon^2}\left(F'(v_M)-F'(w_M)-C_2(v_M-w_M),\chi_M\right).\label{eq:gN_difference_weak}
\end{aligned}
\end{equation*}
Applying \eqref{eq:A_global_lipschitz} pointwise and then using the
Cauchy-Schwarz inequality, we obtain
$$\left| \left( \mathcal A(\nabla v_M)-\mathcal A(\nabla w_M), \nabla\chi_M \right) \right| \le L_A \|\nabla(v_M-w_M)\| \|\nabla\chi_M\|.$$
Moreover, since $F'$ is a polynomial and hence locally Lipschitz continuous, the boundedness of $\mathcal B_M$ and \eqref{eq:inverse_equivalence} yield
$$\|F'(v_M)-F'(w_M)\| \le C_M\|v_M-w_M\|.$$
Consequently,
\begin{equation*}\label{eq:gN_weak_lipschitz}
\left| (g_M(v_M)-g_M(w_M),\chi_M) \right| \le C_M \|v_M-w_M\|_{H^1} \|\chi_M\|_{H^1}.
\end{equation*}
Using \eqref{eq:inverse_equivalence}, we have
$$\|\chi_M\|_{H^1} \le C_M\|\chi_M\|,$$
which yields
$$\|g_M(v_M)-g_M(w_M)\|=\sup_{\substack{\chi_M\in X_M\\ \|\chi_M\|=1}}\left|\left(g_M(v_M)-g_M(w_M),\chi_M\right)\right|\leq C_M\|v_M-w_M\|_{H^1}.$$
Applying \eqref{eq:inverse_equivalence} once more gives
\begin{equation}\label{eq:gN_strong_lipschitz}
\|g_M(v_M)-g_M(w_M)\|_{H^1} \le C_M\|v_M-w_M\|_{H^1}.
\end{equation}

Next, since $E_{1,M}$ is continuously differentiable with a bounded derivative
on bounded subsets of $X_M$, the mean-value theorem gives
\begin{equation*}\label{eq:E1_lipschitz}
|E_{1,M}(v_M)-E_{1,M}(w_M)| \le C_M\|v_M-w_M\|_{H^1}.
\end{equation*}
Moreover, combining the uniform lower bound in \eqref{eq:E1_positive} with the preceding estimate and applying the mean-value theorem, we obtain
\begin{equation}\label{eq:E1_inverse_lipschitz}
\left|\frac1{\sqrt{E_{1,M}(v_M)}}-\frac1{\sqrt{E_{1,M}(w_M)}}\right|\leq\frac{1}{2c_*^{3/2}}\left|E_{1,M}(v_M)-E_{1,M}(w_M)\right|\le C_M\|v_M-w_M\|_{H^1}.
\end{equation}
Thus, the following decomposition, together with \eqref{eq:gN_strong_lipschitz} and \eqref{eq:E1_inverse_lipschitz}, yields
\begin{equation*}
\begin{aligned}
&\|G_M(v_M)-G_M(w_M)\|_{H^1}\\
={}&\left\|\frac{g_M(v_M)-g_M(w_M)}{\sqrt{E_{1,M}(v_M)}}
+g_M(w_M)\left(\frac1{\sqrt{E_{1,M}(v_M)}}-\frac1{\sqrt{E_{1,M}(w_M)}}\right)\right\|_{H^1}\\
\leq{}&C_M\|v_M-w_M\|_{H^1},
\end{aligned}
\end{equation*}
which proves \eqref{eq:GN_lipschitz}.
The estimates \eqref{eq:HN_lipschitz} and \eqref{eq:PN_lipschitz} follow by a similar argument. Indeed, the Lipschitz continuity of $h'$ on the prescribed bounded set, the uniform lower bound \eqref{eq:E1_positive}, and the finite-dimensional product estimate
$$
\|\mathcal Q_M(v_Mz_M)\|_{H^1}\leq C_M\|v_M\|_{H^1}\|z_M\|_{H^1}
$$
give the desired estimates.

Finally, we prove \eqref{eq:operator_uniform_bound} by fixing a reference point
$(\bar v_M,\bar z_M,\bar m_M)\in\mathcal B_M$. Since $\mathcal B_M$ is bounded, there exists $R_M>0$ such that for all $(v_M,z_M,m_M)\in\mathcal B_M$,
$$
\|v_M-\bar v_M\|_{H^1}+\|z_M-\bar z_M\|_{H^1}+\|m_M-\bar m_M\|_{H^1}\leq R_M.
$$
Therefore, by the triangle inequality and \eqref{eq:GN_lipschitz}-\eqref{eq:PN_lipschitz},we have 
$$
\begin{aligned}
&\|G_M(v_M)\|_{H^1}+\|H_M(v_M,z_M)\|_{H^1}+\|P_M(v_M,m_M)\|_{H^1}\\
&\qquad\leq\|G_M(\bar v_M)\|_{H^1}+\|H_M(\bar v_M,\bar z_M)\|_{H^1}+\|P_M(\bar v_M,\bar m_M)\|_{H^1}+C_MR_M\leq C_M,
\end{aligned}
$$
which proves \eqref{eq:operator_uniform_bound} and completes the proof.
\end{proof}

We next present the consistency error analysis for the Runge-Kutta scheme in the fixed Fourier-Galerkin space $X_M$. Let $(\phi_M(t),T_M(t),U_M(t),\mu_M(t))$ denote the exact solution of the Fourier semi-discrete system. We then define the exact stage and nodal values by
$$
\begin{alignedat}{3}
&\phi_{M,ni}^*=\phi_M(t_{ni}),\ &&T_{M,ni}^*=T_M(t_{ni}),\ &&U_{M,ni}^*=U_M(t_{ni}),\\
&\phi_{M,n}^*=\phi_M(t_n),&&T_{M,n}^*=T_M(t_n),&&U_{M,n}^*=U_M(t_n).
\end{alignedat}
$$
For convenience, we use the notation $\widehat v_{M,ni}^*:=\Pi_{n-1}v_{M,ni}^*$ for $v=\phi,T,\mu$, where $\mu_{M,ni}^*$ is defined by the stage relation below. The same convention is adopted for the corresponding numerical stage values. We then set
\begin{equation*}\label{eq:B_P_exact_star}
B_{M,ni}^*=G_M(\widehat\phi_{M,ni}^*)+H_M(\widehat\phi_{M,ni}^*,\widehat T_{M,ni}^*),\  \mathcal P_{M,ni}^*=P_M(\widehat\phi_{M,ni}^*,\widehat\mu_{M,ni}^*).
\end{equation*}

With these definitions, the stage quantities and consistency residuals are defined by
\begin{equation}\label{eq:exact_stage_rates_fixedN}
\left\{
\begin{aligned}
\mu_{M,ni}^*={}&C_1\Delta\phi_{M,ni}^*-\frac{C_2}{\varepsilon^2}\phi_{M,ni}^*-U_{M,ni}^*B_{M,ni}^*,\\
\dot T_{M,ni}^*={}&D\Delta T_{M,ni}^*+U_{M,ni}^*\mathcal P_{M,ni}^*,\\
\dot U_{M,ni}^*={}&\frac1{2\tau}(B_{M,ni}^*,\mu_{M,ni}^*)-\frac{\lambda}{2\varepsilon K}(\mathcal P_{M,ni}^*,T_{M,ni}^*),
\end{aligned}
\right.
\end{equation}
\begin{equation}\label{eq:perturbed_stage_vector}
\hspace{-10.5mm}\left\{
\begin{aligned}
\phi_{M,ni}^*&=\phi_{M,n}^*+\frac{\dt}{\tau}\sum_{j=1}^q a_{ij}\mu_{M,nj}^*+R_{M,ni}^{\phi},\\
T_{M,ni}^*&=T_{M,n}^*+\dt\sum_{j=1}^q a_{ij}\dot T_{M,nj}^*+R_{M,ni}^{T},\\
U_{M,ni}^*&=U_{M,n}^*+\dt\sum_{j=1}^q a_{ij}\dot U_{M,nj}^*+R_{M,ni}^{U},
\end{aligned}
\right.
\end{equation}
and
\begin{equation}\label{eq:perturbed_nodal_vector}
\hspace{-6.5mm}\left\{
\begin{aligned}
\phi_{M,n+1}^*&=\phi_{M,n}^*+\frac{\dt}{\tau}\sum_{i=1}^q b_i\mu_{M,ni}^*+R_{M,n+1}^{\phi},\\
T_{M,n+1}^*&=T_{M,n}^*+\dt\sum_{i=1}^q b_i\dot T_{M,ni}^*+R_{M,n+1}^{T},\\
U_{M,n+1}^*&=U_{M,n}^*+\dt\sum_{i=1}^q b_i\dot U_{M,ni}^*+R_{M,n+1}^{U}.
\end{aligned}
\right.
\end{equation}
The quantities $\mu_{M,ni}^*$, $\dot T_{M,ni}^*$, and $\dot U_{M,ni}^*$ in \eqref{eq:exact_stage_rates_fixedN} are obtained by evaluating the right-hand sides of the stage equations using the exact Fourier semi-discrete solution and the extrapolated values. In general, they do not coincide with $\mu_M(t_{ni})$, $\partial_tT_M(t_{ni})$, and $\partial_tU_M(t_{ni})$, respectively, but differ from them by $O(\dt^q)$. Moreover, by the definition of $\Pi_{n-1}$, $\widehat\mu_{M,ni}^*$ depends only on the stage values from the previous time step and is therefore known at the current step.

The final term in each relation of \eqref{eq:perturbed_stage_vector}-\eqref{eq:perturbed_nodal_vector} is the corresponding consistency residual. To state their estimate compactly, we introduce the product space $\mathbb X_M:=X_M\times X_M\times\mathbb R,$
endowed with the norm
\begin{equation*}\label{eq:XN_norm}
\|(v_M,z_M,r)\|_{\mathbb X_M}^2:=\|v_M\|_{H^1}^2+\|z_M\|_{H^1}^2+|r|^2.
\end{equation*}
For convenience, we define the residual vectors by
$$\mathbf R_{M,ni}=(R_{M,ni}^{\phi},R_{M,ni}^{T},R_{M,ni}^{U}),\ 
\mathbf R_{M,n+1}=(R_{M,n+1}^{\phi},R_{M,n+1}^{T},R_{M,n+1}^{U}).$$
The following lemma gives the corresponding consistency estimate.

\begin{lemma}\label{lem32}
Assume that the $q$-stage Runge-Kutta method satisfies the assumptions stated in Subsection~\ref{sec322}, and that the exact Fourier semi-discrete solution satisfies
$$\phi_M,T_M\in W^{q+1,1}(0,t_{\mathrm{end}};X_M),\ U_M\in W^{q+1,1}(0,t_{\mathrm{end}};\mathbb R),$$
then the following consistency error estimate holds:
\begin{equation}\label{eq:fixed_M_consistency}
\|\mathbf R_{M,n+1}\|_{\mathbb X_M} + \sum_{i=1}^q \|\mathbf R_{M,ni}\|_{\mathbb X_M} \le C_M\dt^{q+1}. 
\end{equation}
Here $C_M$ may depend on $M$, the final time $t_{\mathrm{end}}$, the model and Runge-Kutta parameters, and the indicated regularity bounds, but is independent of $\dt$ and $n$.
\end{lemma}

\begin{proof}
The estimate follows from the extrapolation property discussed in Remark~\ref{re32}, the Runge-Kutta order conditions, and the nonlinear estimates in Lemma~\ref{lem31}. For details of the consistency proof, we refer to \cite{Akrivis19,wang2023stable,wang2025class}, with the argument adapted here to the fixed space $X_M$.
\end{proof}

\subsubsection{Error estimates}\label{332}
Let $(\phi_{M,n},T_{M,n},U_{M,n})$ denote the numerical solution obtained from \eqref{ark_phi}-\eqref{ark_n+1} for the Fourier semi-discrete system in $X_M$. We define the nodal errors by
\begin{equation*}\label{eq:fixedN_nodal_error}
e_{\phi,n}=\phi_{M,n}^*-\phi_{M,n},\ e_{T,n}=T_{M,n}^*-T_{M,n},\ e_{U,n}=U_{M,n}^*-U_{M,n},
\end{equation*}
and the stage errors by
\begin{equation*}\label{eq:fixedN_stage_error}
\begin{aligned}
e_{\phi,ni}&=\phi_{M,ni}^*-\phi_{M,ni}, & e_{T,ni}&=T_{M,ni}^*-T_{M,ni},\\
e_{U,ni}&=U_{M,ni}^*-U_{M,ni}, & e_{\mu,ni}&=\mu_{M,ni}^*-\mu_{M,ni}.
\end{aligned}
\end{equation*}
We also define
$$e_{\dot T,ni}=\dot T_{M,ni}^*-\dot T_{M,ni},\ e_{\dot U,ni}=\dot U_{M,ni}^*-\dot U_{M,ni}.$$
For convenience, we introduce the following vector notation:
$$	\mathbf e_n=(e_{\phi,n},e_{T,n},e_{U,n}),\ 	\mathbf e_{ni}=(e_{\phi,ni},e_{T,ni},e_{U,ni}),\ \mathbf k_{ni}=\left(\frac1\tau e_{\mu,ni},e_{\dot T,ni},e_{\dot U,ni}\right).$$

Subtracting the numerical scheme \eqref{ark_phi}-\eqref{ark_n+1} in $X_M$ from \eqref{eq:perturbed_stage_vector}-\eqref{eq:perturbed_nodal_vector} yields
\begin{equation}\label{eq:abstract_stage_error}
\mathbf e_{ni}=\mathbf e_n+\dt\sum_{j=1}^q a_{ij}\mathbf k_{nj}+\mathbf R_{M,ni},
\end{equation}
and
\begin{equation}\label{eq:abstract_nodal_error}
\mathbf e_{n+1}=\mathbf e_n+\dt\sum_{i=1}^q b_i\mathbf k_{ni}+\mathbf R_{M,n+1}.
\end{equation}

Define
$$\errN_n = \|\mathbf e_n\|_{\mathbb X_M}^2, \ \errS_n = \sum_{i=1}^q \|\mathbf e_{ni}\|_{\mathbb X_M}^2, \ \errK_n = \sum_{i=1}^q \|\mathbf k_{ni}\|_{\mathbb X_M}^2.$$
The following theorem gives the temporal error estimate.

\begin{theorem}\label{th33}
Let $M$ be fixed. Assume that the following conditions hold:
\begin{enumerate}
\item The $q$-stage Runge-Kutta method satisfies the conditions in Subsection~\ref{sec322}.
\item The exact Fourier semi-discrete solution satisfies
$$\phi_M,T_M\in W^{q+1,1}(0,t_{\mathrm{end}};X_M),\ U_M\in W^{q+1,1}(0,t_{\mathrm{end}};\mathbb R).$$
\item The starting approximations satisfy
\begin{equation}\label{eq:startup_fixedN}
\mathcal N_2+\dt\mathcal S_1+\dt\mathcal S_0\le C_{0,M}\dt^{2q},
\end{equation}
for some constant $C_{0,M}>0$ independent of $\dt$.
\end{enumerate}
Then there exists $\dt_M>0$ such that, for $0<\dt\le\dt_M$,
\begin{equation}\label{eq:global_temporal_error_fixedN}
\max_{2\le n\le N_t-1}\left(\mathcal N_{n+1}+\dt\mathcal S_n\right)\le C_M\dt^{2q}.
\end{equation}
Equivalently,
$$\begin{aligned}
&\max_{2\le n\le N_t-1}\Bigl\{\|\phi_M(t_{n+1})-\phi_{M,n+1}\|_{H^1}^2+\|T_M(t_{n+1})-T_{M,n+1}\|_{H^1}^2\\
&\qquad\qquad+|U_M(t_{n+1})-U_{M,n+1}|^2\\
&\quad+\dt\sum_{i=1}^q\Bigl(\|\phi_{M,ni}^*-\phi_{M,ni}\|_{H^1}^2+\|T_{M,ni}^*-T_{M,ni}\|_{H^1}^2+|U_{M,ni}^*-U_{M,ni}|^2\Bigr)\Bigr\}\le C_M\dt^{2q}.
\end{aligned}$$
The constants $C_M$ and $\dt_M$ may depend on $M$, $t_{\mathrm{end}}$, the model and Runge-Kutta parameters, the indicated regularity bounds, and $C_{0,M}$, but are independent of $\dt$ and $n$.
\end{theorem}

\begin{proof}
We prove the result by mathematical induction. Fix an integer $n_0$ with $2\le n_0\le N_t-1$ and assume that the numerical stage values up to step $n_0-1$ remain in the prescribed bounded neighborhood of the corresponding exact stage values. Under this induction hypothesis, Lemma~\ref{lem31} applies at all preceding stages, with constants independent of $n_0$.

We first note the following stability estimate for the extrapolation operator. For each extrapolated variable $v$, we have
\begin{equation}\label{eq:extrapolation_stability_fixedN}
\sum_{i=1}^q\|\widehat v_{M,ni}\|_{H^1}^2\le C_\Pi\sum_{i=1}^q\|v_{M,n-1,i}\|_{H^1}^2,
\end{equation}
where $C_\Pi$ depends only on the Runge-Kutta nodes. This follows directly from the boundedness of the fixed $q$-point Lagrange extrapolation operator.

\medskip
\noindent\textbf{Step 1.}
Our aim in this step is to prove that, for $2\le n\le n_0$,
\begin{equation}\label{eq:rate_error_estimate_fixedN}
\errK_n\le C_M\left(\errS_n+\errS_{n-1}+\errS_{n-2}\right).
\end{equation}
By the definition of $\mathcal K_n$, it suffices to estimate $e_{\mu,ni}$, $e_{\dot T,ni}$, and $e_{\dot U,ni}$, since $\tau$ is fixed. We begin with $e_{\mu,ni}$. Let
$$B_{M,ni}=G_M(\widehat\phi_{M,ni})+H_M(\widehat\phi_{M,ni},\widehat T_{M,ni}).$$
Subtracting the numerical relation for $\mu_{M,ni}$ in \eqref{ark_phi} from the first relation in \eqref{eq:exact_stage_rates_fixedN} gives
$$e_{\mu,ni}=C_1\Delta e_{\phi,ni}-\frac{C_2}{\varepsilon^2}e_{\phi,ni}-e_{U,ni}B_{M,ni}-U_{M,ni}^*(B_{M,ni}^*-B_{M,ni}).$$
To control the last term, we need to estimate $B_{M,ni}^*-B_{M,ni}$. Using Lemma~\ref{lem31}, the linearity of the extrapolation operator, and \eqref{eq:extrapolation_stability_fixedN}, we have
\begin{equation}\label{eq:B_difference_estimate}
\begin{aligned}
\sum_{i=1}^q\|B_{M,ni}^*-B_{M,ni}\|_{H^1}^2&\le C_M\sum_{i=1}^q\left(\|\widehat\phi_{M,ni}^*-\widehat\phi_{M,ni}\|_{H^1}^2+\|\widehat T_{M,ni}^*-\widehat T_{M,ni}\|_{H^1}^2\right)\\
&=C_M\sum_{i=1}^q\left(\|\Pi_{n-1}e_{\phi,ni}\|_{H^1}^2+\|\Pi_{n-1}e_{T,ni}\|_{H^1}^2\right)\\
&\le C_M\sum_{i=1}^q\left(\|e_{\phi,n-1,i}\|_{H^1}^2+\|e_{T,n-1,i}\|_{H^1}^2\right)\le C_M\errS_{n-1}.
\end{aligned}
\end{equation}
Combining \eqref{eq:inverse_equivalence}, \eqref{eq:operator_uniform_bound}, and \eqref{eq:B_difference_estimate} with the boundedness of $U_{M,ni}^*$ yields, for $1\le n\le n_0$,
\begin{equation}\label{eq:emu_sum_estimate}
\begin{aligned}
\sum_{i=1}^q\|e_{\mu,ni}\|_{H^1}^2&\le C_M\sum_{i=1}^q\left(\|e_{\phi,ni}\|_{H^1}^2+|e_{U,ni}|^2+\|B_{M,ni}^*-B_{M,ni}\|_{H^1}^2\right)\\
&\le C_M\left(\errS_n+\errS_{n-1}\right).
\end{aligned}
\end{equation}

We next estimate $e_{\dot T,ni}$. Let
$$\mathcal P_{M,ni}=P_M(\widehat\phi_{M,ni},\widehat\mu_{M,ni}).$$
Subtracting the numerical relation for $\dot T_{M,ni}$ in \eqref{ark_u} from the second relation in \eqref{eq:exact_stage_rates_fixedN} gives
\begin{equation*}\label{eq:T_rate_error_fixedN}
e_{\dot T,ni}=D\Delta e_{T,ni}+e_{U,ni}\mathcal P_{M,ni}+U_{M,ni}^*(\mathcal P_{M,ni}^*-\mathcal P_{M,ni}).
\end{equation*}
An estimate of $\mathcal P_{M,ni}^*-\mathcal P_{M,ni}$ requires a bound for the extrapolated error of $\mu$. Since $\Pi_{n-1}$ is linear, we have
$$\widehat e_{\mu,ni}:=\widehat\mu_{M,ni}^*-\widehat\mu_{M,ni}=\Pi_{n-1}e_{\mu,ni}.$$
Applying \eqref{eq:extrapolation_stability_fixedN} and then
\eqref{eq:emu_sum_estimate} with $n$ replaced by $n-1$, we obtain
\begin{equation}\label{eq:emu_extrapolated_estimate}
\sum_{i=1}^q\|\widehat e_{\mu,ni}\|_{H^1}^2\le C_M\sum_{i=1}^q\|e_{\mu,n-1,i}\|_{H^1}^2\le C_M\left(\errS_{n-1}+\errS_{n-2}\right).
\end{equation}
Lemma~\ref{lem31}, \eqref{eq:extrapolation_stability_fixedN}, and \eqref{eq:emu_extrapolated_estimate} then give
\begin{equation}\label{eq:P_difference_estimate}
\begin{aligned}
\sum_{i=1}^q\|\mathcal P_{M,ni}^*-\mathcal P_{M,ni}\|_{H^1}^2&\le C_M\sum_{i=1}^q\left(\|\widehat\phi_{M,ni}^*-\widehat\phi_{M,ni}\|_{H^1}^2+\|\widehat e_{\mu,ni}\|_{H^1}^2\right)\\
&\le C_M\left(\errS_{n-1}+\errS_{n-2}\right).
\end{aligned}
\end{equation}
From \eqref{eq:inverse_equivalence}, \eqref{eq:operator_uniform_bound}, \eqref{eq:P_difference_estimate}, and the boundedness of $U_{M,ni}^*$, we obtain
\begin{equation}\label{eq:T_rate_sum_fixedN}
\sum_{i=1}^q\|e_{\dot T,ni}\|_{H^1}^2\le C_M\left(\errS_n+\errS_{n-1}+\errS_{n-2}\right).
\end{equation}

It remains to estimate $e_{\dot U,ni}$. Subtracting the numerical relation for $\dot U_{M,ni}$ in \eqref{ark_r} from the third relation in \eqref{eq:exact_stage_rates_fixedN} gives
\begin{equation*}\label{eq:U_rate_error_fixedN}
\begin{aligned}
e_{\dot U,ni}={}&\frac1{2\tau}\left[(B_{M,ni}^*-B_{M,ni},\mu_{M,ni}^*)+(B_{M,ni},e_{\mu,ni})\right]\\
&-\frac{\lambda}{2\varepsilon K}\left[(\mathcal P_{M,ni}^*-\mathcal P_{M,ni},T_{M,ni}^*)+(\mathcal P_{M,ni},e_{T,ni})\right].
\end{aligned}
\end{equation*}
The Cauchy-Schwarz inequality, together with \eqref{eq:operator_uniform_bound}, \eqref{eq:B_difference_estimate}, \eqref{eq:emu_sum_estimate},  \eqref{eq:P_difference_estimate}, and the boundedness of the exact stage quantities, gives
\begin{equation}\label{eq:U_rate_sum_fixedN}
\sum_{i=1}^q|e_{\dot U,ni}|^2\le C_M\left(\errS_n+\errS_{n-1}+\errS_{n-2}\right).
\end{equation}
Combining \eqref{eq:emu_sum_estimate}, \eqref{eq:T_rate_sum_fixedN}, and \eqref{eq:U_rate_sum_fixedN} with the definition of $\errK_n$, we obtain \eqref{eq:rate_error_estimate_fixedN}.

\medskip
\noindent\textbf{Step 2.}
For $2\le n\le n_0$ and sufficiently small $\dt$, we establish the following two estimates. Here and below, $C>0$ denotes a constant independent of $M$, $\dt$, and $n$, whereas $C_M$ may additionally depend on $M$.
\begin{equation}\label{eq:stage_error_fixedN}
\errS_n\le C\errN_n+C_M\dt^2\left(\errS_{n-1}+\errS_{n-2}\right)+C_M\dt^{2q+2},
\end{equation}
and
\begin{equation}\label{eq:nodal_error_fixedN}
\errN_{n+1}\le(1+C_M\dt)\errN_n+C_M\dt\left(\errS_{n-1}+\errS_{n-2}\right)+C_M\dt^{2q+1}.
\end{equation}

We first prove \eqref{eq:stage_error_fixedN}. From \eqref{eq:abstract_stage_error}, for each $i=1,\ldots,q$, we have
$$\|\mathbf e_{ni}\|_{\mathbb X_M}^2\le C\|\mathbf e_n\|_{\mathbb X_M}^2+C\dt^2\sum_{j=1}^q\|\mathbf k_{nj}\|_{\mathbb X_M}^2+C\|\mathbf R_{M,ni}\|_{\mathbb X_M}^2.$$
Summing over $i$ and applying \eqref{eq:fixed_M_consistency} and \eqref{eq:rate_error_estimate_fixedN} yields
$$\errS_n\le C\errN_n+C_M\dt^2\left(\errS_n+\errS_{n-1}+\errS_{n-2}\right)+C_M\dt^{2q+2}.$$
Since $M$ is fixed, choose $\dt_M>0$ such that 
$$C_M\dt_M^2\le\dfrac{1}{2},\ \dt_M\le1.$$
Then for $0<\dt\le\dt_M$, we have $C_M\dt^2\le 1/2$, and hence the term $C_M\dt^2\errS_n$ can be absorbed into the left-hand side, which proves \eqref{eq:stage_error_fixedN}.

It remains to prove \eqref{eq:nodal_error_fixedN}. From \eqref{eq:abstract_nodal_error}, we have
$$\mathbf e_{n+1}=\mathbf e_n+\dt\sum_{i=1}^q b_i\mathbf k_{ni}+\mathbf R_{M,n+1}.$$
Applying Young's inequality to the above relation yields
\begin{equation*}
\begin{aligned}
	\errN_{n+1}&\le (1+\dt)\errN_n+\left(1+\frac1{\dt}\right)\left\|\dt\sum_{i=1}^q b_i\mathbf k_{ni}+\mathbf R_{M,n+1}\right\|_{\mathbb X_M}^2\\
	&\le(1+C\dt)\errN_n+C\dt\errK_n+\frac{C}{\dt}\|\mathbf R_{M,n+1}\|_{\mathbb X_M}^2,
\end{aligned}
\end{equation*}
where we have used $\dt\le1$ and
$$\left\|\sum_{i=1}^q b_i\mathbf k_{ni}\right\|_{\mathbb X_M}^2\le C\sum_{i=1}^q\|\mathbf k_{ni}\|_{\mathbb X_M}^2=C\errK_n.$$
Applying \eqref{eq:fixed_M_consistency} and \eqref{eq:rate_error_estimate_fixedN} gives
$$\errN_{n+1}\le(1+C\dt)\errN_n+C_M\dt\left(\errS_n+\errS_{n-1}+\errS_{n-2}\right)+C_M\dt^{2q+1}.$$
Substituting \eqref{eq:stage_error_fixedN} into this estimate yields
$$\errN_{n+1}\le(1+C_M\dt)\errN_n+C_M(\dt+\dt^3)\left(\errS_{n-1}+\errS_{n-2}\right)+C_M\left(\dt^{2q+1}+\dt^{2q+3}\right).$$
Since $\dt\le1$, we have
$$\dt+\dt^3\le2\dt,\ \dt^{2q+1}+\dt^{2q+3}\le2\dt^{2q+1},$$
which proves \eqref{eq:nodal_error_fixedN}.

\medskip
\noindent\textbf{Step 3.}
In this step, we combine \eqref{eq:stage_error_fixedN} and \eqref{eq:nodal_error_fixedN} by introducing the quantity
\begin{equation*}\label{eq:augmented_error_fixedN}
\errJ_{n+1} := \errN_{n+1} + a\dt\errS_n + b\dt\errS_{n-1},
\end{equation*}
where $a,b>0$ are independent of $\Delta t$ and $n$.
Let $L_M>0$ be sufficiently large to bound the $M$-dependent constants appearing in \eqref{eq:stage_error_fixedN} and \eqref{eq:nodal_error_fixedN}. Multiplying \eqref{eq:stage_error_fixedN} by $a\dt$ and adding the result to \eqref{eq:nodal_error_fixedN}, with $b\dt \mathcal S_{n-1}$ added to both sides, we obtain
$$\begin{aligned}
\errJ_{n+1}\le{}&\left[1+(L_M+aC)\dt\right]\errN_n\\
&+\left[(L_M+b)\dt+aL_M\dt^3\right]\errS_{n-1}\\
&+\left[L_M\dt+aL_M\dt^3\right]\errS_{n-2}+C_M\dt^{2q+1}.
\end{aligned}$$
Since $M$ is fixed, choose $b>L_M$ and then $a>L_M+b$. Thus,
$$b-L_M>0,\ a-(L_M+b)>0.$$
Reducing $\dt_M$ if necessary, we assume that, for $0<\dt\le\dt_M$,
$$aL_M\dt^2\le b-L_M,\ aL_M\dt^2\le a-(L_M+b).$$
These conditions imply
$$L_M\dt+aL_M\dt^3\le b\dt,\ (L_M+b)\dt+aL_M\dt^3\le a\dt.$$
Therefore, recalling that
$$\errJ_n=\errN_n+a\dt\errS_{n-1}+b\dt\errS_{n-2},$$
and enlarging $C_M$ if necessary, we obtain
\begin{equation*}\label{eq:augmented_recursion_fixedN}
\errJ_{n+1}\le(1+C_M\dt)\errJ_n+C_M\dt^{2q+1},\ 2\le n\le n_0.
\end{equation*}
It follows that, for $3\le n\le n_0+1$,
$$\errJ_n\le (1+C_M\dt)^{n-2}\errJ_2+C_M\dt^{2q+1}\sum_{j=0}^{n-3}(1+C_M\dt)^j.$$
Since $n\dt\le t_{\mathrm{end}}$, we have
$$(1+C_M\dt)^{n-2}\le e^{C_Mt_{\mathrm{end}}}\le C_M,\ \sum_{j=0}^{n-3}(1+C_M\dt)^j\le\frac{C_M}{\dt}.$$
Substituting these bounds into the preceding inequality yields
$$\errJ_n\le C_M\errJ_2+C_M\dt^{2q}.$$
Moreover, \eqref{eq:startup_fixedN} and the definition of $\errJ_2$ give
$$\begin{aligned}
\errJ_2&=\errN_2+a\dt\errS_1+b\dt\errS_0\\
&\le \max\{1,a,b\}\left(\errN_2+\dt\errS_1+\dt\errS_0\right)\le C_M\dt^{2q}.
\end{aligned}$$
Combining the last two estimates yields
\begin{equation}\label{eq:J_global_bound_fixedN}
	\max_{2\le n\le n_0+1}\errJ_n\le C_M\dt^{2q}.
\end{equation}
Finally, since $a,b>0$, the definition of $\mathcal J_{n+1}$ gives
$$\mathcal N_{n+1}+\dt \mathcal S_n \le C_M\mathcal J_{n+1}.$$
Together with \eqref{eq:J_global_bound_fixedN}, this inequality yields
$$\max_{2\le n\le n_0}\left(\errN_{n+1}+\dt\errS_n\right)\le C_M\dt^{2q}.$$

\medskip
\noindent\textbf{Step 4.}
It remains to close the induction argument. By \eqref{eq:J_global_bound_fixedN} and the definition of $\errJ_{n+1}$, we have
$$\errS_n\le C_M\dt^{2q-1},\ 2\le n\le n_0.$$
Therefore,
$$\max_{2\le n\le n_0}\max_{1\le i\le q}\|\mathbf e_{ni}\|_{\mathbb X_M}\le C_M\dt^{q-\frac12}.$$
Moreover, combining \eqref{eq:emu_sum_estimate} and the preceding bound yields
$$\max_{2\le n\le n_0}\max_{1\le i\le q}\|e_{\mu,ni}\|_{H^1}\le C_M\dt^{q-\frac12}.$$
For sufficiently small $\dt_M$, chosen independently of $n_0$, the above bounds ensure that the numerical stage values at step $n_0$ remain in the prescribed bounded neighborhood. This completes the induction. Therefore, the argument applies for all $2\le n\le N_t-1$, and \eqref{eq:global_temporal_error_fixedN} follows. Finally, the equivalent error estimate is obtained directly from the definitions of $\errN_{n+1}$ and $\errS_n$.
\end{proof}


\subsection{Fully discrete implementation}\label{sec34}
Having established the temporal convergence of the proposed Runge-Kutta discretization in a fixed Fourier-Galerkin space, we now describe the Fourier spectral spatial discretization used in the numerical implementation. The grid resolution $N$ introduced below is distinct from the Galerkin dimension $M$ used in the temporal convergence analysis. For simplicity, we consider the computational domain $\Omega=(0,2\pi)^2$. 
The extension to general rectangular periodic domains is straightforward. Then the Fourier space on $\Omega$ is defined by
\begin{equation*}\label{F_space}
X_{N}:=\mbox{\rm{span}}\{e^{\mathrm{i}(kx+ly)}:|k|\leq \dfrac{N}{2},\ |l|\leq \dfrac{N}{2}\},\ k,l\in\Z,\ x,y\in \R,
\end{equation*}
where $\mathrm{i}=\sqrt{-1}$ and $N$ is generally taken as a positive even integer. Let
\begin{equation*}
x_{j_1}=j_1h=j_1\dfrac{2\pi }{N},\quad y_{j_2}=j_2h=j_2\dfrac{2\pi }{N},\quad 0\leq j_1,j_2\leq N-1,
\end{equation*}
be the uniform grid points on $\Omega$.	
	
The fully discrete implementation proceeds as follows.

\textbf{(1) Computation of $g_n$, $B_n$, $h_n$, and $P_n$.}

We use $g_n$ to illustrate the procedure, since $B_n$, $h_n$, and $P_n$ are computed in a similar way. Let $g_n= V\hat{g}_n$, $\hat{g}_n=(\hat{g}_{n1},\ldots,\hat{g}_{nq})^T$, then \eqref{g_B_1} can be written as 
\begin{equation}\label{g_dis}
\left(\Lambda -\frac{\dt C_1}{\tau }\Delta I+\frac {\dt C_2}{\tau\varepsilon^2}I\right)\hat{g}_n=\mathcal{G}_n,
\end{equation}
where 
\begin{equation*}
\mathcal{G}_n= V^{-1}\phi_n A^{-1}\mathbb{1}=(\mathcal{G}_{n1},\ldots,\mathcal{G}_{nq})^T.
\end{equation*}
Thus, \eqref{g_dis} can be decomposed into solving $q$ linear elliptic equations with constant coefficients:
\begin{equation*}\label{g_dis_q}
\left(\lambda_i -\frac{\dt C_1}{\tau }\Delta I+\frac {\dt C_2}{\tau\varepsilon^2}I\right)\hat{g}_{ni}=\mathcal{G}_{ni}, \quad i=1, \ldots, q.
\end{equation*}
Applying the Fourier transform in space yields
\begin{equation}\label{g_dis_q1}
\left(\lambda_i +\frac{\dt C_1}{\tau }(k^2+l^2)+\frac {\dt C_2}{\tau\varepsilon^2}\right)\tilde{\hat{g}}_{ni,k,l}=\tilde{\mathcal{G}}_{ni,k,l}, \quad i=1, \ldots, q,
\end{equation}
where $\tilde{\mathcal{G}}_{ni,k,l}$ is the discrete Fourier coefficient of $\mathcal{G}_{ni}$, which can be obtained by the following discrete Fourier transform:
$$\tilde{\mathcal{G}}_{ni,k,l}=\dfrac{1}{N^2w_kw_l}\sum_{j_1=0}^{N-1}\sum_{j_2=0}^{N-1}\mathcal{G}_{ni}(x_{j_1},y_{j_2})e^{-i(kx_{j_1}+ly_{j_2})},\quad k,l=-N/2,\ldots,N/2,$$
where $w_k=1$ for $|k|<N/2$ and $w_k=2$ for $k=\pm N/2$, with $w_l$ defined analogously.
After solving \eqref{g_dis_q1} for	$\tilde{\hat g}_{ni,k,l}$, the function $\hat g_n$ can be obtained by the inverse Fourier transform:
\begin{equation*}
\hat{g}_{ni}(x_{j_1},y_{j_2})=\sum_{k,l=-N/2}^{N/2}\tilde{\hat{g}}_{ni,k,l}e^{i(kx+ly)},
\end{equation*}
which yields $g_n$. $B_n$, $h_n$ and $P_n$ can be obtained in the same way.

For a fixed number of stages $q$, the eigendecomposition of $A^{-1}$ needs to be computed only once. The resulting constant-coefficient elliptic equations are independent and can be solved in parallel by FFTs. When $\dt$ is fixed, the corresponding factors in Fourier space can be precomputed and reused at each time step.

\textbf{(2) Evaluation of $f_n$, $s_n$, $F_n$, and $S_n$.}

Let $f_n=\frac{1}{2\tau}(f_{n1},\ldots,f_{nq})^T$. The components $f_{ni}$ are evaluated by the following discrete inner product:
$$f_{ni}=\dfrac{1}{N^2}\sum_{j_1=0}^{N-1}\sum_{j_2=0}^{N-1}b_{ni}(x_{j_1},y_{j_2})\mathcal{F}_{ni}(x_{j_1},y_{j_2}),$$
where
\begin{equation*}
\mathcal{F}_n=\tau A^{-1}(g_n-\phi_n\mathbb{1})=(\mathcal{F}_{n1},\ldots,\mathcal{F}_{nq})^T.
\end{equation*}
The same calculation yields $s_n, F_n$ and $S_n$.

\textbf{(3) Determination of $\mathcal U_n$.}
	
The vector $\mathcal U_n$ is obtained by solving the $q\times q$ linear system \eqref{de_r3}.

\textbf{(4) Computation of the remaining variables and solution update.}
	
Once $\mathcal U_n$ is determined, $\Phi_n$, $\mathcal T_n$, $\mathcal{M}_n$, $\dot{\mathcal T}_n$, and $\dot{\mathcal U}_n$ can be computed directly. The solution at the next time level is then updated through \eqref{ark_n+1}.

\section{Numerical examples}\label{sec4}
\setcounter{equation}{0}
In this section, we examine the numerical performance of the proposed schemes from several aspects. We first test the temporal accuracy using smooth fabricated solutions and compare the computational performance of the decoupled and coupled implementations. We then verify the discrete energy dissipation property. The two-dimensional simulations investigate the effects of anisotropy, latent heat, orientation angle, and the initial nucleus on dendritic growth. Schemes of different orders are also compared in these tests. Finally, we consider three-dimensional dendritic growth with fourfold anisotropy. Unless otherwise stated, the computations are carried out on the periodic domain $(0,2\pi)^d$, $d=2,3$, and the Fourier spectral method is used for spatial discretization.

\subsection{Accuracy test}\label{sec41}
To verify the convergence rates of the proposed schemes, we consider the two-dimensional benchmark problem in \cite{wang2025class} with fourfold anisotropy ($m=4$). Its exact solution is given by
\begin{equation*}
\phi(x,y,t)=\sin(t)\cos(x)\cos(y),\quad T(x,y,t)=\sin(t)\sin(x)\sin(y).
\end{equation*}
Suitable source terms are introduced to ensure that the above solution satisfies the governing system \eqref{aeq}. The parameters are set as
\begin{equation*}
\begin{gathered}
	C_1=0.9,\quad C_2=10,\quad C_0=10^4,\quad \tau=4\times10^3,\quad \varepsilon=0.1,\\
	\epsilon_4=5\times10^{-2},\quad \lambda=0.1,\quad D=2.25\times10^{-2},\quad K=10^{-2},\quad \theta_0=0^\circ.
\end{gathered}
\end{equation*}

We use $128$ Fourier modes in each spatial direction and vary the time step size to examine the temporal convergence orders. Figure~\ref{fig_example1_1} shows the $H^1$-errors of $\phi$ and $T$ at $t_{\mathrm{end}}=1$ obtained by the first-order IMEX scheme and the $q$-stage Gauss schemes with $q=2,3,4,5$. The observed convergence rates agree well with the theoretical predictions. Since the first-order IMEX scheme is significantly less accurate than the high-order schemes, it is not considered in the remaining numerical tests.

To assess the computational efficiency of the proposed decoupling strategy,we compare the decoupled implementation (Scheme I) with the original coupled scheme in \cite{wang2025class} (Scheme II).\. The CPU times of the two schemes for different values of $q$ and time step sizes $\dt$ are listed in Table~\ref{table_example1}. 
It can be seen that Scheme I consistently requires less CPU time than Scheme II in all the tested cases. The resulting speedup, measured by the ratio of the CPU time of Scheme II to that of Scheme I, ranges from approximately 3.6 to 18.5, depending on $q$ and $\dt$.
Figure~\ref{fig_example1_2} further shows the $H^1$-errors of $\phi$ and $T$ versus CPU time. For a given CPU time, Scheme I generally yields smaller errors than Scheme II. It also requires less CPU time than Scheme II to achieve a given error level. These results demonstrate the computational advantage of the proposed matrix diagonalization technique. All computations were performed in MATLAB R2025b on a Lenovo computer equipped with an Intel Core Ultra 9 275HX processor and 32 GB of RAM.

\begin{figure*}[htbp]
	\begin{minipage}[t]{0.49\linewidth}
		\centerline{\includegraphics[scale=0.5]{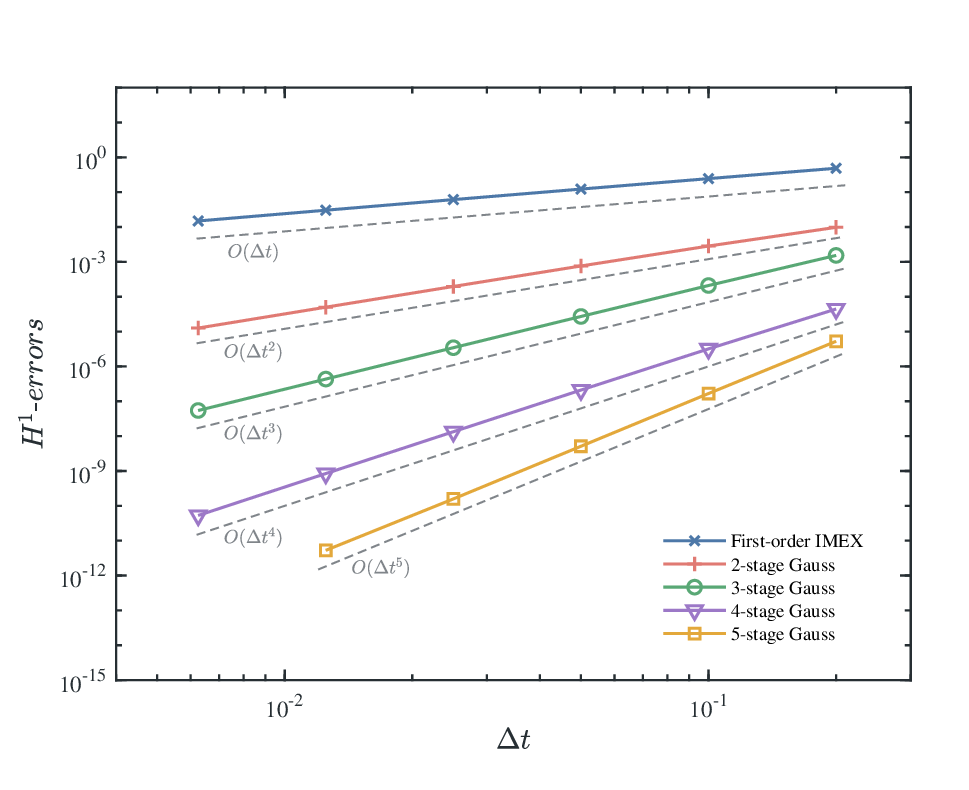}}
		\centerline{(a) Phase field $\phi$}
	\end{minipage}
	\begin{minipage}[t]{0.49\linewidth}
		\centerline{\includegraphics[scale=0.5]{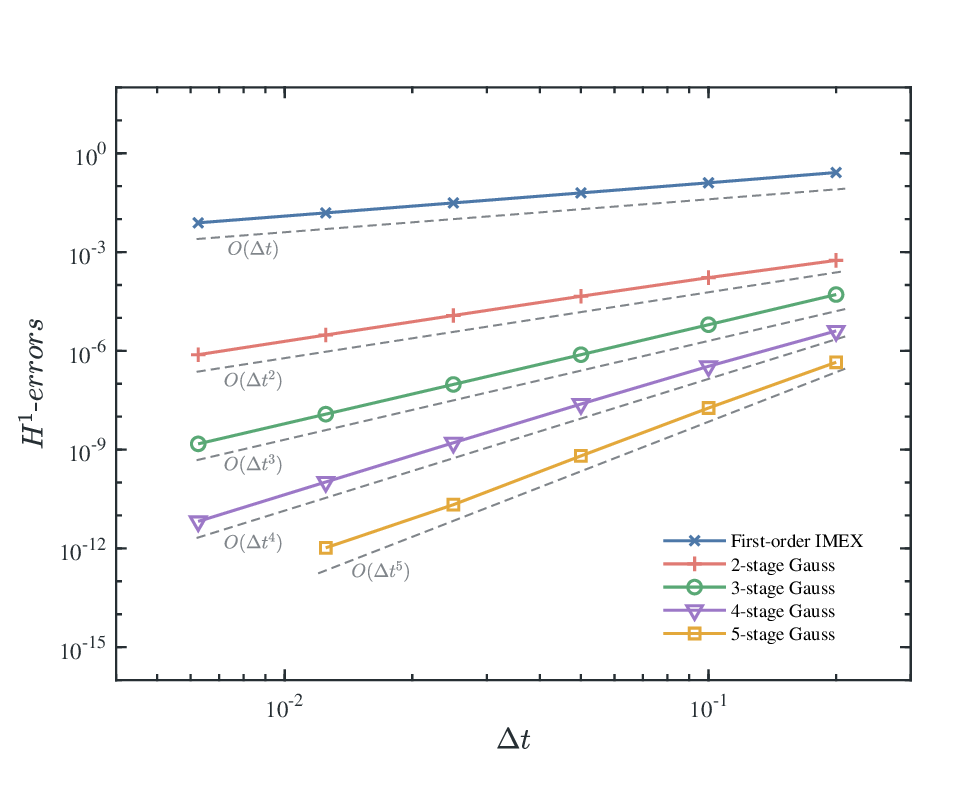}}
		\centerline{(b) Temperature $T$}
	\end{minipage}
	\caption{ $H^1$-errors of $\phi$ and $T$ at $t_{\mathrm{end}}=1$ for the first-order IMEX scheme and the $q$-stage Gauss schemes with $q=2,3,4,5$ using different time step sizes.}\label{fig_example1_1}
	\vskip 8mm 
\end{figure*}

\begin{table}[!htbp]
	\footnotesize
	\begin{tabular}{ccccccccc}
		\hline
		\multirow{2}{*}{$\dt$} & \multicolumn{2}{c}{$q=2$} & \multicolumn{2}{c}{$q=3$} & \multicolumn{2}{c}{$q=4$} & \multicolumn{2}{c}{$q=5$} \\ \cline{2-9} 
		& Scheme I   & Scheme II   & Scheme I   & Scheme II   & Scheme I   & Scheme II  & Scheme I   & Scheme II  \\ \hline
		1/5   & 0.05 & 0.38 & 0.08 & 1.48 & 0.18 & 1.09 & 0.26 & 1.63 \\
        1/10  & 0.10 & 0.52 & 0.19 & 1.78 & 0.41 & 2.30 & 0.58 & 2.94 \\
        1/20  & 0.22 & 1.01 & 0.40 & 3.47 & 0.90 & 4.55 & 1.23 & 6.06 \\
        1/40  & 0.46 & 1.87 & 0.86 & 6.49 & 1.77 & 8.54 & 2.53 & 12.26 \\
        1/80  & 0.94 & 3.35 & 1.63 & 11.38 & 3.59 & 14.93 & 5.15 & 24.93 \\ \hline
	\end{tabular}
	\vskip 3mm 
	\caption{CPU times (in seconds) for the decoupled implementation (Scheme I) and the coupled implementation (Scheme II) for $q=2,3,4,5$ and different time step sizes.
	}\label{table_example1}
\end{table}

\begin{figure*}[htbp]
    \begin{minipage}[t]{0.49\linewidth}
		\centerline{\includegraphics[scale=0.5]{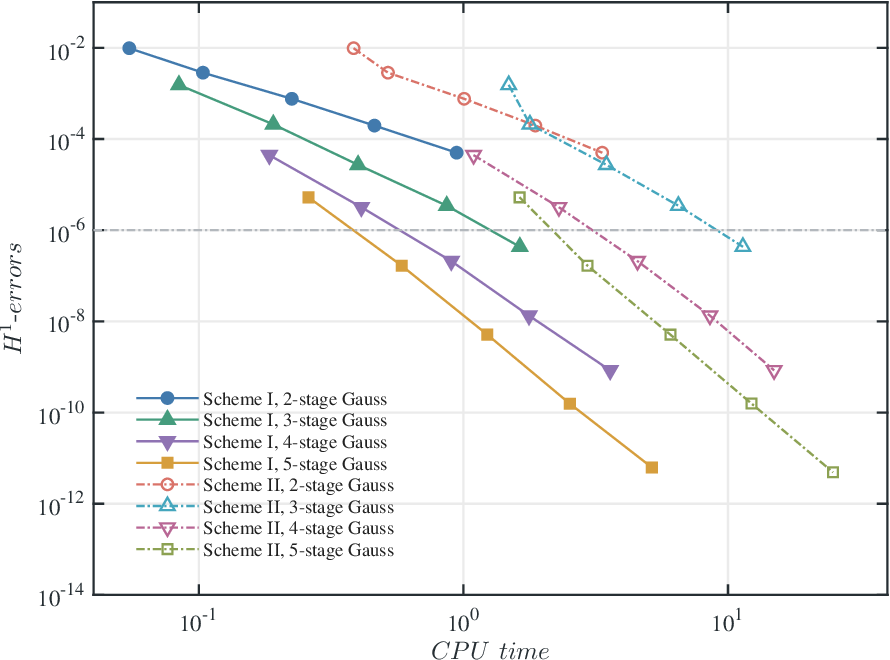}}
		\centerline{(a) Phase field $\phi$}
	\end{minipage}
	\begin{minipage}[t]{0.49\linewidth}
		\centerline{\includegraphics[scale=0.5]{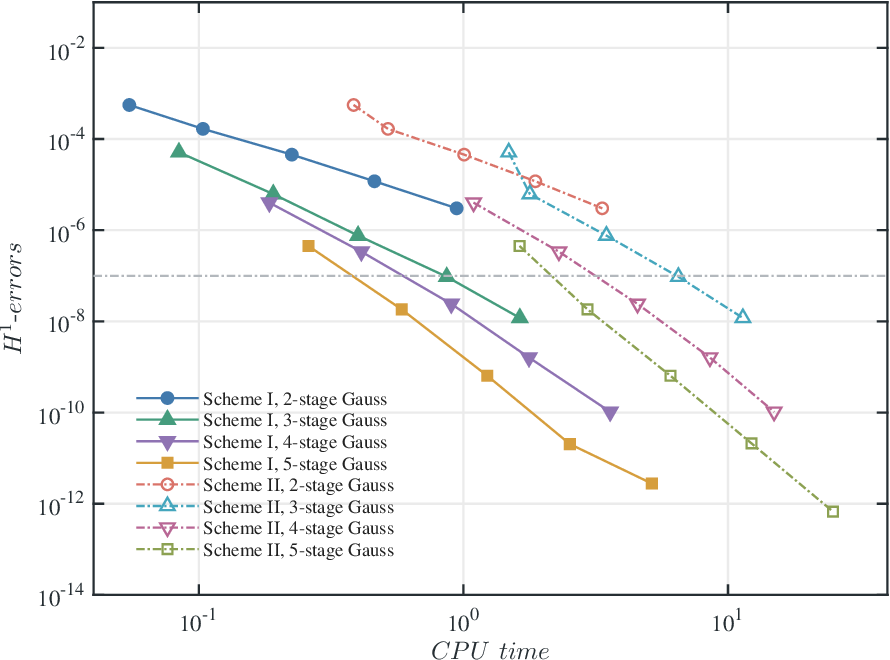}}
		\centerline{(b) Temperature $T$}
	\end{minipage}
	\caption{$H^1$-errors of $\phi$ and $T$ versus CPU time for the decoupled implementation (Scheme I) and the coupled implementation (Scheme II) of the $q$-stage Gauss schemes with $q=2,3,4,5$. }\label{fig_example1_2}
\end{figure*}

\subsection{Stability test}\label{sec42}
In this subsection, we examine the discrete energy stability of the proposed schemes for the fourfold anisotropy model \eqref{aeq} and investigate the effect of the stabilization parameters $C_1$ and $C_2$ on numerical oscillations. The initial conditions are chosen as
\begin{equation*}\label{ex2_initial}
\phi^0(\x)=\tanh\left(\dfrac{r_0-|\x-\x_0|^2}{\epsilon_0}\right),\quad T^0(\x)=-0.5\phi^0(\x),
\end{equation*}
where $\x_0=(\pi,\pi),\ r_0=0.25\pi,\ \epsilon_0=0.1$. The source terms are taken to be zero, and the remaining parameters are given by
\begin{equation*}\label{ex2_params}
\begin{gathered}
	C_1=0.9,\quad C_2=10,\quad C_0=10^5,\quad \tau=10^3,\quad \varepsilon=0.1,\\
	\epsilon_4=6\times10^{-2},\quad \lambda=1,\quad D=10^{-3},\quad K=0.1,\quad \theta_0=0^\circ.
\end{gathered}
\end{equation*}

Figure~\ref{fig_example2_1} shows that the anisotropy coefficient $\kappa(\nabla\phi)$ exhibits strong oscillatory behavior, which motivates the following stability test. We use the same spatial discretization as in Subsection~\ref{sec41}, together with the $3$-stage Gauss method in time. Figure~\ref{fig_example2_2} shows the evolution of the modified energy for several time step sizes. The energy decreases monotonically for all the tested values of $\dt$, in agreement with the discrete energy dissipation law.

To further examine the effect of the stabilizers, we consider the quantity $\xi^{n}=U^{n}/\sqrt{E_1(\phi^{n})}$. We then compare the following two choices while keeping all other settings unchanged:
\begin{equation*}\label{stabilizers}
(1)\ C_1=C_2=0; \qquad (2)\ C_1=0.9,\ C_2=10.
\end{equation*}
Figure~\ref{fig_example2_3} compares the two cases with the same time step size $\dt=0.1$. Without the stabilizers, $\xi^n$ exhibits pronounced oscillations and deviates substantially from $1$. By contrast, with $C_1=0.9$ and $C_2=10$, $\xi^n$ remains close to $1$ and the oscillations are effectively reduced. These results demonstrate the effectiveness of the two stabilization terms.

\begin{figure*}[htbp]
	\begin{minipage}[t]{0.49\linewidth}
		\centerline{\includegraphics[scale=0.5]{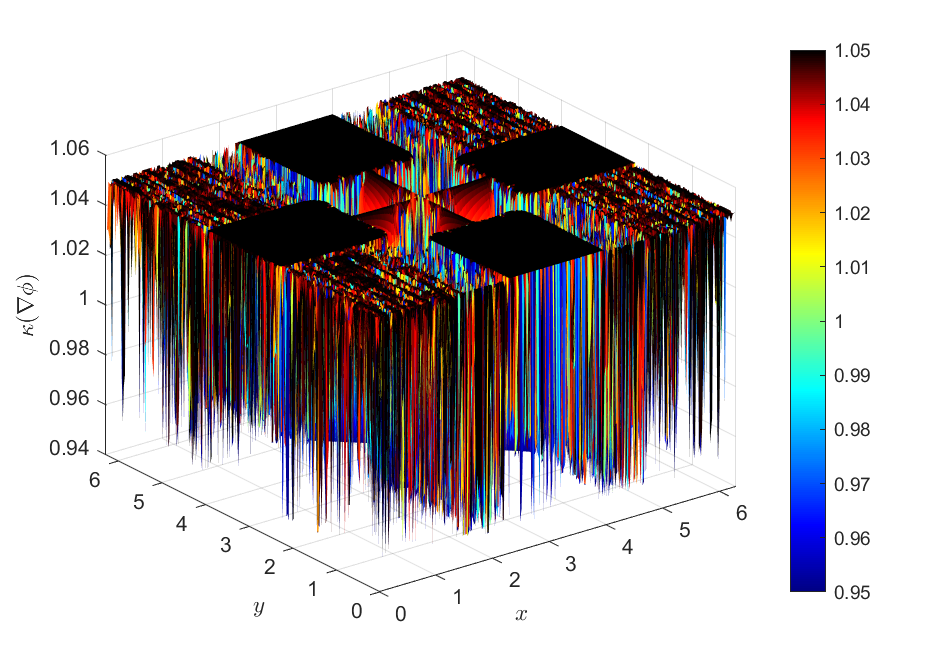}}
		\centerline{(a) Profile of $\kappa(\nabla \phi)$}
	\end{minipage}
	\begin{minipage}[t]{0.49\linewidth}
		\centerline{\includegraphics[scale=0.5]{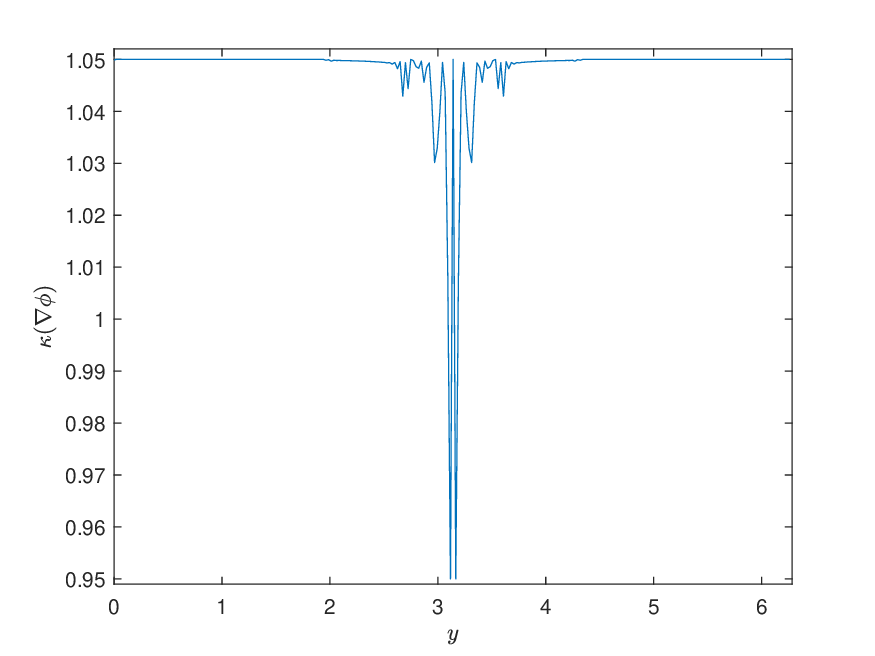}}
		\centerline{(b) $\kappa(\nabla \phi)$ at $x=3.117$}
	\end{minipage}
	\caption{Spatial variation of the anisotropy coefficient $\kappa(\nabla \phi)$ at the initial time.
	}\label{fig_example2_1}
\end{figure*}

\begin{figure*}[htbp]
	\centering
	\includegraphics[width=0.5\textwidth]{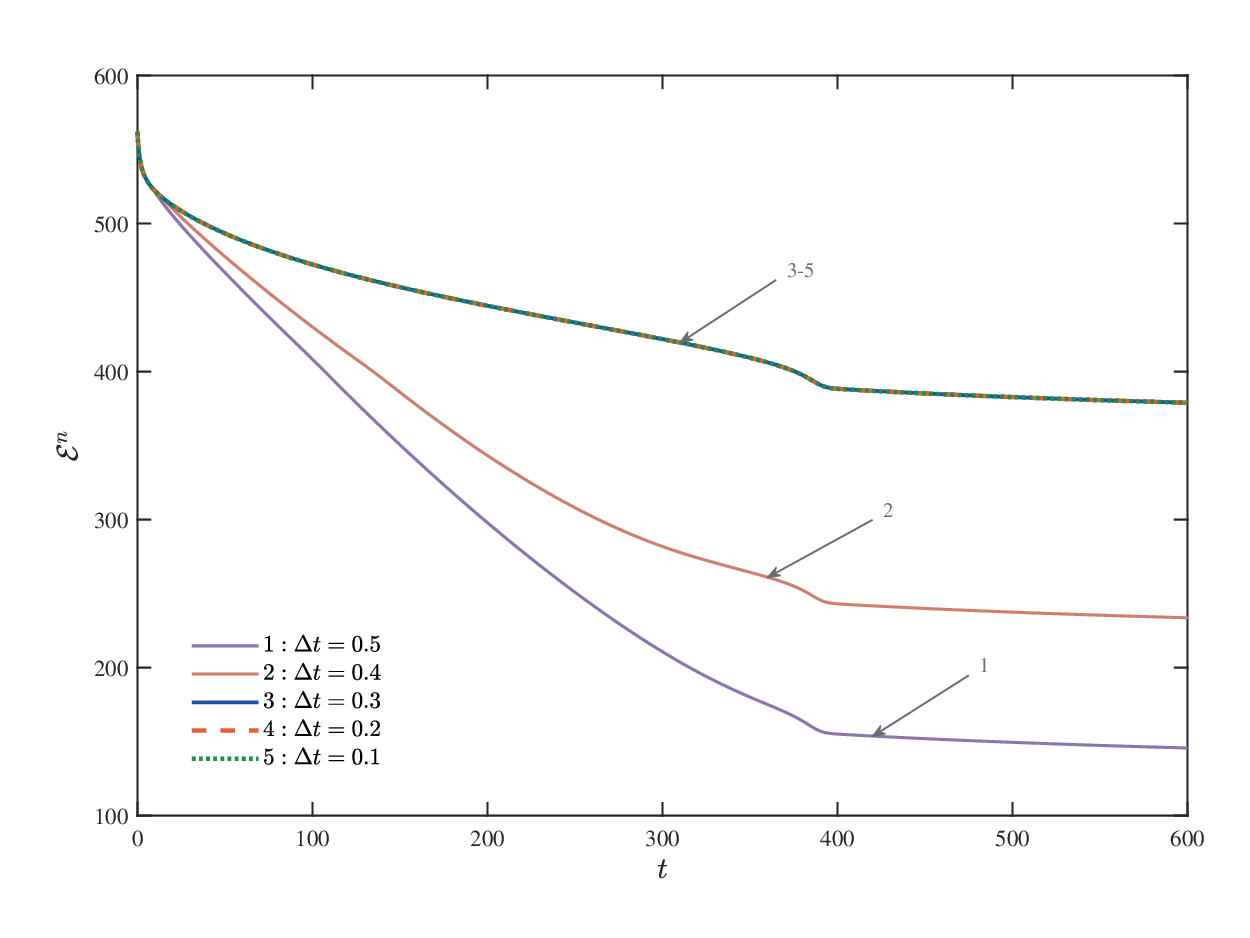}
	\caption{Evolution of the modified energy $\mathcal{E}_{n}$ for different time step sizes.}\label{fig_example2_2}
\end{figure*}

\begin{figure*}[htbp]
	\begin{minipage}[t]{0.49\linewidth}
		\centerline{\includegraphics[scale=0.4]{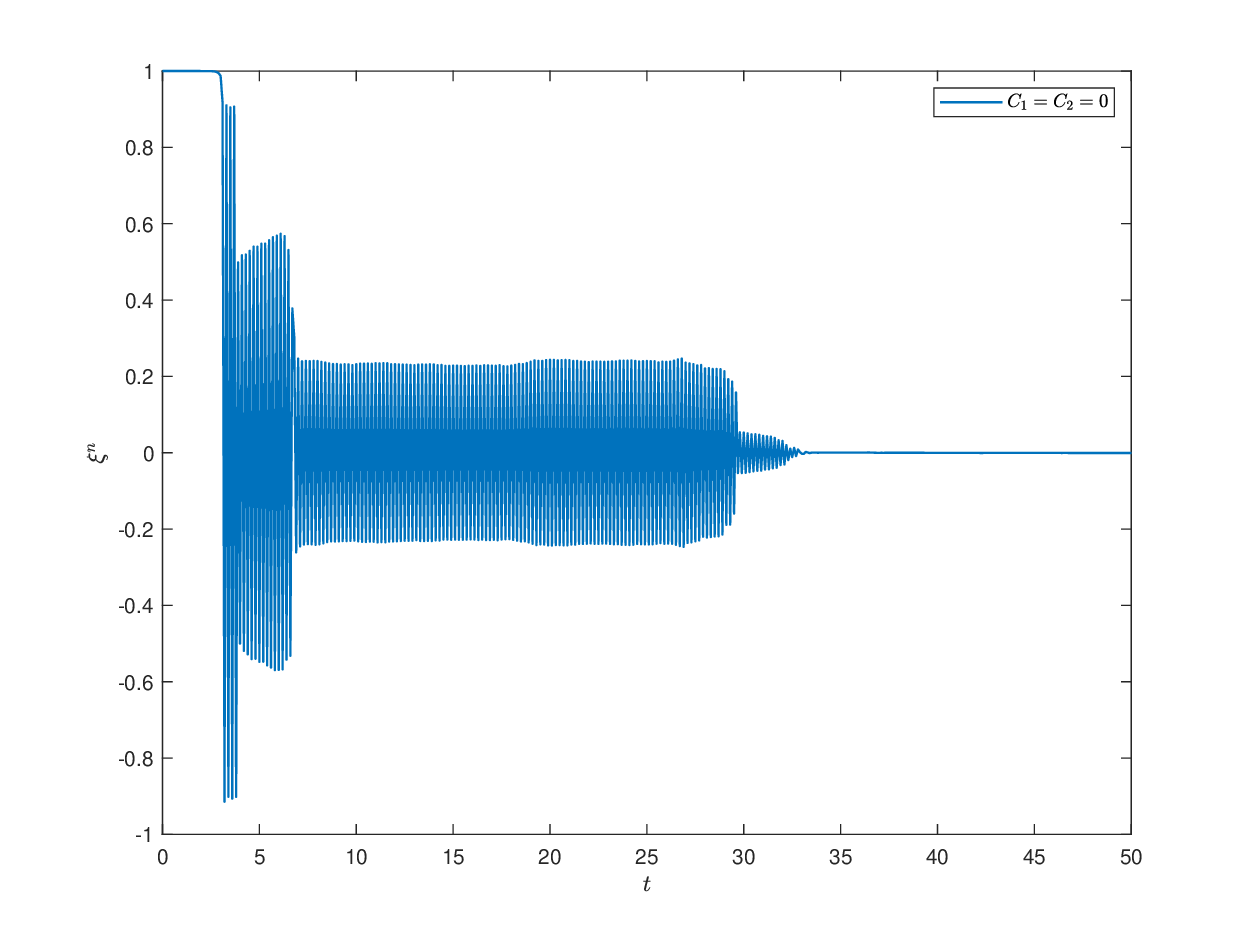}}
		\centerline{(a) $C_1=C_2=0$}
	\end{minipage}
	\begin{minipage}[t]{0.49\linewidth}
		\centerline{\includegraphics[scale=0.4]{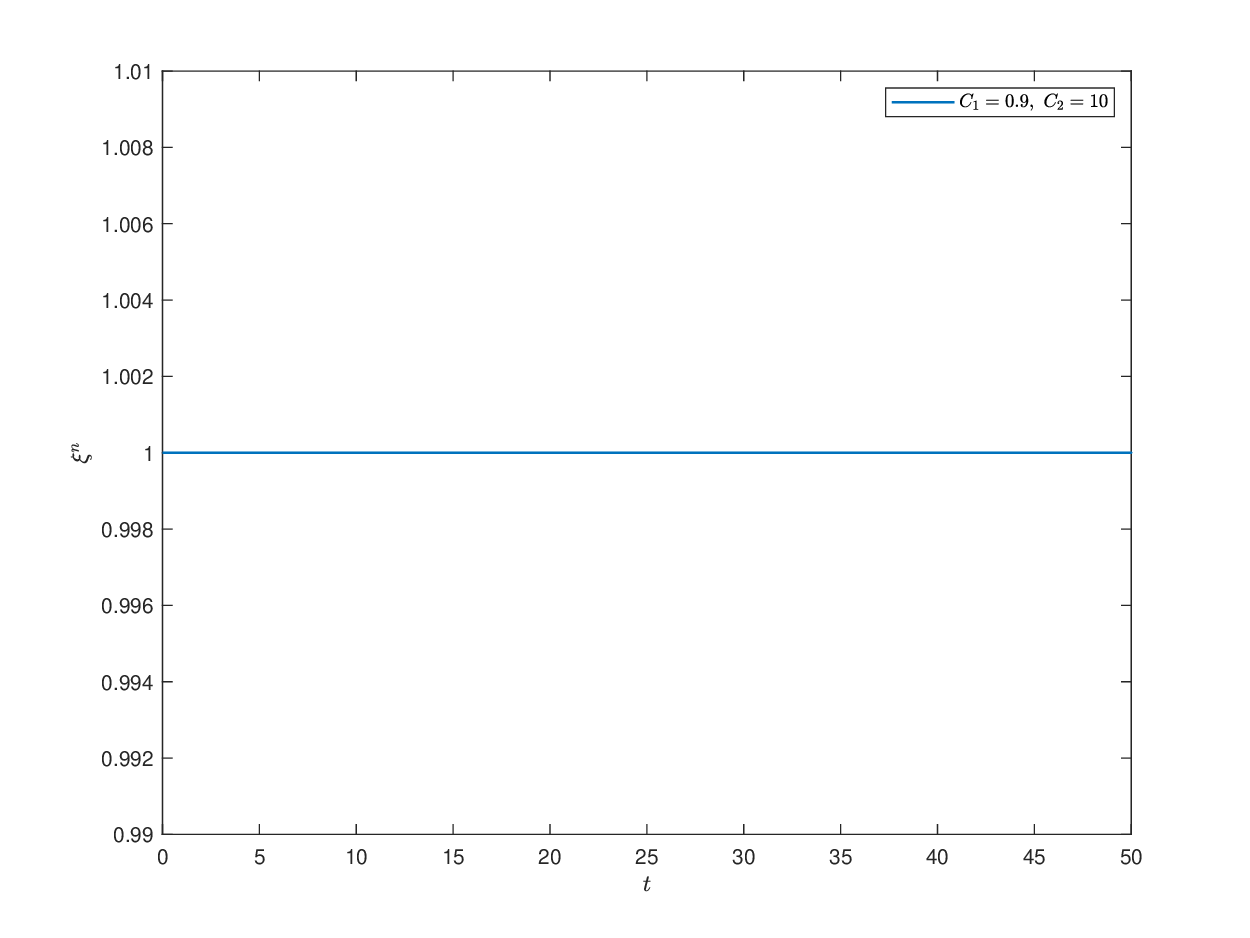}}
		\centerline{(b) $C_1=0.9,\ C_2=10$}
	\end{minipage}
	\caption{Comparison of $\xi^n$ for two choices of the stabilization parameters with $\dt=0.1$.
	}\label{fig_example2_3}
\end{figure*}

\subsection{2D dendritic crystal growth}\label{sec43}
In this subsection, we present two-dimensional simulations of dendritic crystal growth with different anisotropy settings and initial configurations. Unless otherwise stated, we use $512$ Fourier modes in each spatial direction and the 4-stage Gauss method in time with $\dt=0.01$.

We first consider the fourfold anisotropy case. The initial conditions are given by
\begin{equation}\label{ex3_initial}
\phi^0(\x)=\tanh\left(\dfrac{r_0-|\x-\x_0|^2}{\epsilon_0}\right),\quad T^0(\x)= \begin{cases} 0,\qquad \quad \phi^0(\x)>0,\\ -0.55, \quad \mbox{otherwise}, \end{cases}
\end{equation}
where $\x_0=(\pi,\pi),\ r_0=2.7e-2,\ \epsilon_0=5.4e-3$. The remaining model parameters are chosen as
\begin{equation}\label{ex3_params}
\begin{gathered}
	C_1=0.6,\quad C_2=10,\quad C_0=5\times10^5,\quad \tau=6\times10^2,\\
	\varepsilon=1.5\times10^{-2},\quad \lambda=3.6\times10^2,\quad D=2.25\times10^{-3}.
\end{gathered}
\end{equation}

To examine the effect of the latent heat coefficient, we set $m=4$, $\epsilon_4=0.04$, and $\theta_0=0^\circ$, and vary $K$ from $0.6$ to $0.9$ while keeping the other parameters in \eqref{ex3_params} fixed. The computations are carried out to $t_{\mathrm{end}}=35$. Figures~\ref{fig_example3_K6}-\ref{fig_example3_K9} show the evolution of the phase field $\phi$ and temperature $T$ for the four values of $K$. In all cases, the initially circular nucleus develops into a fourfold dendritic structure with four dominant branches. The temperature profiles exhibit a similar dendritic pattern and are consistent with the phase transition behavior. To further illustrate the influence of $K$ on the dendritic morphology, Figure~\ref{fig_example3_contour} shows the evolution of the interface $\phi=0$ at intervals of two time units up to $t=20$. As $K$ increases, the dendrite branches become thinner and the tips become sharper.

Figure~\ref{fig_example3_1} also presents the modified energy and crystal size for different values of $K$. The crystal size is measured by the radius of an equivalent circle corresponding to the solid-phase area $\int_{\Omega}\frac{1+\phi}{2}\,d\boldsymbol{x}$. Figure~\ref{fig_example3_1}(a) shows that the modified energy decreases monotonically for all the tested values of $K$, in agreement with the discrete energy dissipation law. Figure~\ref{fig_example3_1}(b) shows that a smaller value of $K$ leads to a larger crystal size. These numerical observations are consistent with those reported in \cite{kobayashi1993modeling,karma1998,Yang2019,2019Zhang,Lmh2022,guo2024efficient}.

With $K=0.7$ and the parameters in \eqref{ex3_params} unchanged, we vary the anisotropy strength $\epsilon_4$ to investigate its influence on the crystal morphology. Figure~\ref{fig_example3_e4} compares the results for $\epsilon_4=0.01$, $0.02$, $0.04$, and $0.08$. The fourfold character of the dendritic morphology becomes increasingly evident as $\epsilon_4$ increases.

We compare the $2$-, $3$-, and $4$-stage Gauss schemes for $K=0.7$ and $\epsilon_4=0.04$. The time step size is set to $\dt=0.01$, and $256$ Fourier modes are used in each spatial direction. Figure~\ref{fig_example3_compare} presents the phase-field and temperature profiles along the cross-section $y=2.0126$ at $t=20$. The reference profiles, labeled as ``Reference'', are computed using the $5$-stage Gauss scheme with a $512\times512$ Fourier approximation and $\dt=0.001$. The enlarged views show that the results of the $4$-stage scheme are closest to the reference profiles, followed by the $3$-stage results, while larger deviations are observed for the $2$-stage scheme.

\begin{figure*}[htbp]
	\begin{minipage}[t]{0.24\linewidth}
		\centerline{\includegraphics[scale=0.28]{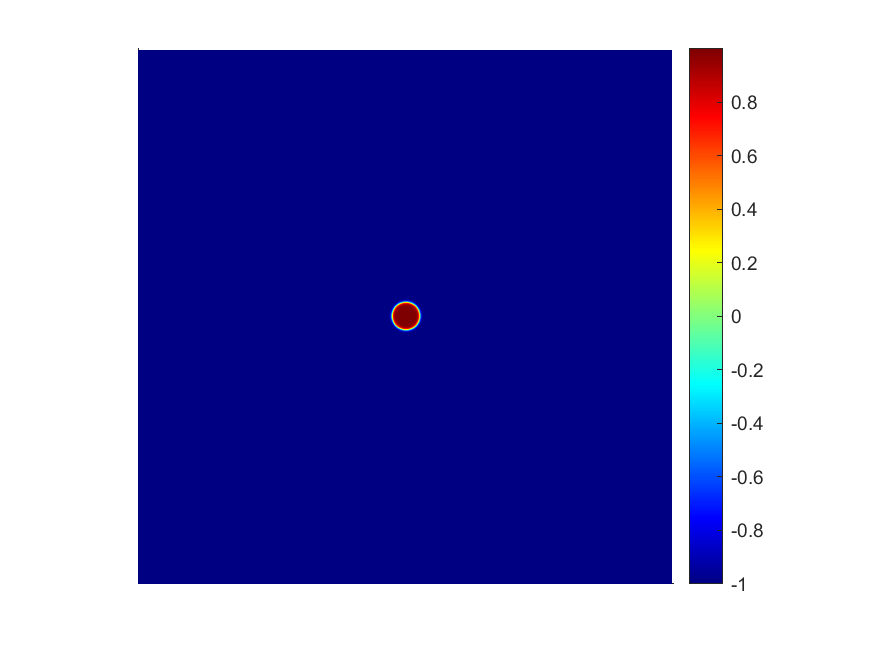}}
	\end{minipage}
	\begin{minipage}[t]{0.24\linewidth}
		\centerline{\includegraphics[scale=0.28]{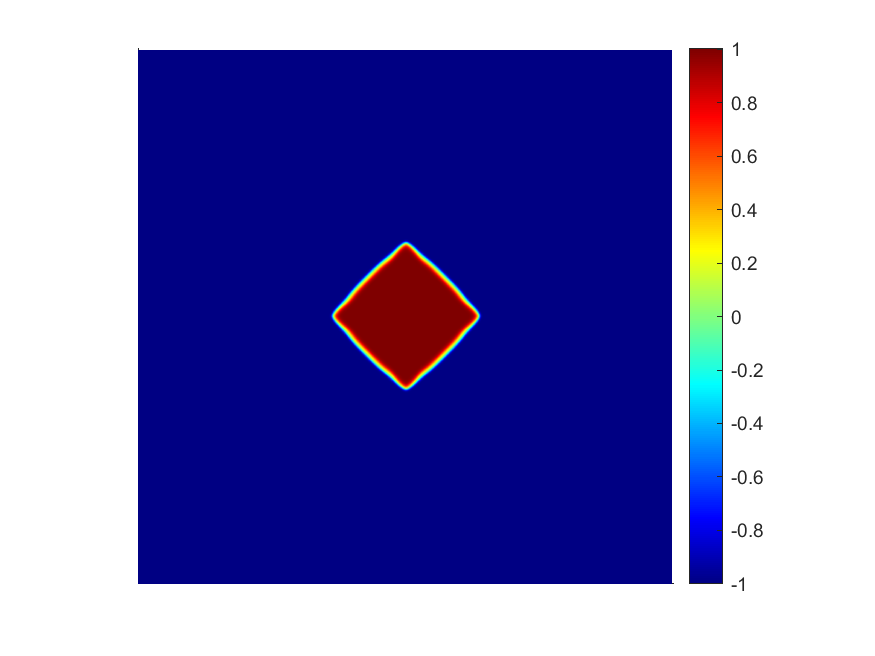}}
	\end{minipage}
	\begin{minipage}[t]{0.24\linewidth}
		\centerline{\includegraphics[scale=0.28]{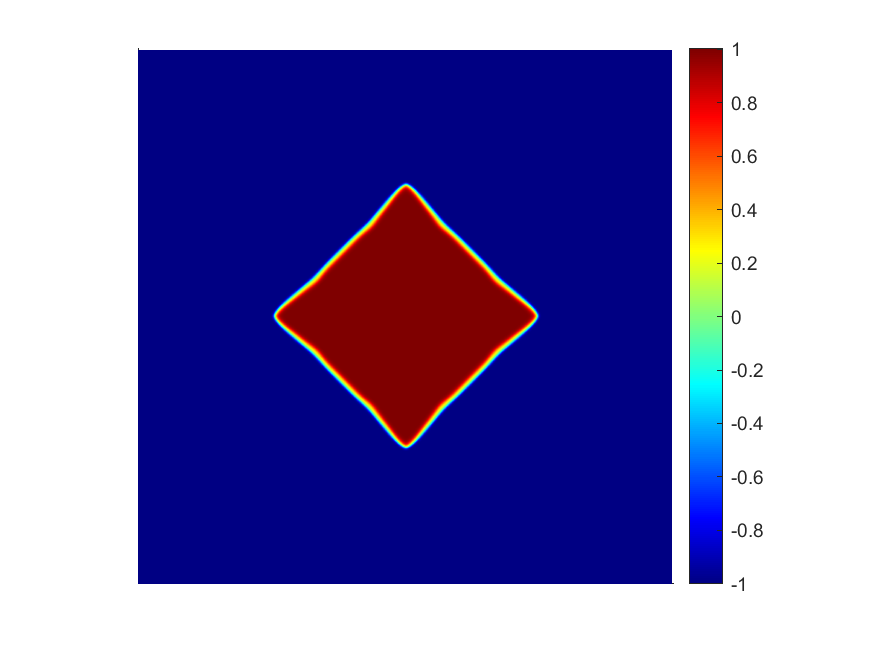}}
	\end{minipage}
	\begin{minipage}[t]{0.24\linewidth}
		\centerline{\includegraphics[scale=0.28]{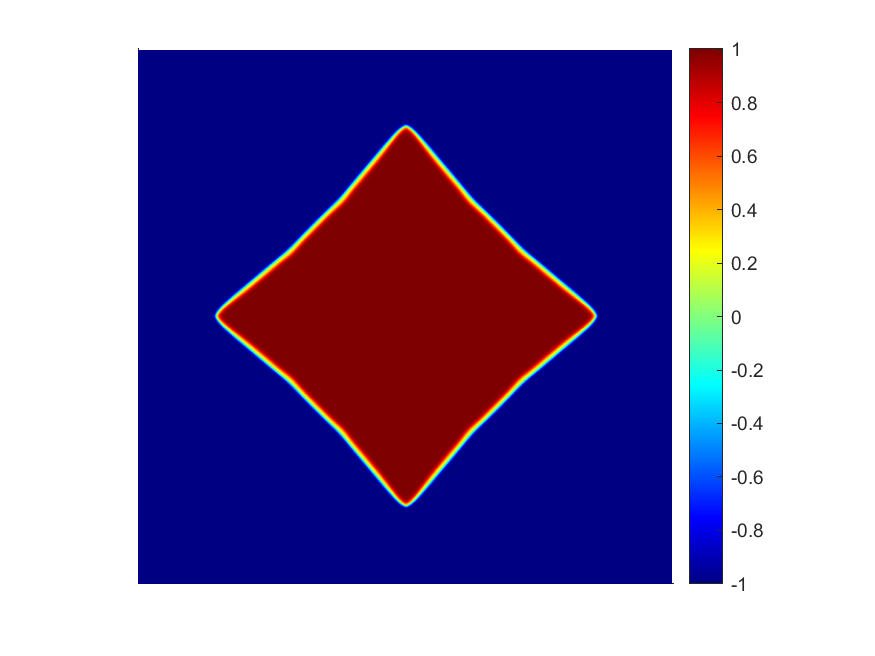}}
	\end{minipage}
	\centerline{(a) Phase field $\phi$}
	\vskip 3mm
	\begin{minipage}[t]{0.24\linewidth}
		\centerline{\includegraphics[scale=0.28]{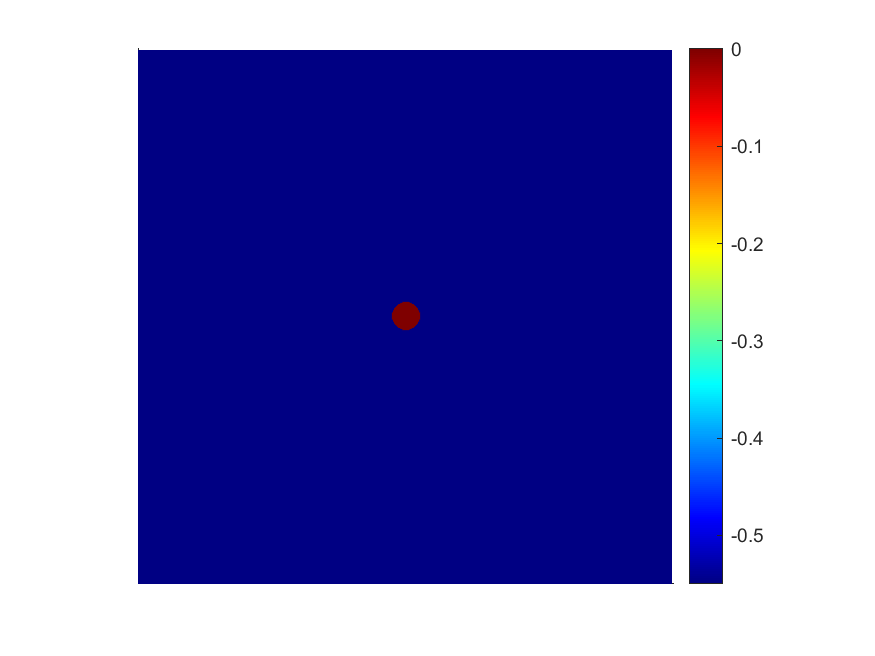}}
	\end{minipage}
	\begin{minipage}[t]{0.24\linewidth}
		\centerline{\includegraphics[scale=0.28]{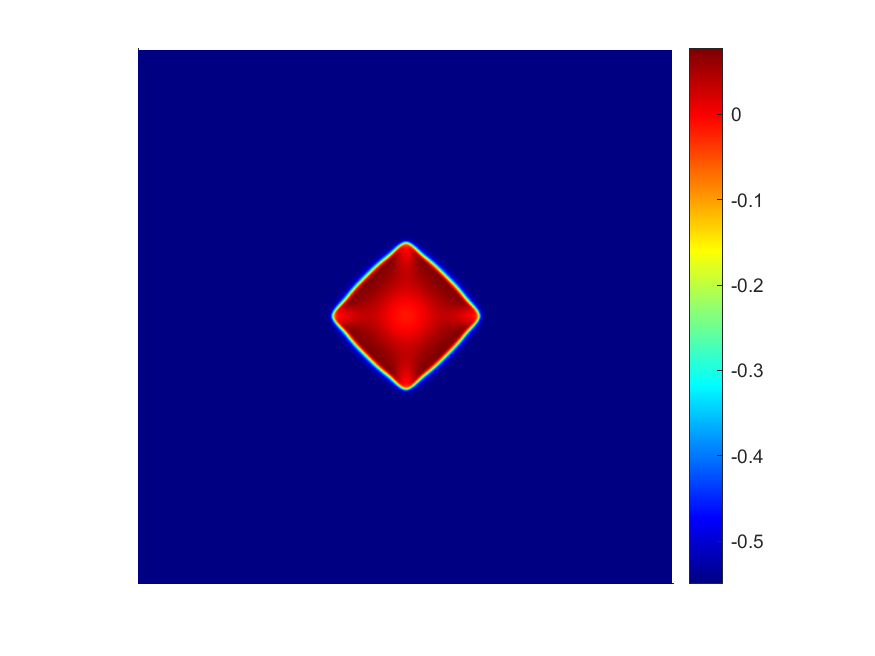}}
	\end{minipage}
	\begin{minipage}[t]{0.24\linewidth}
		\centerline{\includegraphics[scale=0.28]{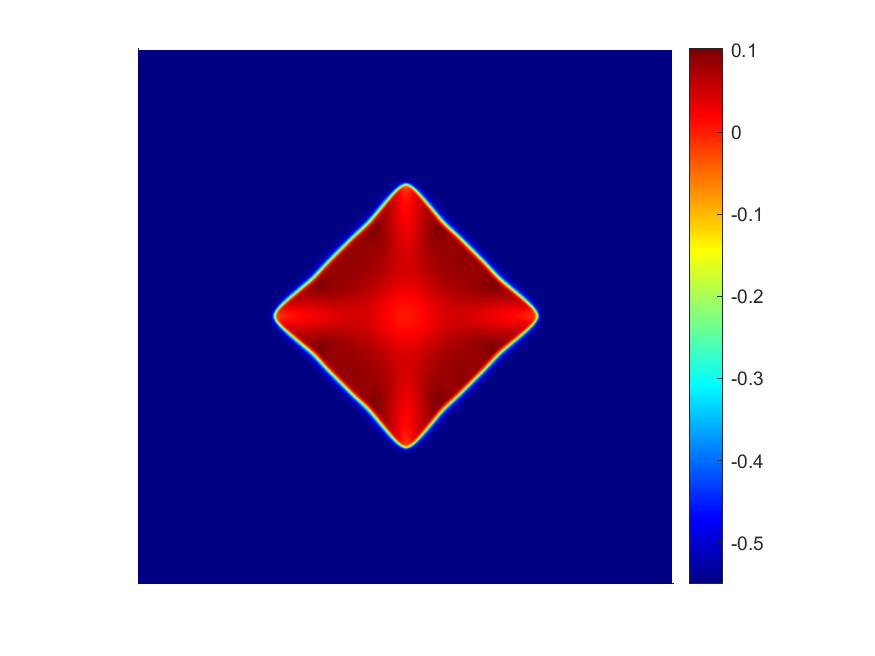}}
	\end{minipage}
	\begin{minipage}[t]{0.24\linewidth}
		\centerline{\includegraphics[scale=0.28]{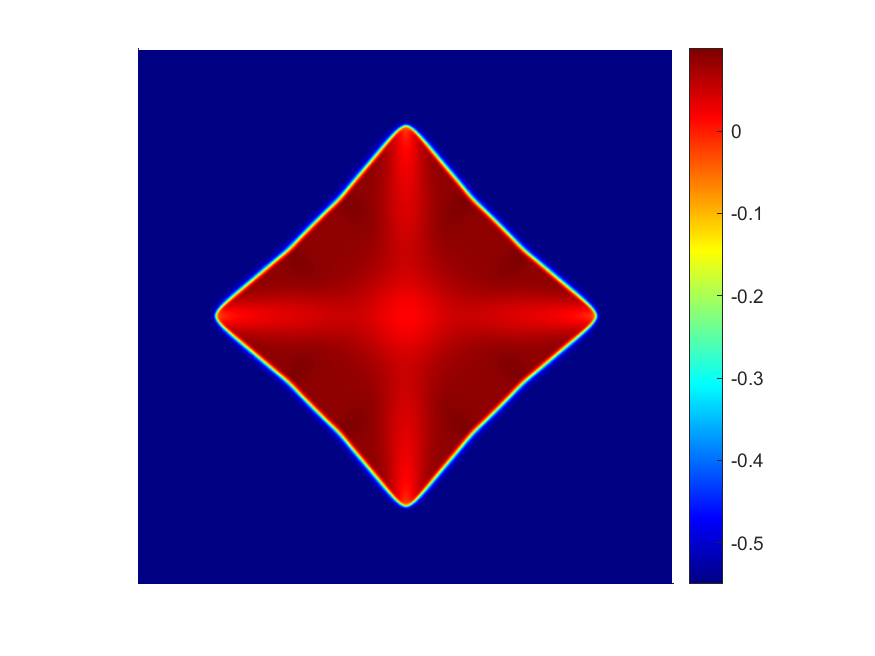}}
	\end{minipage}
	\centerline{(b) Temperature $T$}
	\caption{Evolution of the phase field $\phi$ and temperature $T$ for fourfold anisotropy with $K=0.6$ at $t=0, 5, 10, 15$.
	}\label{fig_example3_K6}
\end{figure*}

\begin{figure*}[htbp]
	\begin{minipage}[t]{0.24\linewidth}
		\centerline{\includegraphics[scale=0.28]{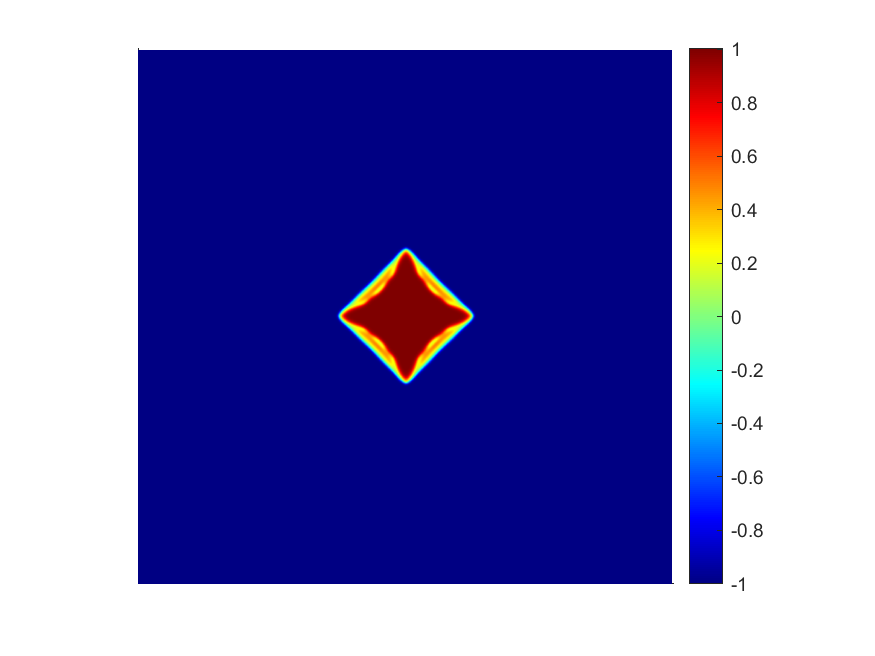}}
	\end{minipage}
	\begin{minipage}[t]{0.24\linewidth}
		\centerline{\includegraphics[scale=0.28]{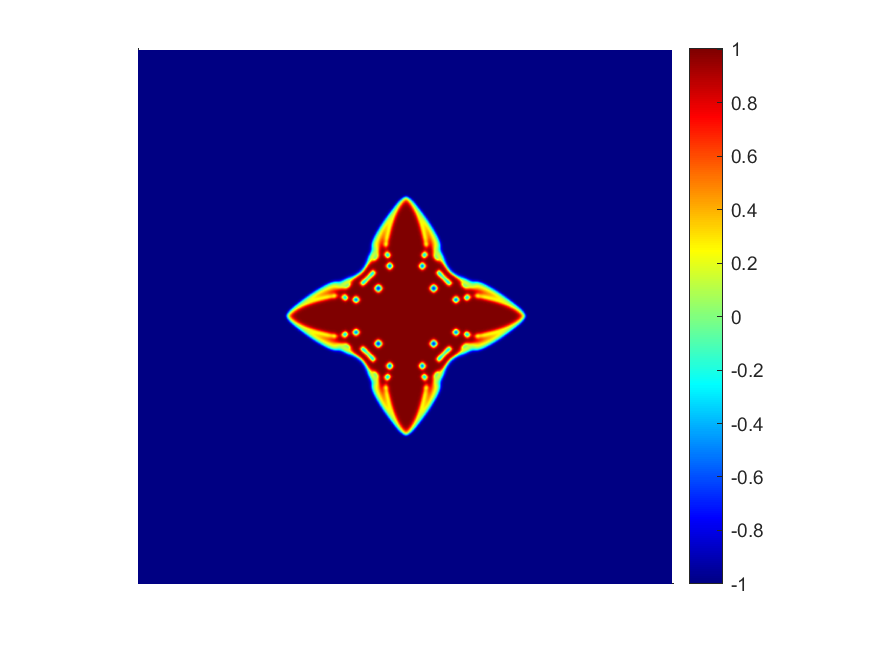}}
	\end{minipage}
	\begin{minipage}[t]{0.24\linewidth}
		\centerline{\includegraphics[scale=0.28]{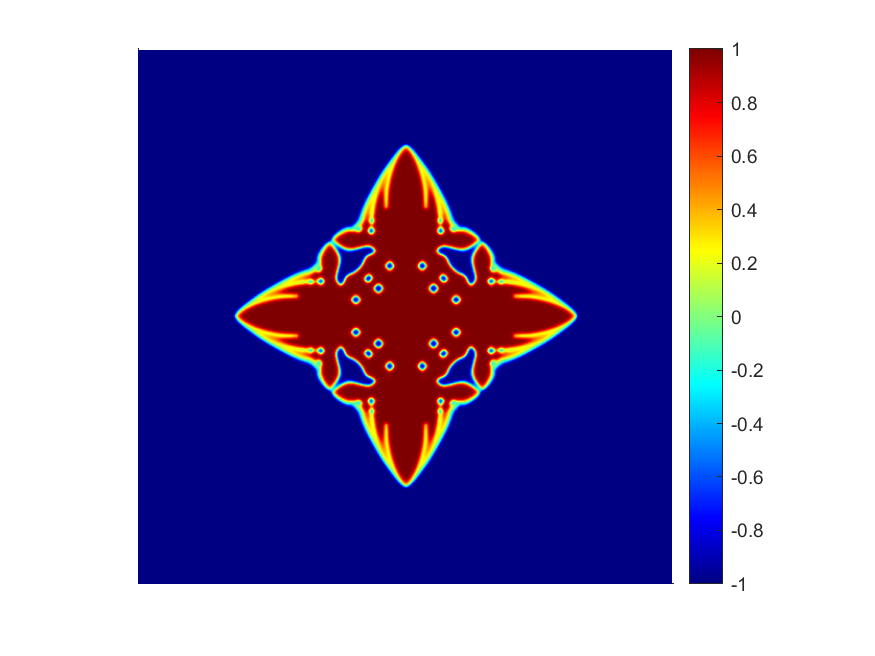}}
	\end{minipage}
	\begin{minipage}[t]{0.24\linewidth}
		\centerline{\includegraphics[scale=0.28]{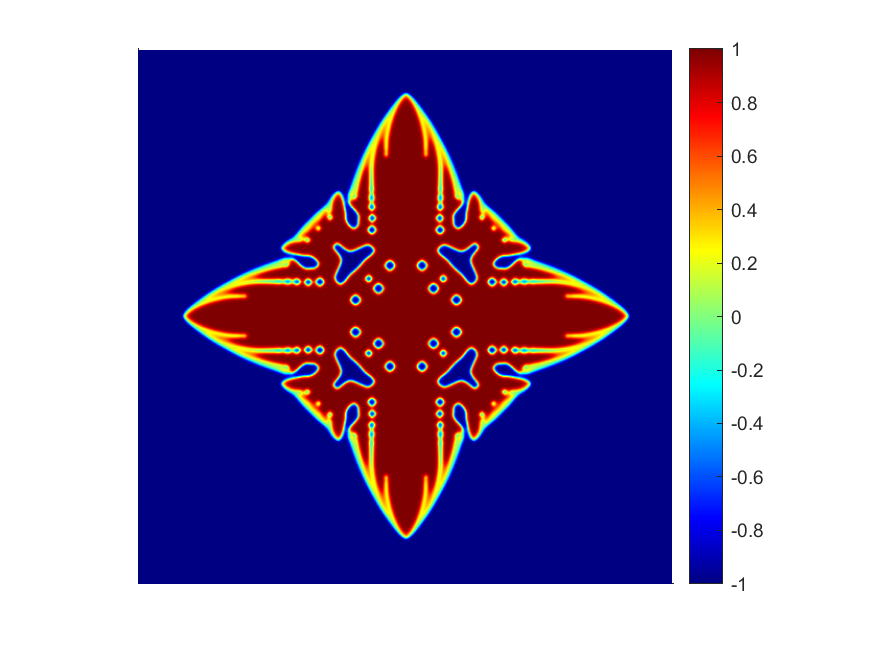}}
	\end{minipage}
	\centerline{(a) Phase field $\phi$}
	\vskip 3mm
	\begin{minipage}[t]{0.24\linewidth}
		\centerline{\includegraphics[scale=0.28]{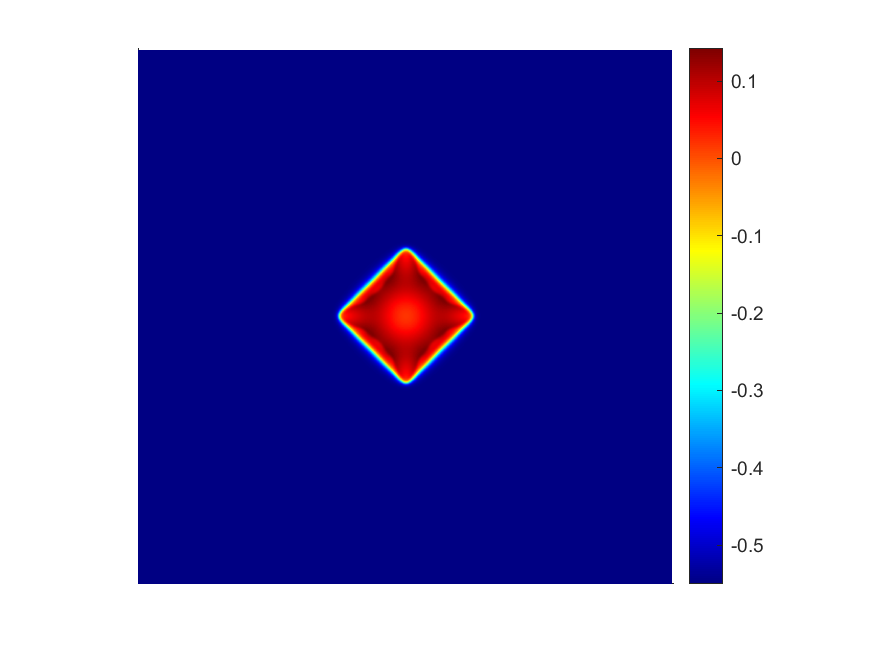}}
	\end{minipage}
	\begin{minipage}[t]{0.24\linewidth}
		\centerline{\includegraphics[scale=0.28]{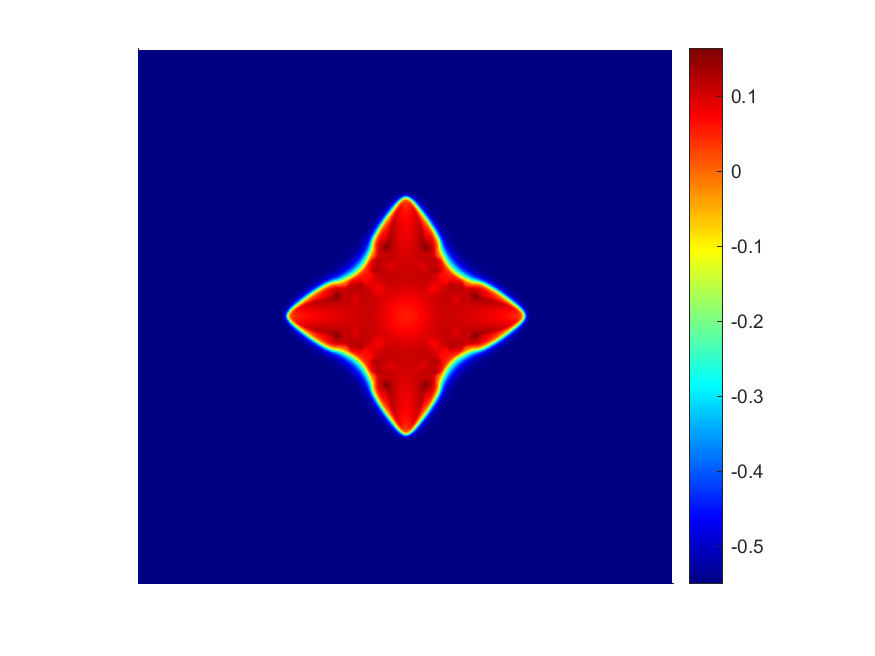}}
	\end{minipage}
	\begin{minipage}[t]{0.24\linewidth}
		\centerline{\includegraphics[scale=0.28]{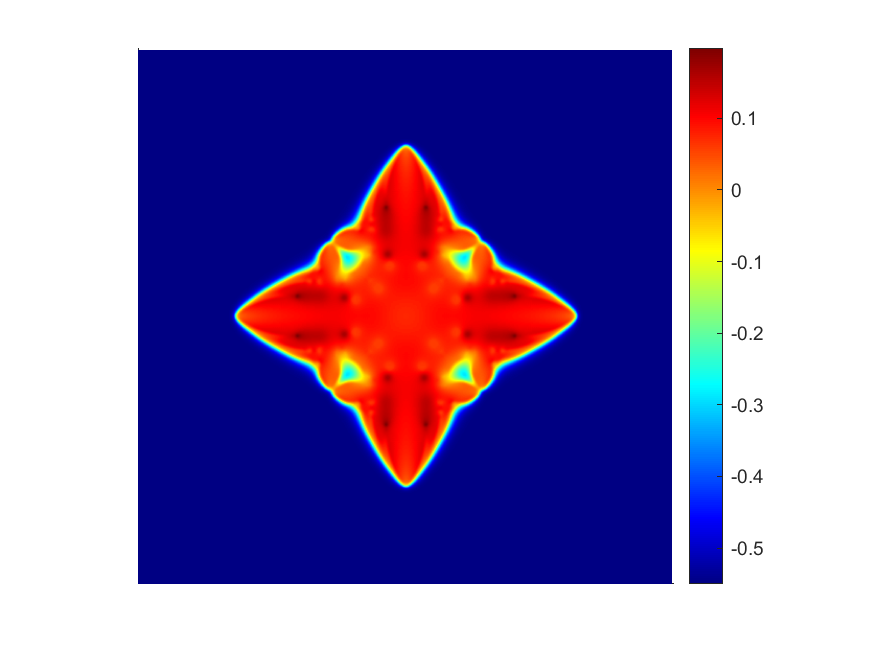}}
	\end{minipage}
	\begin{minipage}[t]{0.24\linewidth}
		\centerline{\includegraphics[scale=0.28]{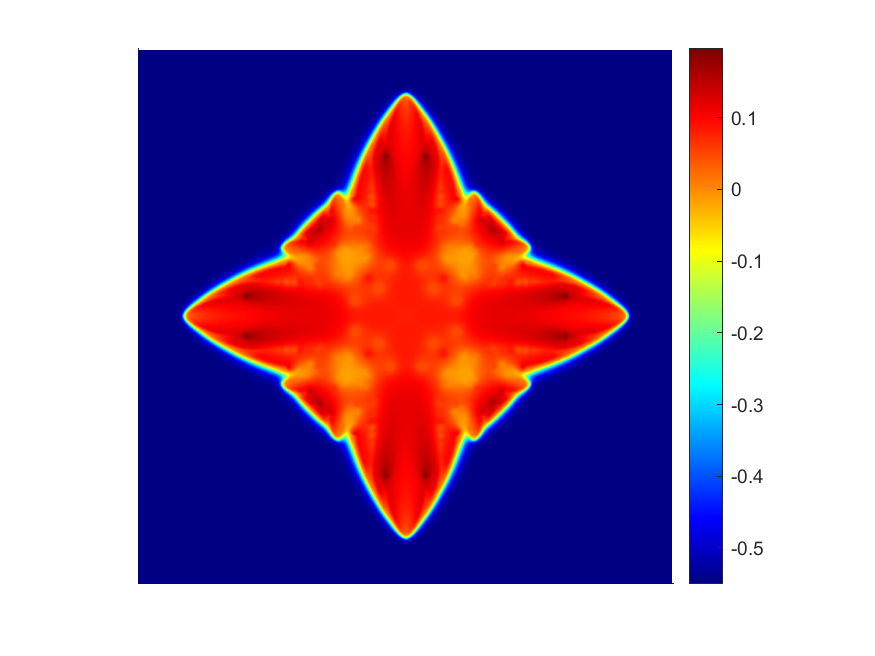}}
	\end{minipage}
	\centerline{(b) Temperature $T$}
	\caption{Evolution of the phase field $\phi$ and temperature $T$ for fourfold anisotropy with $K=0.7$ at $t=5, 10, 15, 20$.
	}\label{fig_example3_K7}
\end{figure*}

\begin{figure*}[htbp]
	\begin{minipage}[t]{0.24\linewidth}
		\centerline{\includegraphics[scale=0.28]{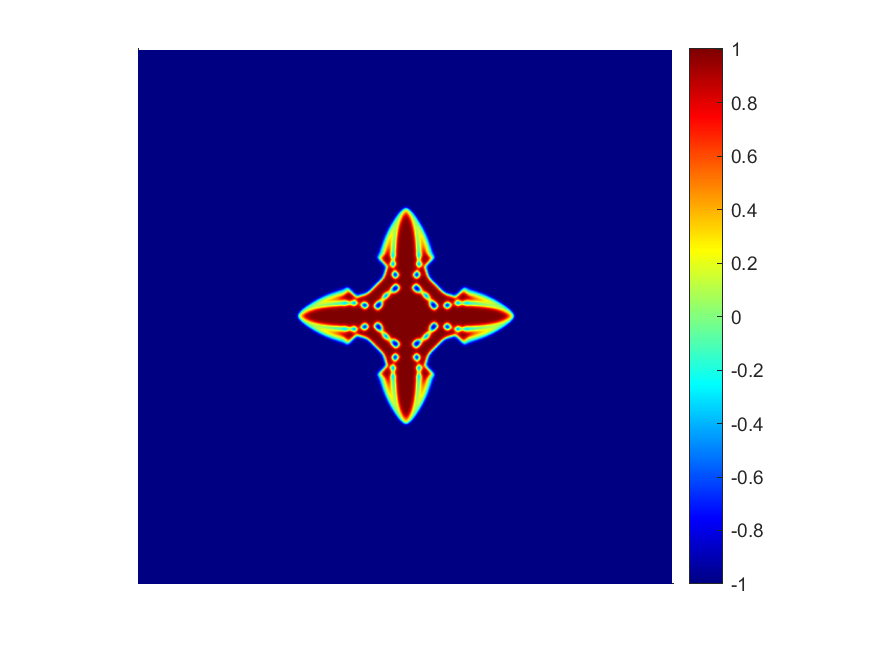}}
	\end{minipage}
	\begin{minipage}[t]{0.24\linewidth}
		\centerline{\includegraphics[scale=0.28]{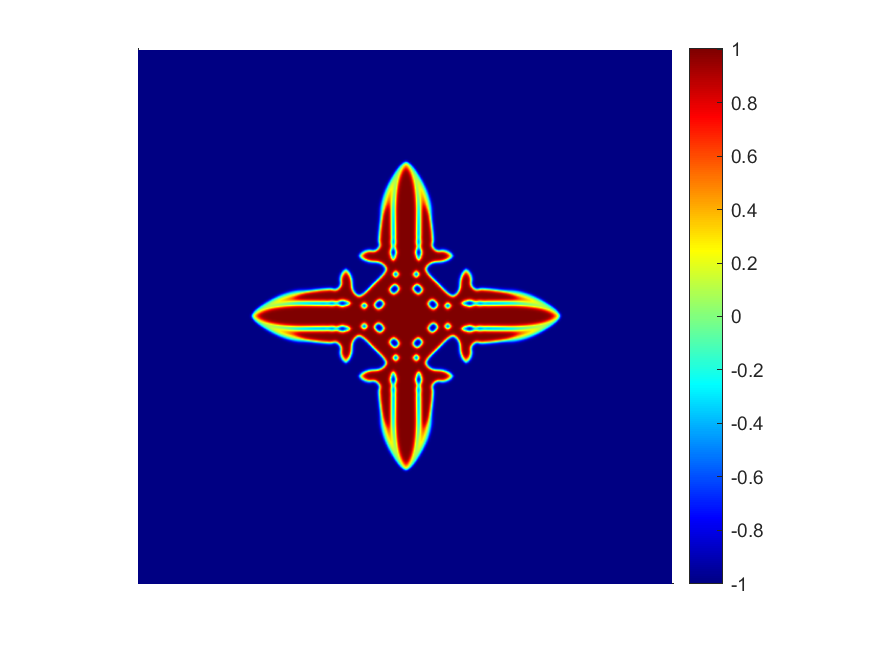}}
	\end{minipage}
	\begin{minipage}[t]{0.24\linewidth}
		\centerline{\includegraphics[scale=0.28]{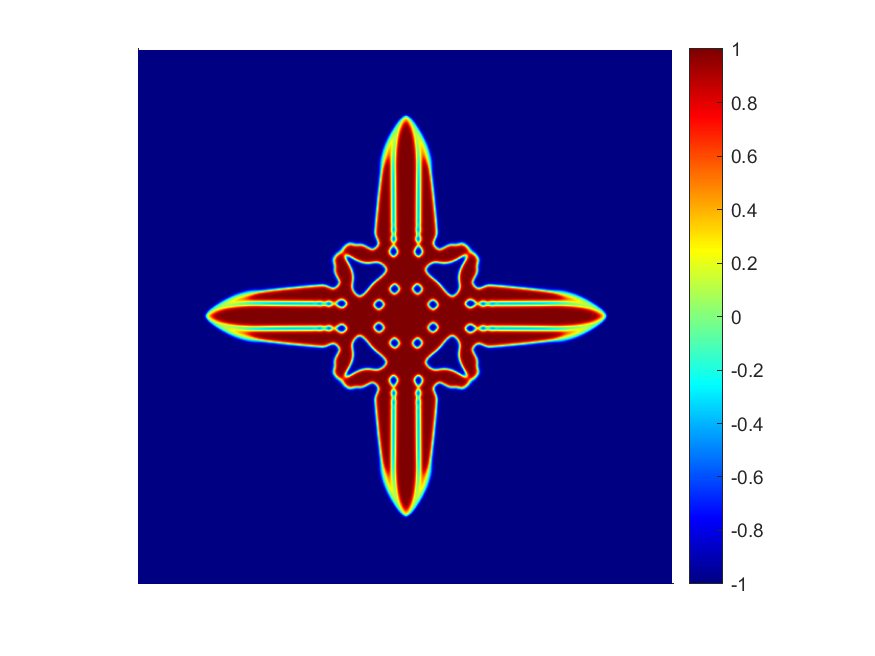}}
	\end{minipage}
	\begin{minipage}[t]{0.24\linewidth}
		\centerline{\includegraphics[scale=0.28]{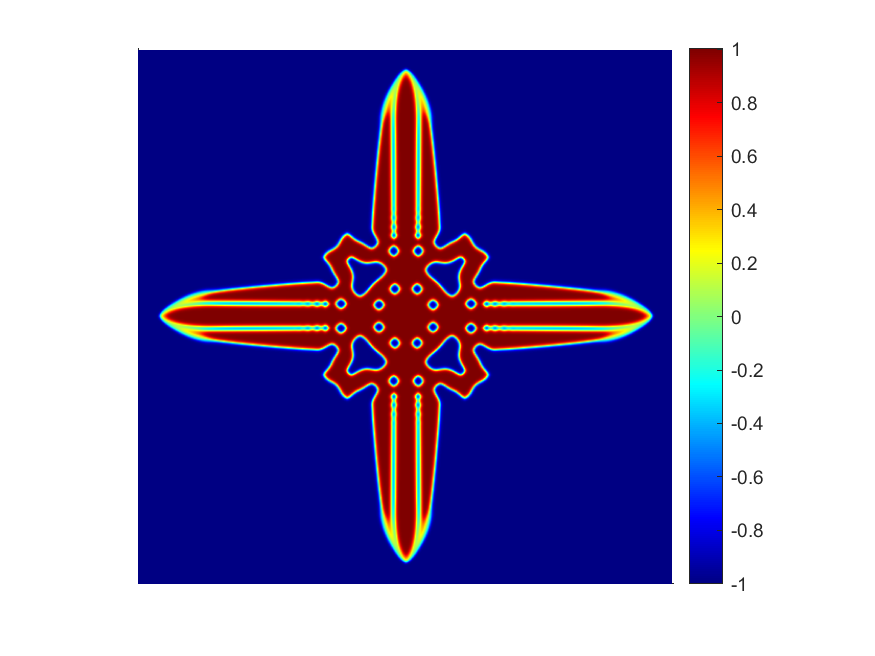}}
	\end{minipage}
	\centerline{(a) Phase field $\phi$}
	\vskip 3mm
	\begin{minipage}[t]{0.24\linewidth}
		\centerline{\includegraphics[scale=0.28]{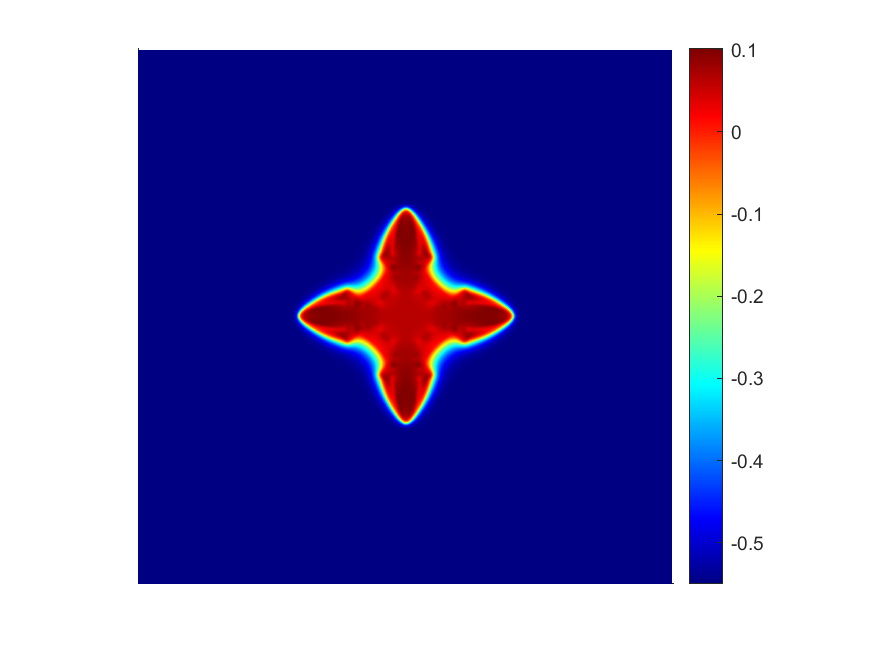}}
	\end{minipage}
	\begin{minipage}[t]{0.24\linewidth}
		\centerline{\includegraphics[scale=0.28]{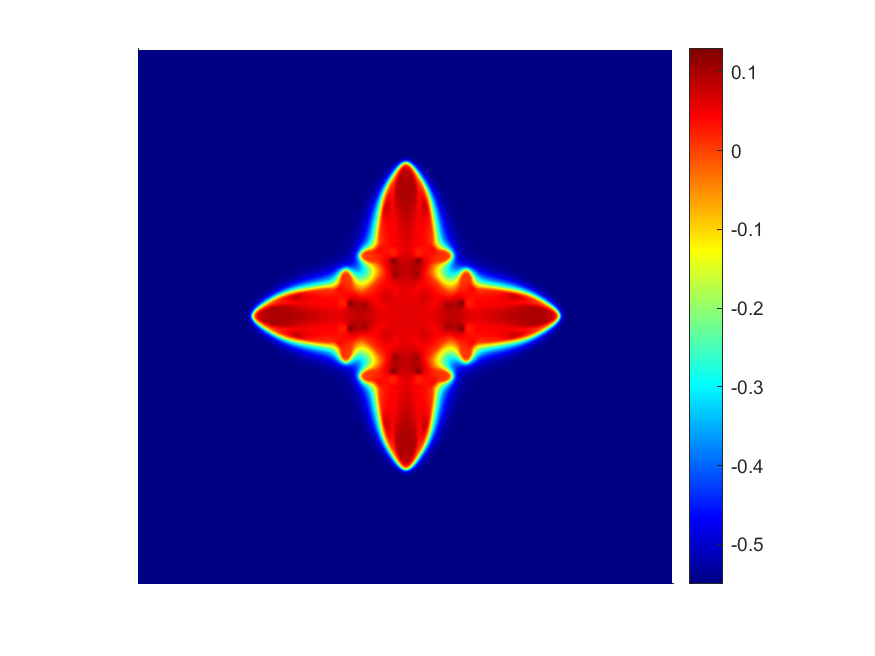}}
	\end{minipage}
	\begin{minipage}[t]{0.24\linewidth}
		\centerline{\includegraphics[scale=0.28]{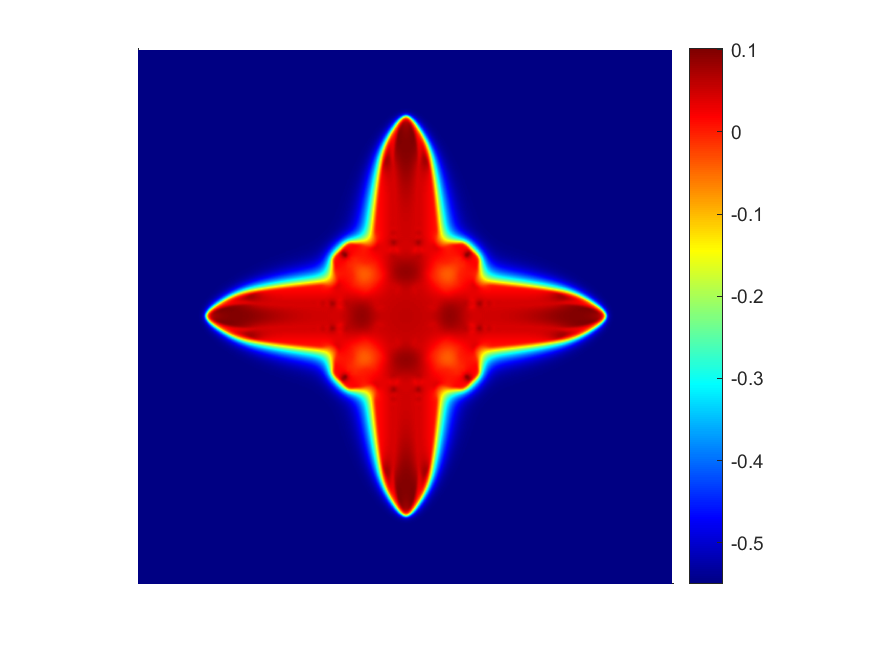}}
	\end{minipage}
	\begin{minipage}[t]{0.24\linewidth}
		\centerline{\includegraphics[scale=0.28]{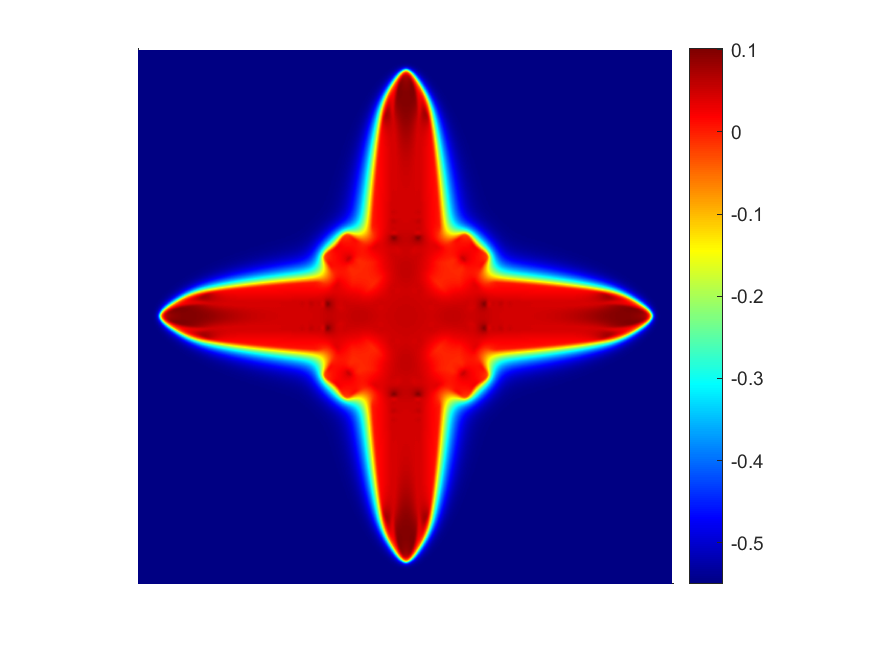}}
	\end{minipage}
	\centerline{(b) Temperature $T$}
	\caption{Evolution of the phase field $\phi$ and temperature $T$ for fourfold anisotropy with $K=0.8$ at $t=10, 15, 20, 25$.
	}\label{fig_example3_K8}
\end{figure*}

\begin{figure*}[htbp]
	\begin{minipage}[t]{0.24\linewidth}
		\centerline{\includegraphics[scale=0.28]{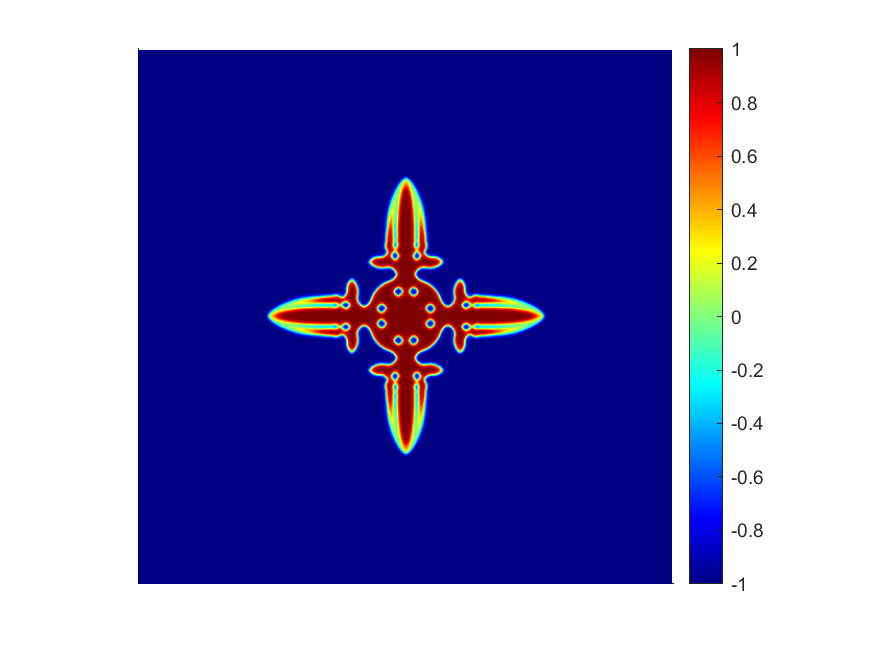}}
	\end{minipage}
	\begin{minipage}[t]{0.24\linewidth}
		\centerline{\includegraphics[scale=0.28]{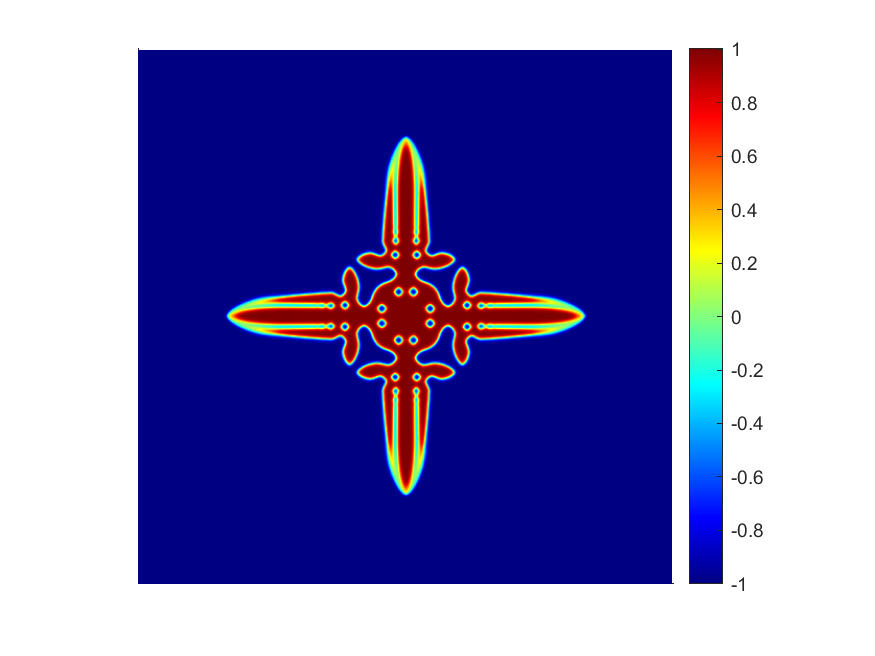}}
	\end{minipage}
	\begin{minipage}[t]{0.24\linewidth}
		\centerline{\includegraphics[scale=0.28]{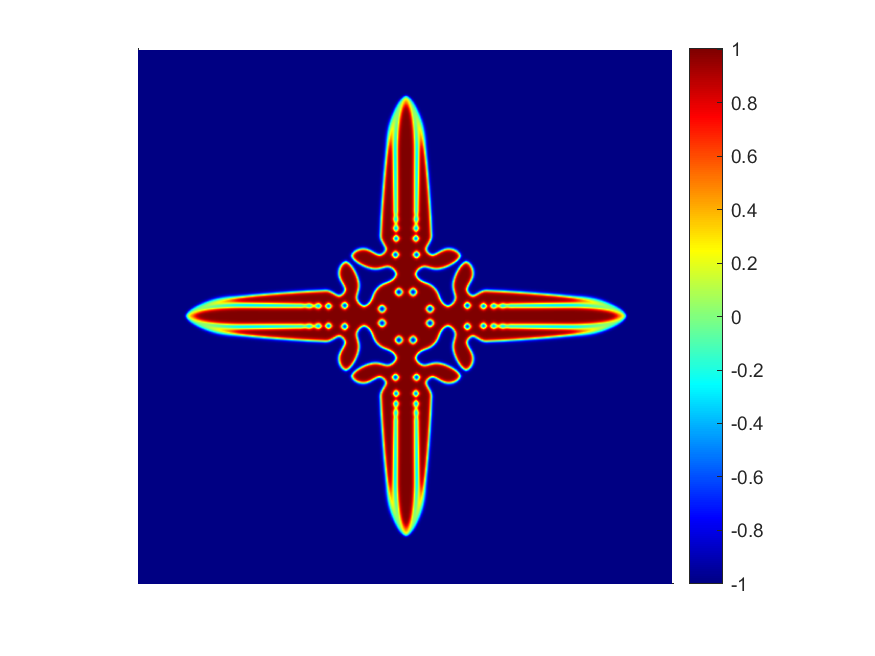}}
	\end{minipage}
	\begin{minipage}[t]{0.24\linewidth}
		\centerline{\includegraphics[scale=0.28]{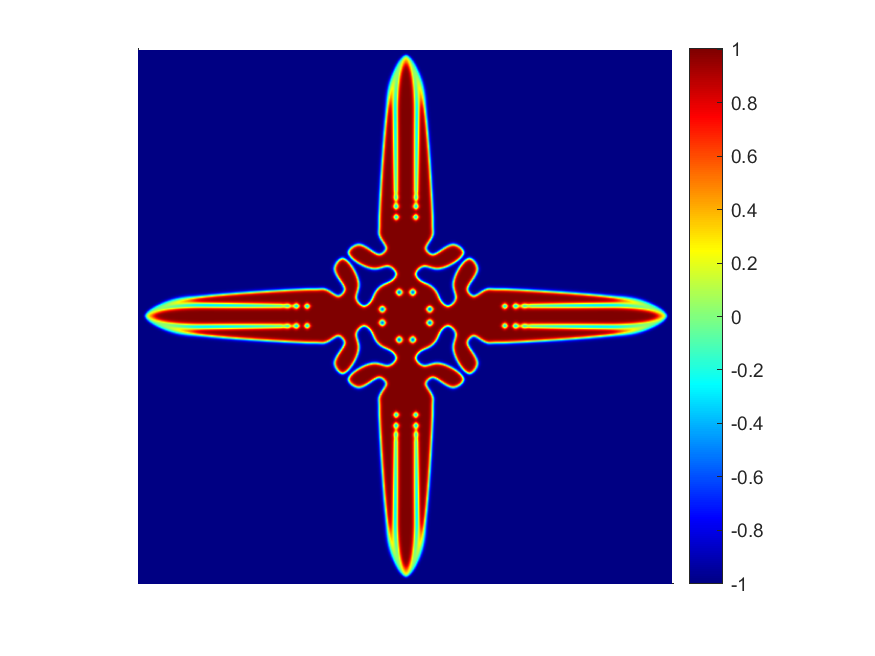}}
	\end{minipage}
	\centerline{(a) Phase field $\phi$}
	\vskip 3mm
	\begin{minipage}[t]{0.24\linewidth}
		\centerline{\includegraphics[scale=0.28]{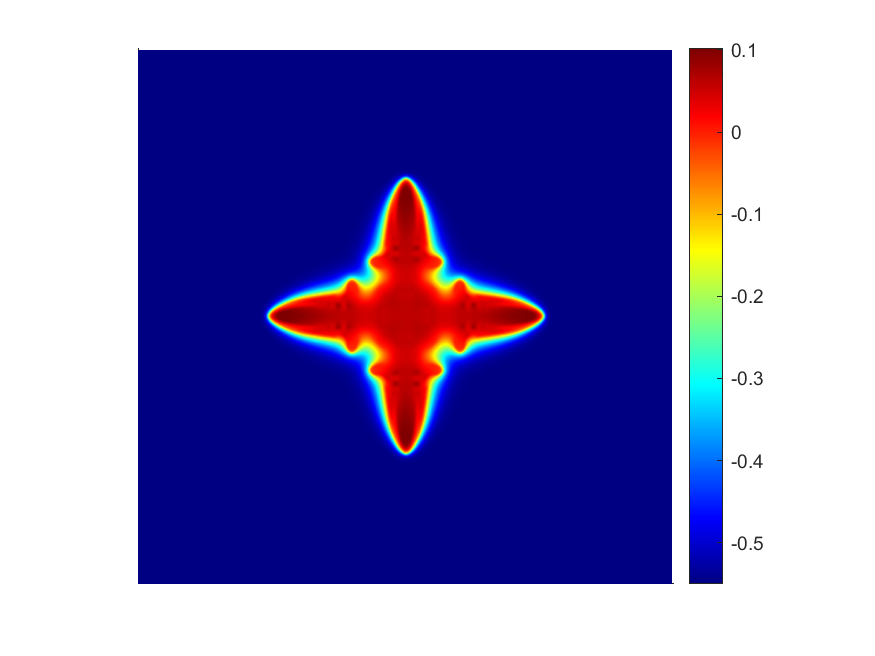}}
	\end{minipage}
	\begin{minipage}[t]{0.24\linewidth}
		\centerline{\includegraphics[scale=0.28]{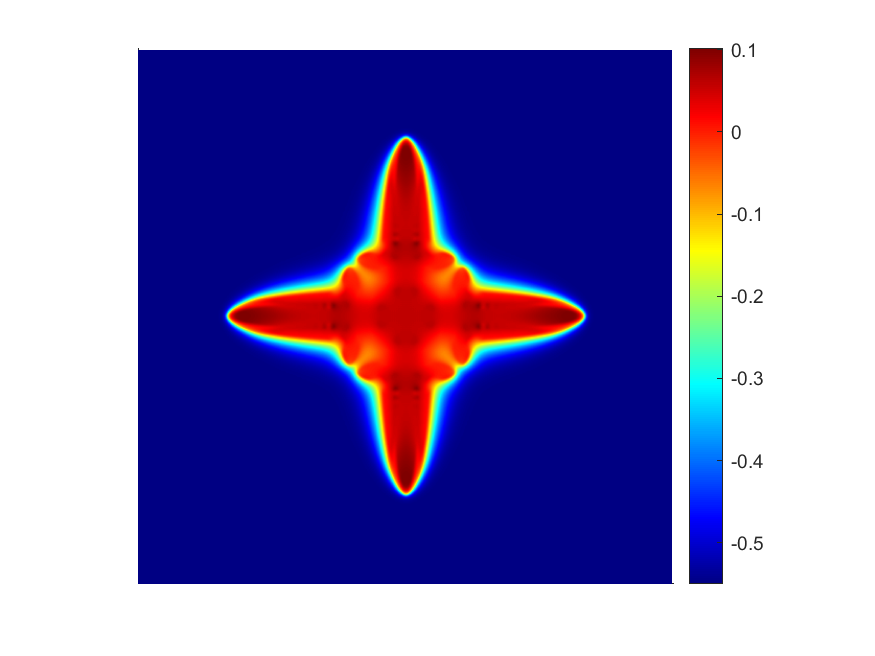}}
	\end{minipage}
	\begin{minipage}[t]{0.24\linewidth}
		\centerline{\includegraphics[scale=0.28]{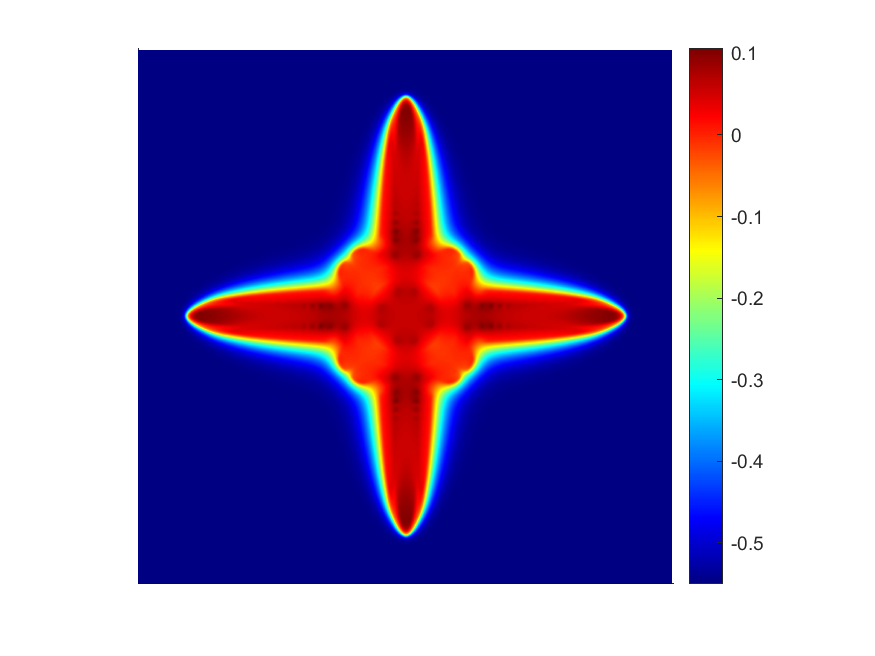}}
	\end{minipage}
	\begin{minipage}[t]{0.24\linewidth}
		\centerline{\includegraphics[scale=0.28]{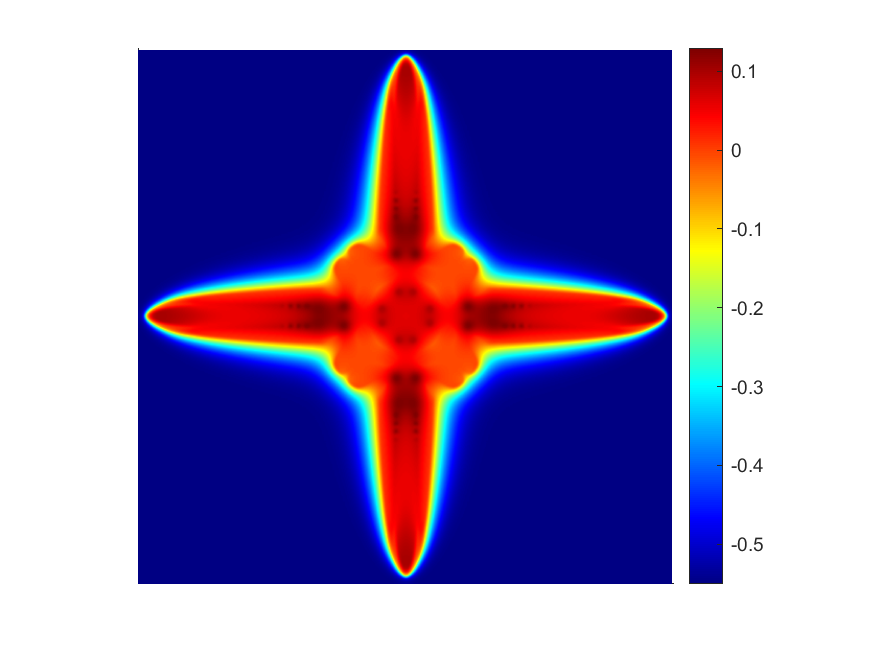}}
	\end{minipage}
	\centerline{(b) Temperature $T$}
	\caption{Evolution of the phase field $\phi$ and temperature $T$ for fourfold anisotropy with $K=0.9$ at $t=15, 20, 25, 30$.
	}\label{fig_example3_K9}
\end{figure*}

\begin{figure*}[htbp]
	\begin{minipage}[t]{0.24\linewidth}
		\centerline{\includegraphics[scale=0.3]{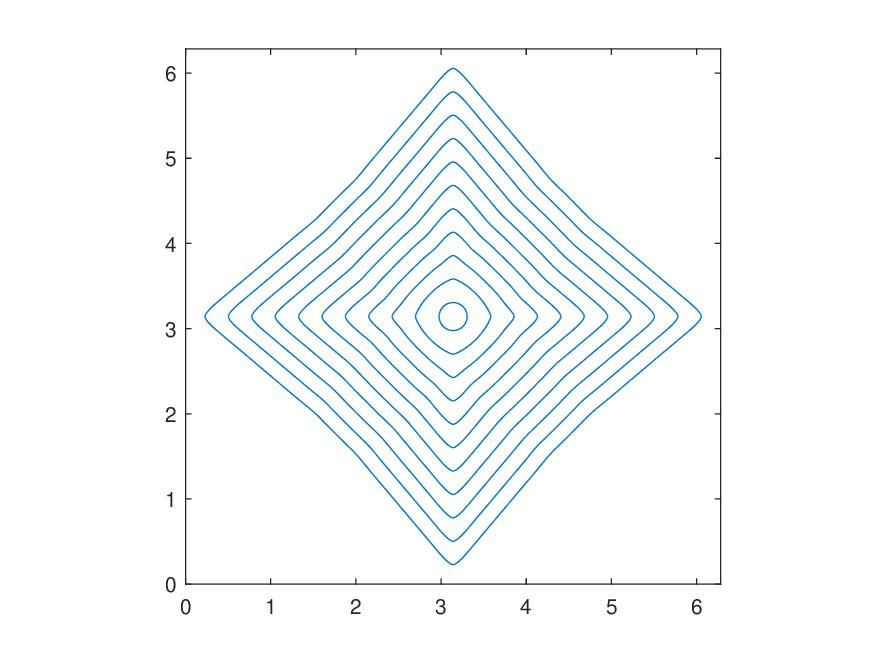}}
		\centerline{(a) $K=0.6$}
	\end{minipage}
	\begin{minipage}[t]{0.24\linewidth}
		\centerline{\includegraphics[scale=0.3]{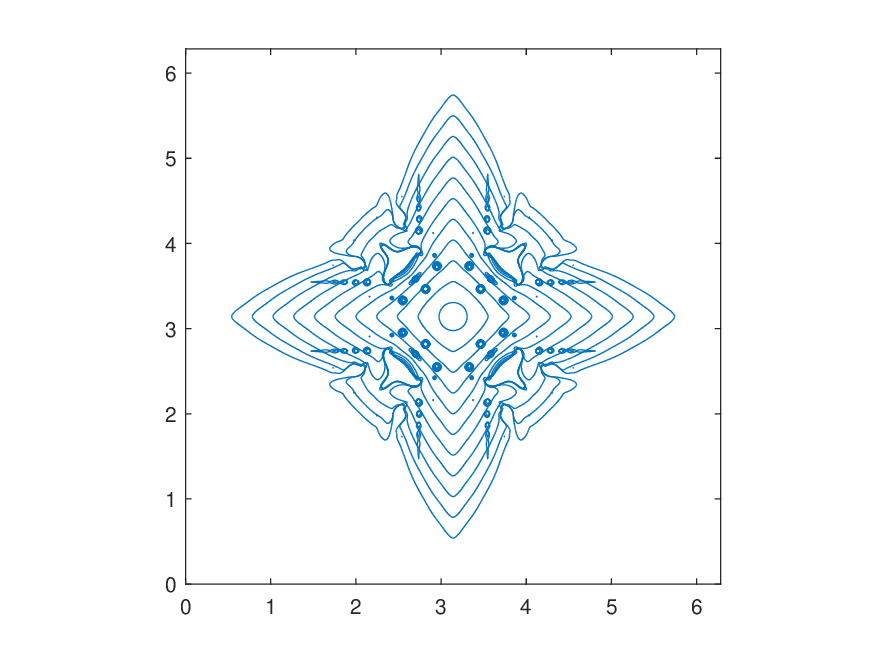}}
		\centerline{(b) $K=0.7$}
	\end{minipage}
	\begin{minipage}[t]{0.24\linewidth}
		\centerline{\includegraphics[scale=0.3]{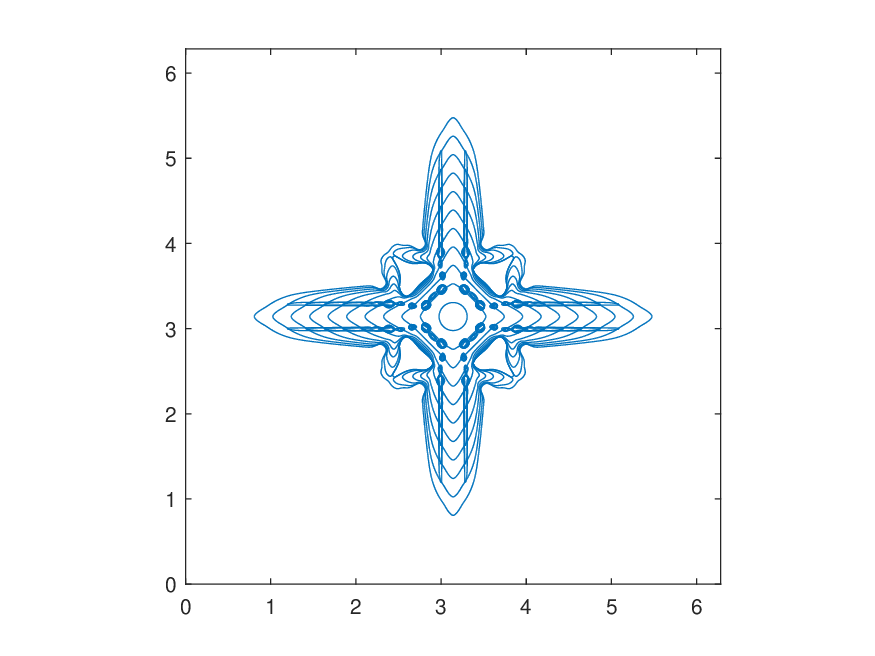}}
		\centerline{(c) $K=0.8$}
	\end{minipage}
	\begin{minipage}[t]{0.24\linewidth}
		\centerline{\includegraphics[scale=0.3]{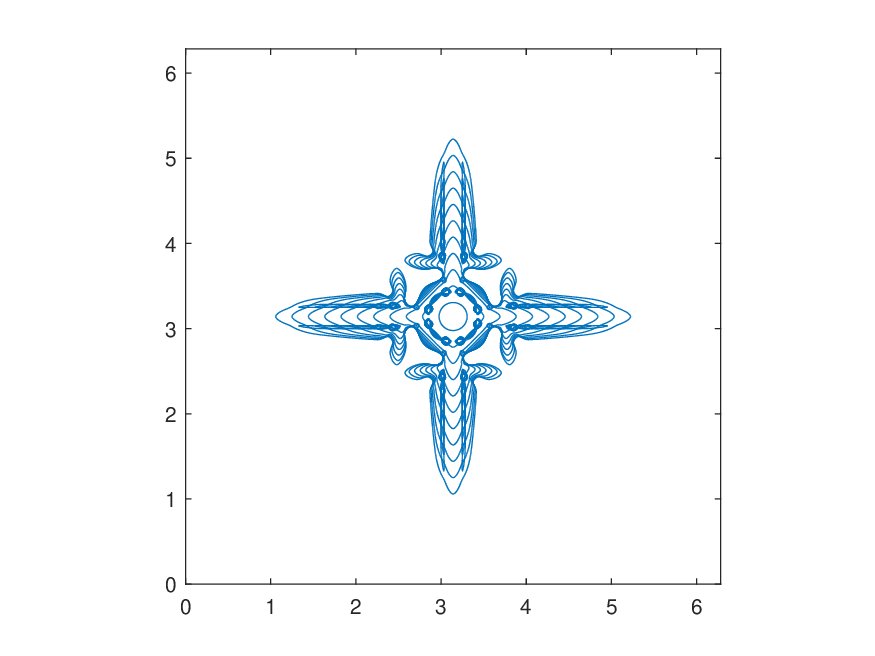}}
		\centerline{(d) $K=0.9$}
	\end{minipage}
	\caption{Evolution of the interface $\phi=0$ for different values of $K$, plotted at two-unit time intervals up to $t=20$.
	}\label{fig_example3_contour}
\end{figure*}

\begin{figure*}[htbp]
	\begin{minipage}[t]{0.49\linewidth}
		\centerline{\includegraphics[scale=0.5]{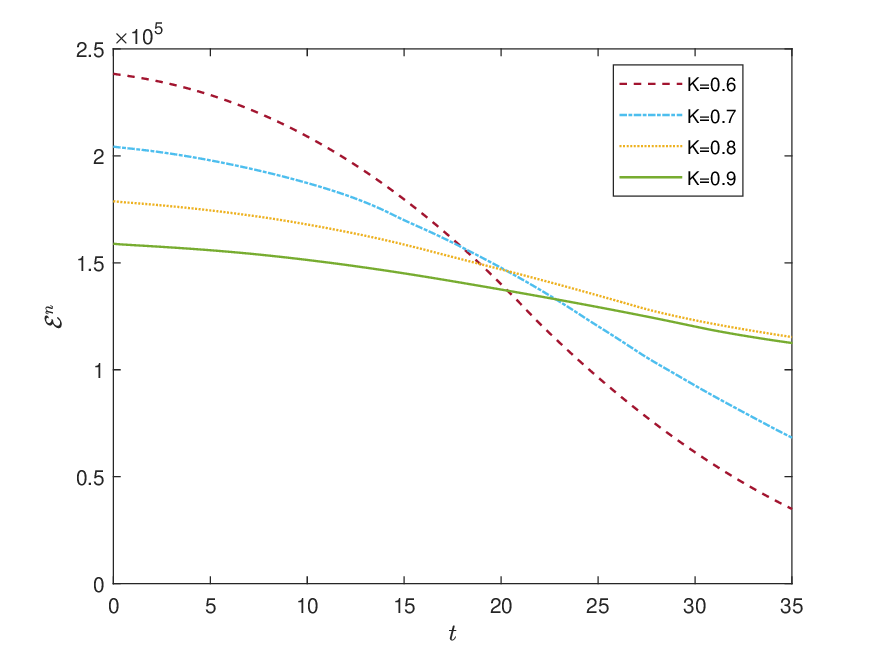}}
		\centerline{(a) Modified energy $\mathcal{E}_{n}$}
	\end{minipage}
	\begin{minipage}[t]{0.49\linewidth}
		\centerline{\includegraphics[scale=0.5]{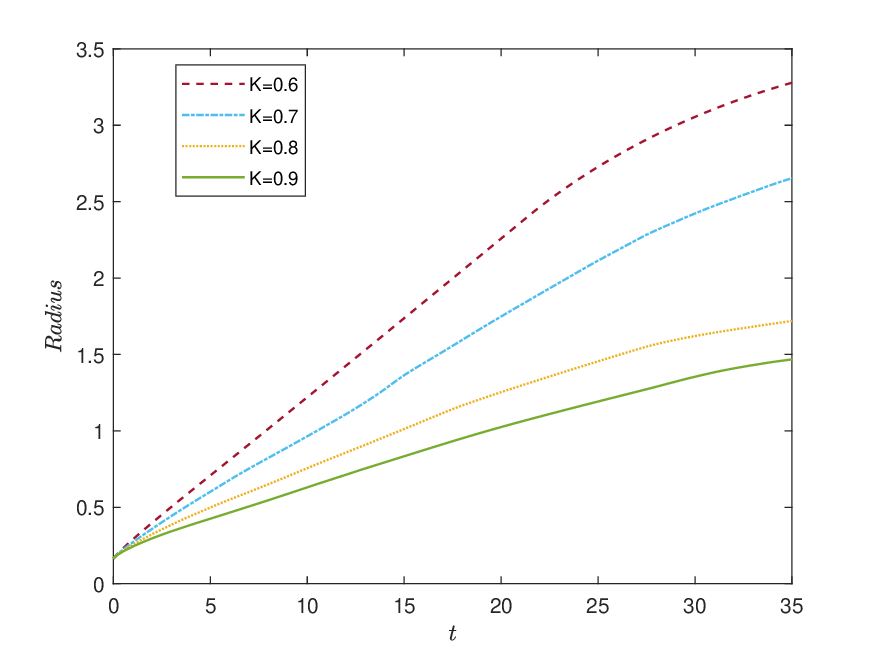}}
		\centerline{(b) Equivalent crystal radius}
	\end{minipage}
	\caption{Evolution of the modified energy $\mathcal E_n$ and equivalent crystal radius for two-dimensional fourfold anisotropy with different values of the latent heat coefficient $K$.
	}\label{fig_example3_1}
\end{figure*}

\begin{figure*}[htbp]
	\begin{minipage}[t]{0.24\linewidth}
		\centerline{\includegraphics[scale=0.28]{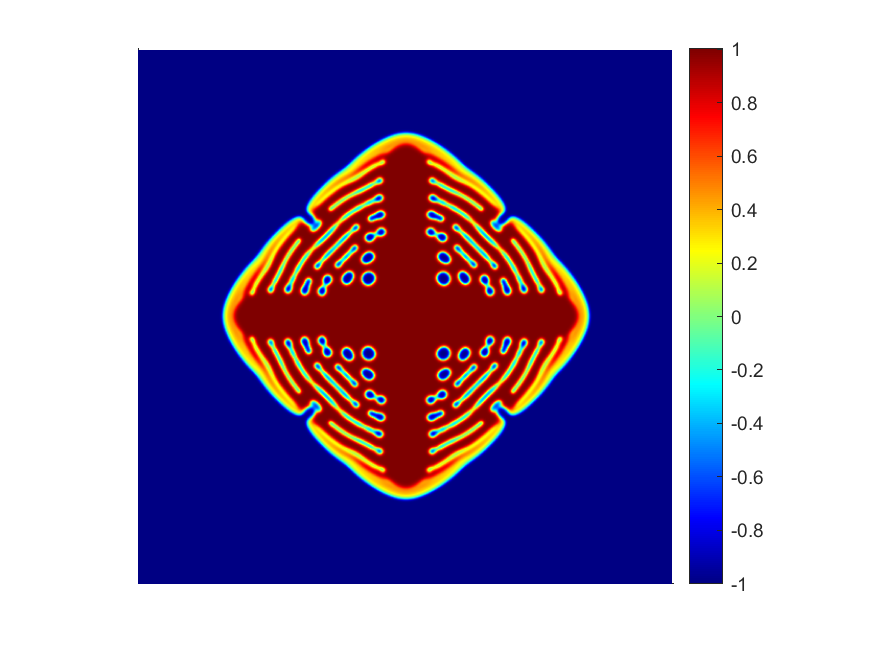}}
	\end{minipage}
	\begin{minipage}[t]{0.24\linewidth}
		\centerline{\includegraphics[scale=0.28]{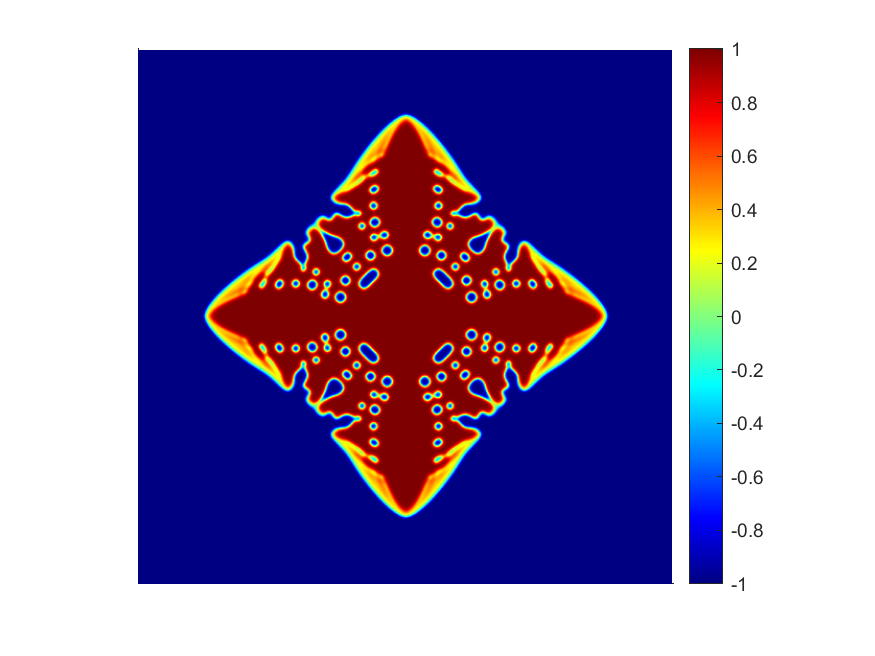}}
	\end{minipage}
	\begin{minipage}[t]{0.24\linewidth}
		\centerline{\includegraphics[scale=0.28]{phi20K7.eps}}
	\end{minipage}
	\begin{minipage}[t]{0.24\linewidth}
		\centerline{\includegraphics[scale=0.28]{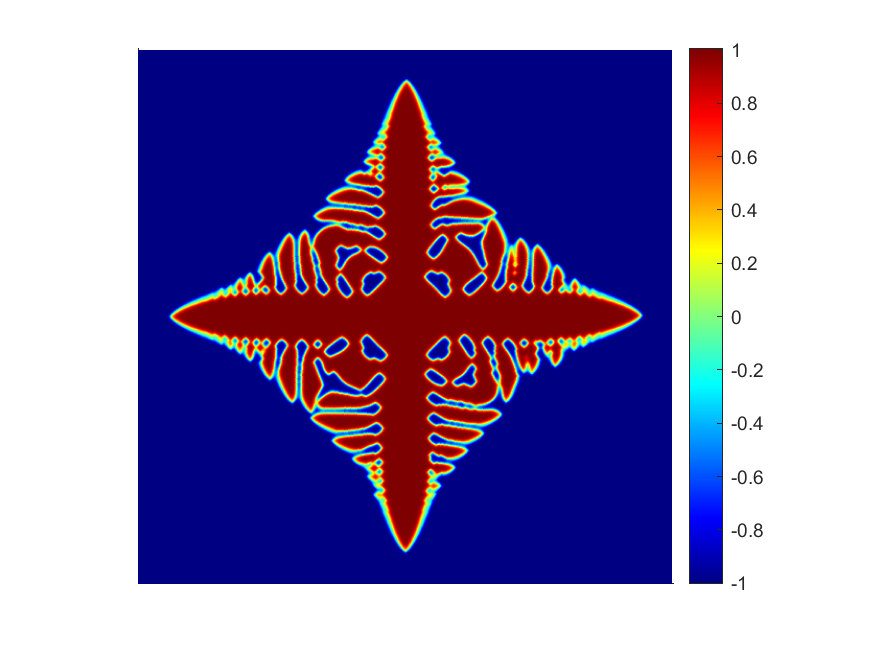}}
	\end{minipage}
	\centerline{(a) Phase field $\phi$ at $t=20$}
	\vskip 3mm
	\begin{minipage}[t]{0.24\linewidth}
		\centerline{\includegraphics[scale=0.3]{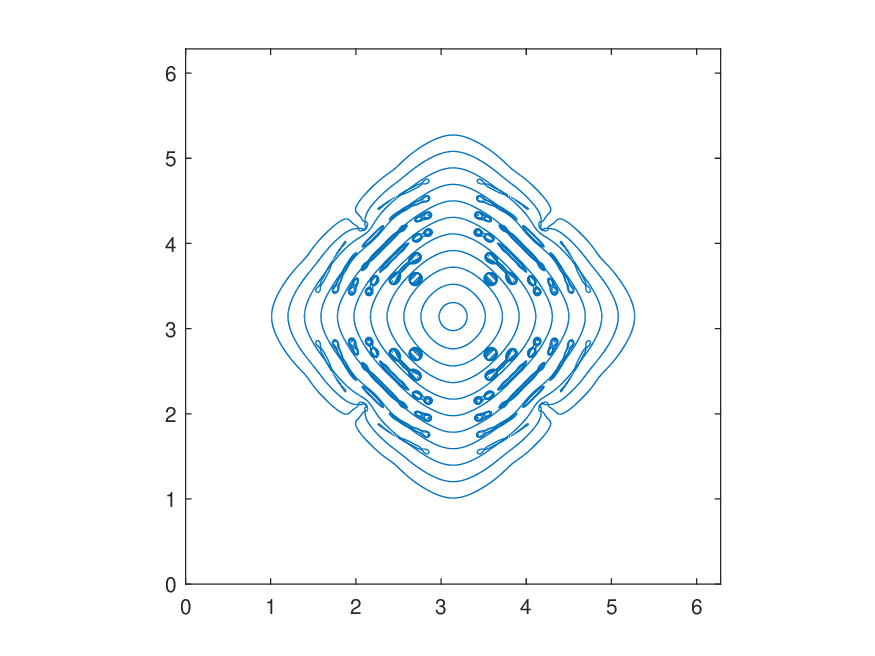}}
	\end{minipage}
	\begin{minipage}[t]{0.24\linewidth}
		\centerline{\includegraphics[scale=0.3]{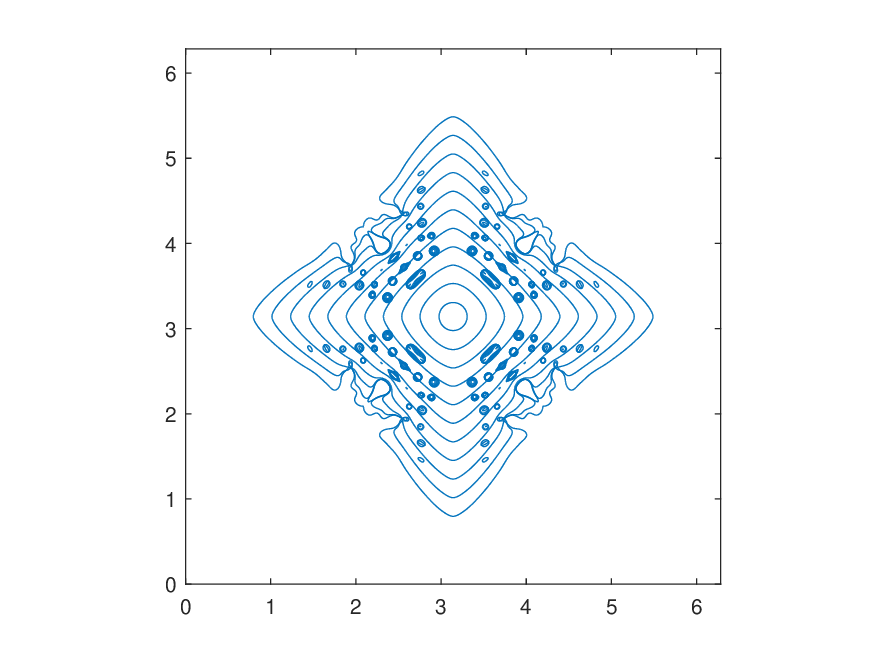}}
	\end{minipage}
	\begin{minipage}[t]{0.24\linewidth}
		\centerline{\includegraphics[scale=0.3]{contourK7.eps}}
	\end{minipage}
	\begin{minipage}[t]{0.24\linewidth}
		\centerline{\includegraphics[scale=0.3]{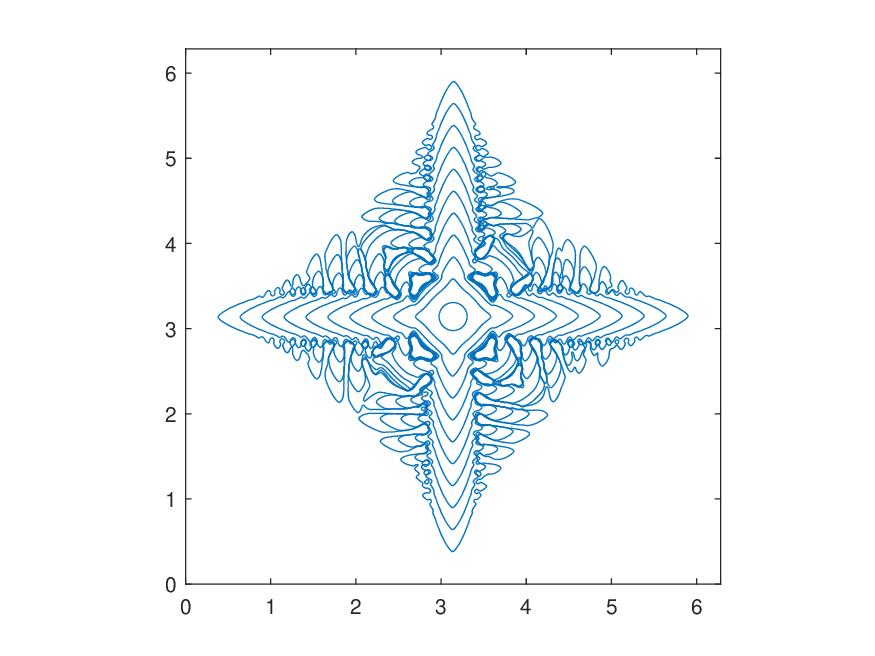}}
	\end{minipage}
	\centerline{(b) Interface $\phi=0$ at two-unit time intervals up to $t=20$}
	\caption{Dendritic morphologies with fourfold anisotropy for different values of $\epsilon_4$. From left to right, $\epsilon_4=0.01$, $0.02$, $0.04$, and $0.08$.
	}\label{fig_example3_e4}
\end{figure*}

\begin{figure*}[htbp]
	\centering
	\begin{minipage}[t]{0.45\linewidth}
		\centerline{\includegraphics[scale=0.45]{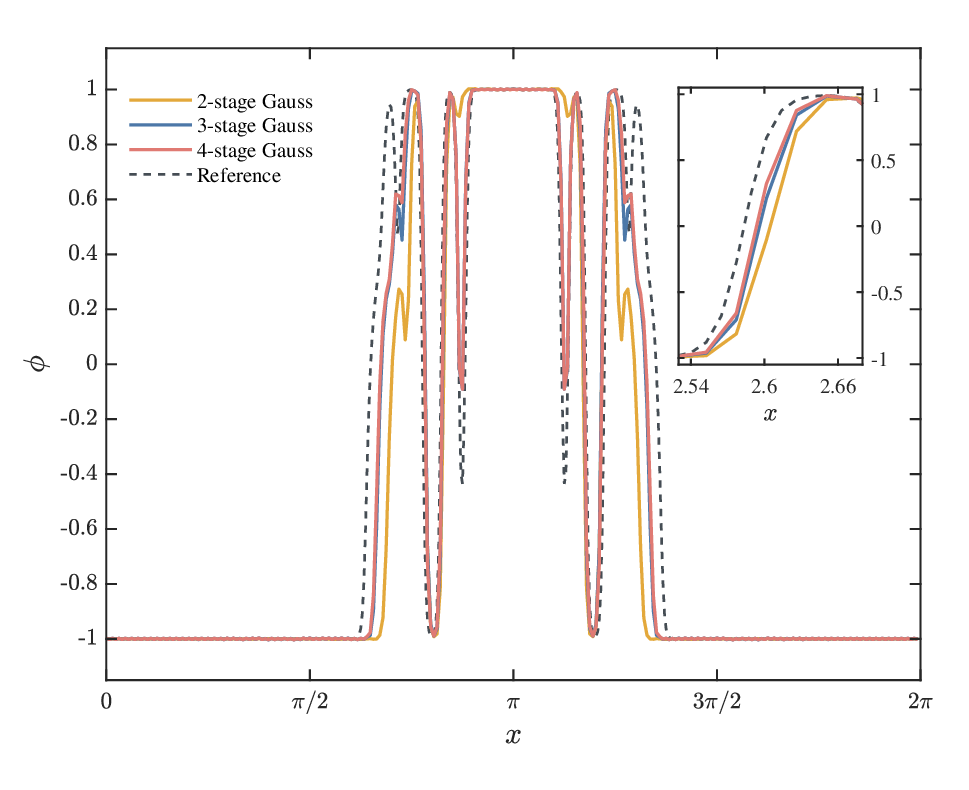}}
		\centerline{(a) Phase field $\phi$}
	\end{minipage}
	\begin{minipage}[t]{0.45\linewidth}
		\centerline{\includegraphics[scale=0.45]{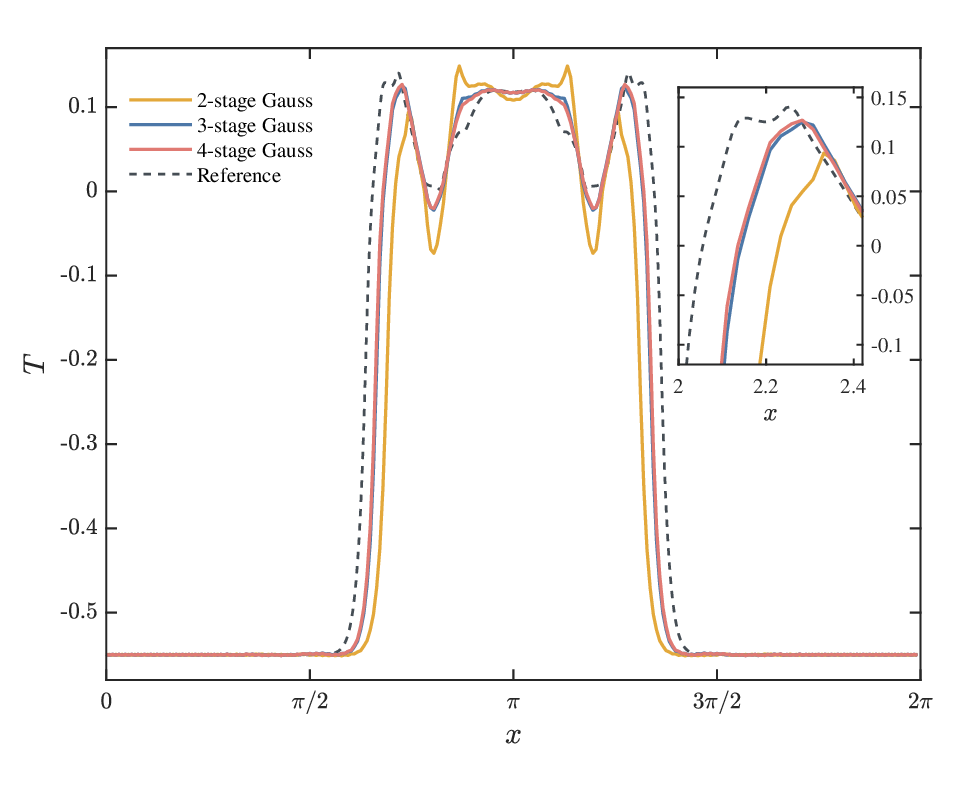}}
		\centerline{(b) Temperature $T$}
	\end{minipage}
	\caption{Comparison of the phase-field and temperature profiles along the cross-section $y=2.0126$ at $t=20$ for the $2$-, $3$-, and $4$-stage Gauss schemes.
	}\label{fig_example3_compare}
\end{figure*}
 
We then consider the sixfold anisotropy case with $m=6$. For comparison with the fourfold case, we use the same initial conditions and parameters in \eqref{ex3_initial}-\eqref{ex3_params}, and set $\epsilon_4=0.04$ and $\theta_0=90^\circ$. Figure~\ref{fig_example4_K} shows the evolution of the phase field $\phi$ for $K=0.6$, $0.7$, $0.8$, and $0.9$, together with the corresponding temperature fields at the last snapshot of each case. Starting from the same circular nucleus, the crystal develops into a sixfold dendritic structure. As $K$ increases, the primary branches become thinner and the tips become sharper, while secondary branches appear along them, producing a more complex snowflake-like morphology. The temperature profiles exhibit a similar sixfold pattern and are consistent with the phase transition behavior.

The influence of $K$ on the energy dissipation and crystal growth is further examined in Figure~\ref{fig_example4}. The modified energy decreases monotonically in all cases. Larger values of $K$ lead to a slower energy decay and a lower growth rate of the equivalent crystal radius. This behavior is consistent with that observed in the fourfold anisotropy case.

We also vary the anisotropy order $m$ from $4$ to $7$ while fixing $\epsilon_4=0.04$, $\theta_0=0^\circ$, and $K=0.8$. The corresponding phase-field and temperature distributions at $t=20$ are compared in Figure~\ref{fig_example4_m}. The number of primary dendritic branches grows with $m$, while the temperature field exhibits a symmetry consistent with the corresponding phase-field morphology.

\begin{figure*}[htbp]
	\begin{minipage}[t]{0.24\linewidth}
		\centerline{\includegraphics[scale=0.28]{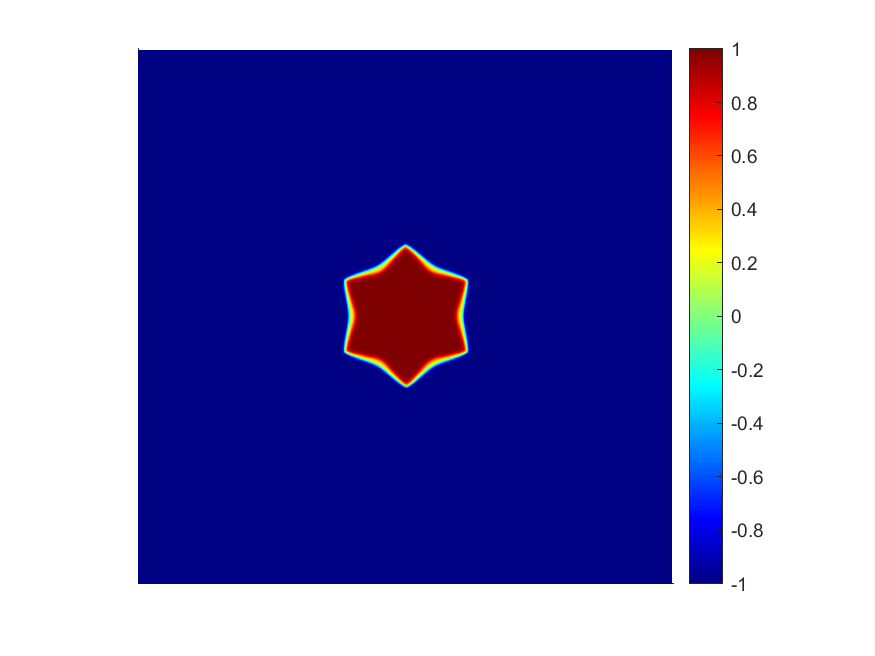}}
	\end{minipage}
	\begin{minipage}[t]{0.24\linewidth}
		\centerline{\includegraphics[scale=0.28]{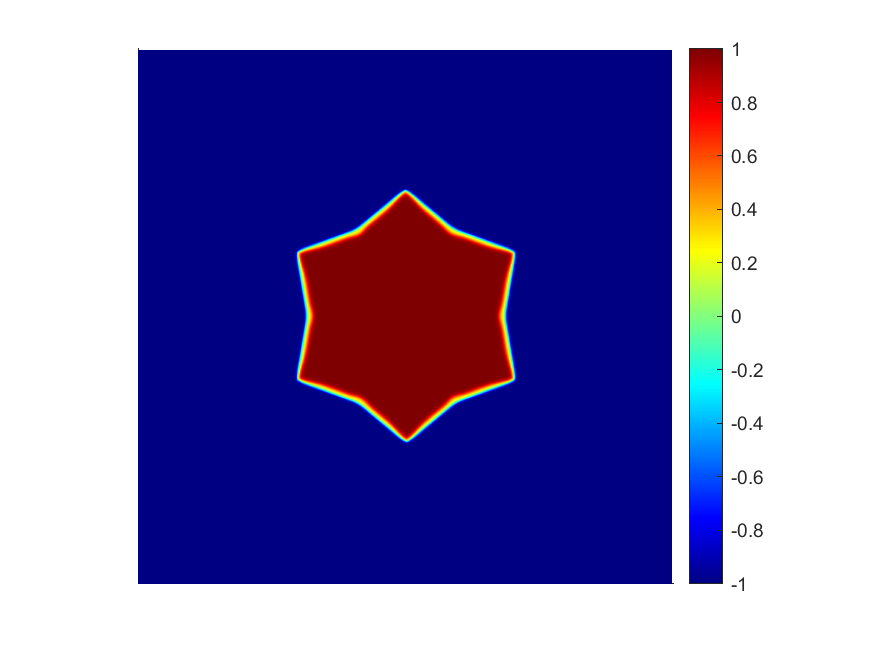}}
	\end{minipage}
	\begin{minipage}[t]{0.24\linewidth}
		\centerline{\includegraphics[scale=0.28]{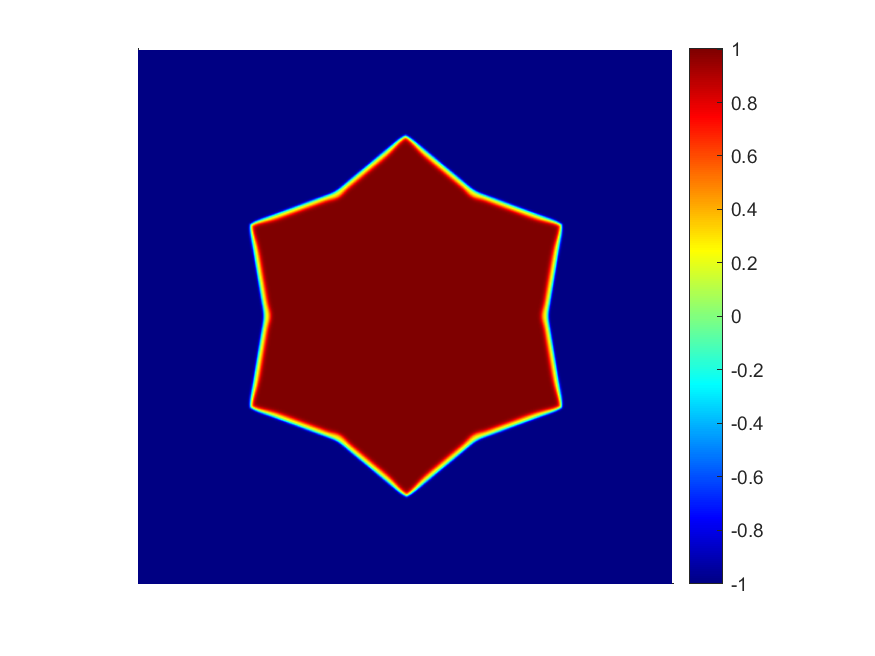}}
	\end{minipage}
	\begin{minipage}[t]{0.24\linewidth}
		\centerline{\includegraphics[scale=0.28]{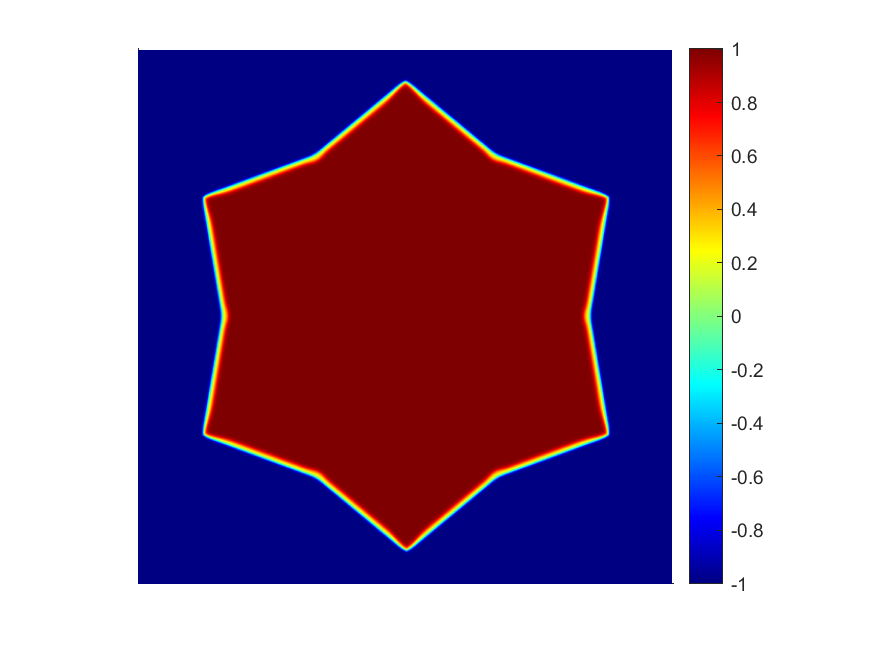}}
	\end{minipage}
	\centerline{(a) Phase field $\phi$ for $K=0.6$ at $t=5, 10, 15, 20$}
	\vskip 3mm	
	\begin{minipage}[t]{0.24\linewidth}
		\centerline{\includegraphics[scale=0.28]{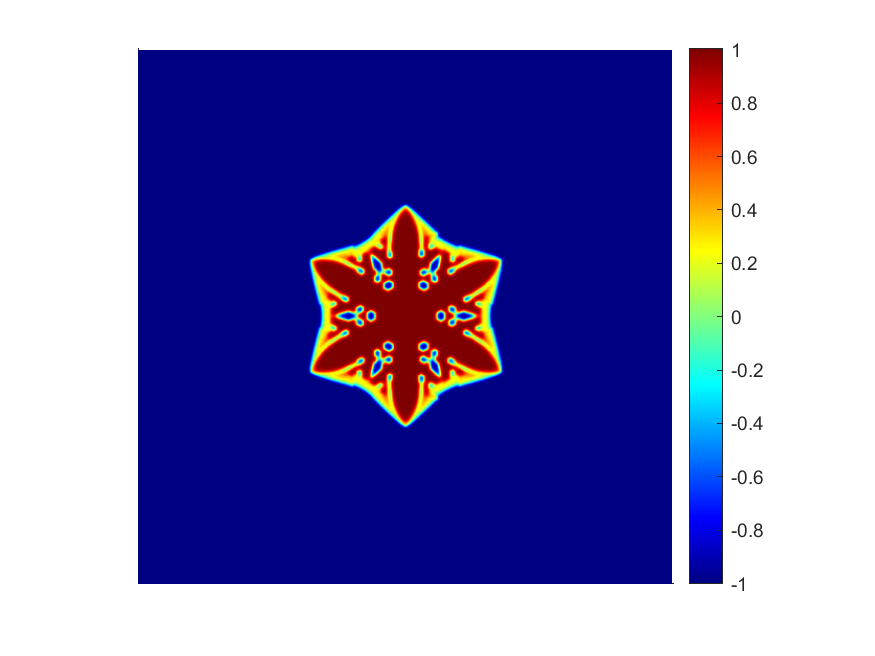}}
	\end{minipage}
	\begin{minipage}[t]{0.24\linewidth}
		\centerline{\includegraphics[scale=0.28]{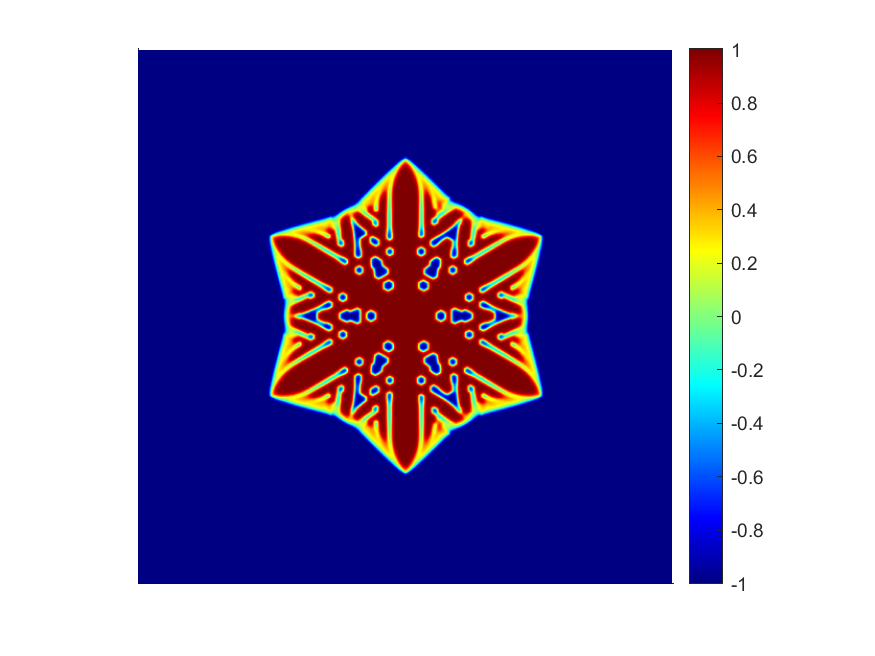}}
	\end{minipage}
	\begin{minipage}[t]{0.24\linewidth}
		\centerline{\includegraphics[scale=0.28]{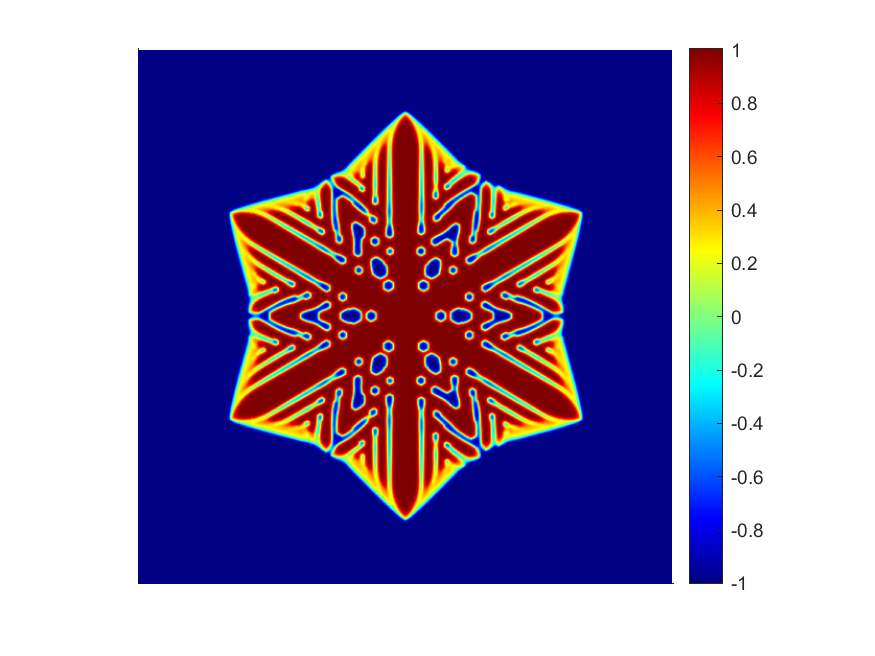}}
	\end{minipage}
	\begin{minipage}[t]{0.24\linewidth}
		\centerline{\includegraphics[scale=0.28]{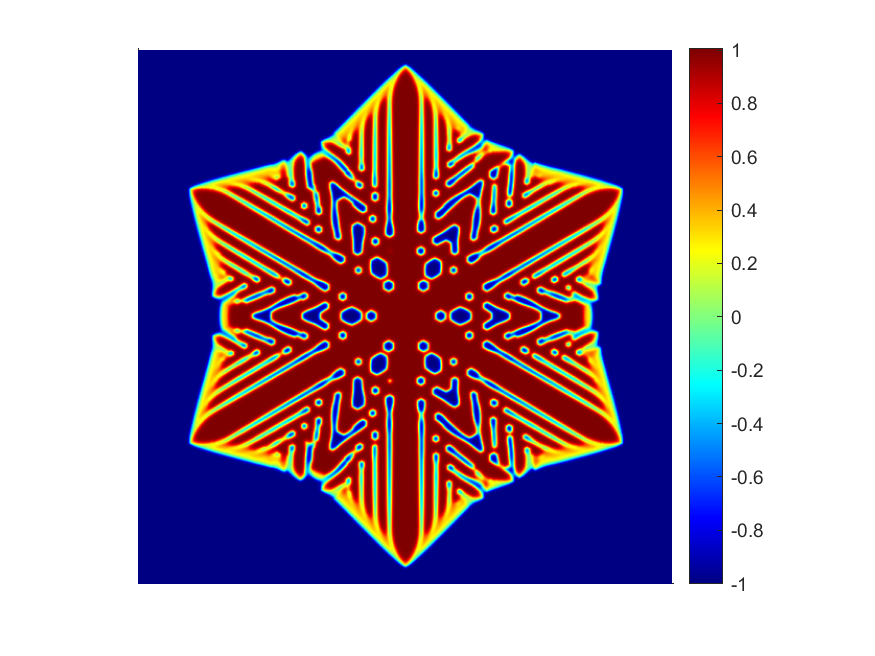}}
	\end{minipage}
	\centerline{(b) Phase field $\phi$ for $K=0.7$ at $t=10, 15, 20, 25$}
	\vskip 3mm	
	\begin{minipage}[t]{0.24\linewidth}
		\centerline{\includegraphics[scale=0.28]{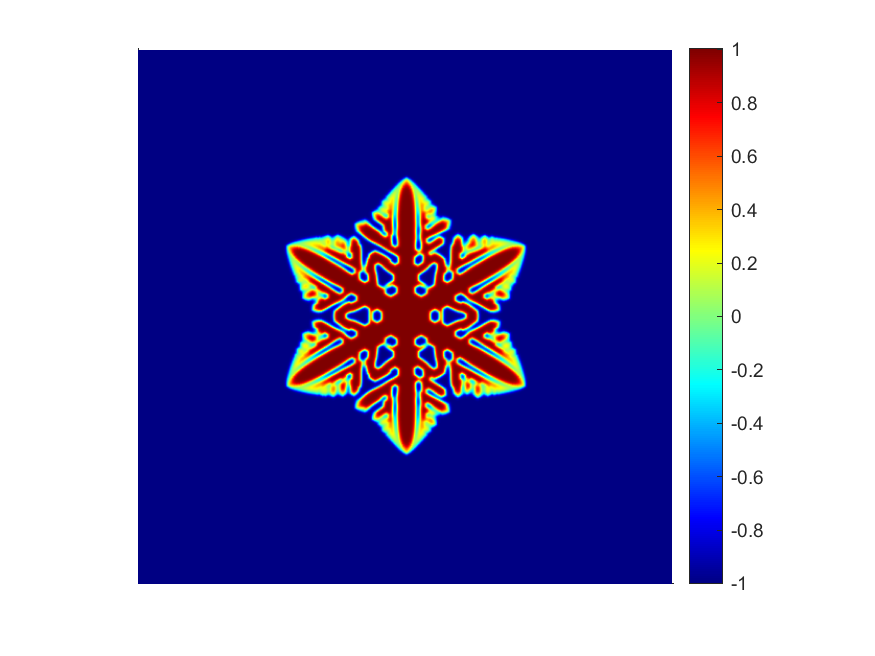}}
	\end{minipage}
	\begin{minipage}[t]{0.24\linewidth}
		\centerline{\includegraphics[scale=0.28]{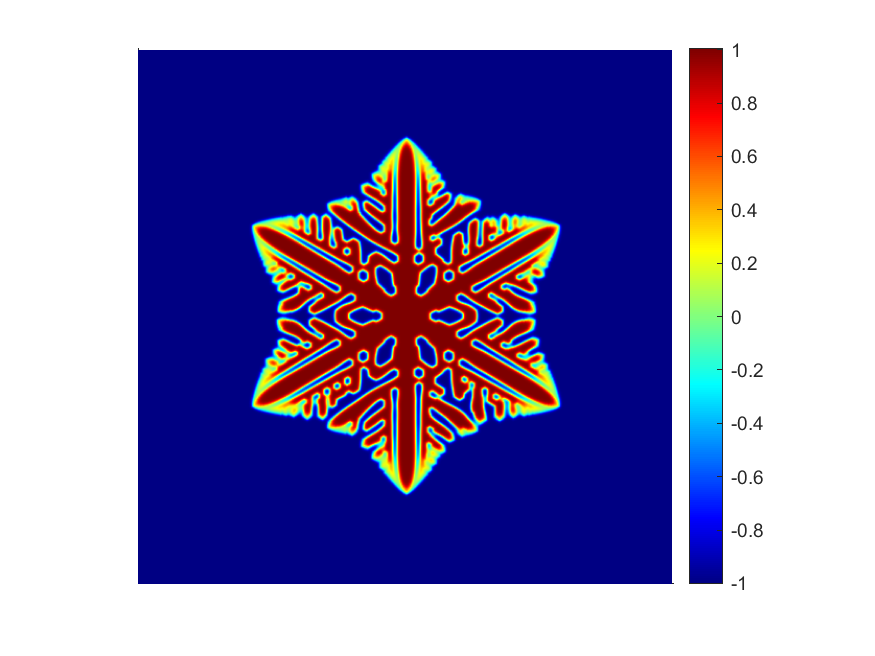}}
	\end{minipage}
	\begin{minipage}[t]{0.24\linewidth}
		\centerline{\includegraphics[scale=0.28]{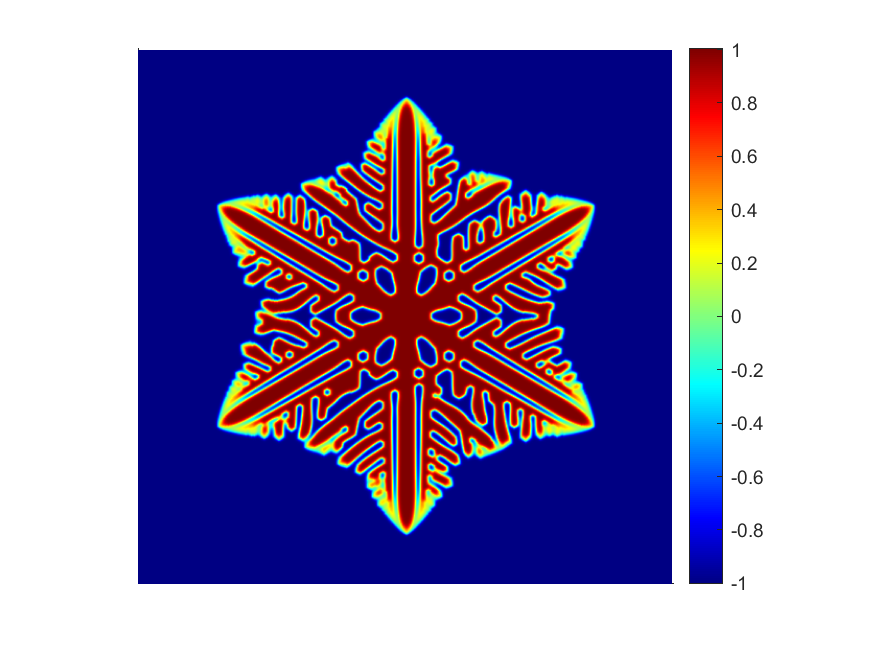}}
	\end{minipage}
	\begin{minipage}[t]{0.24\linewidth}
		\centerline{\includegraphics[scale=0.28]{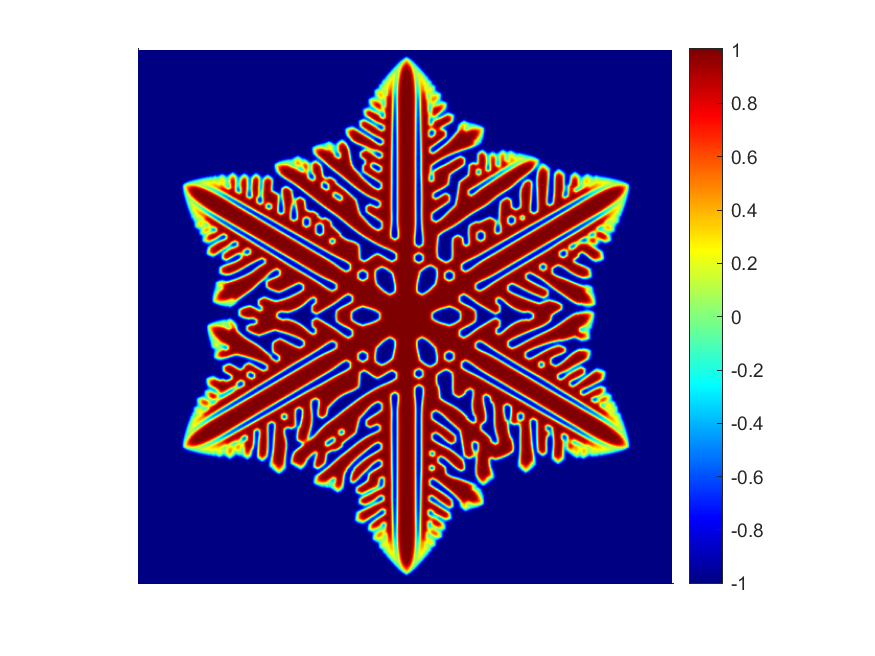}}
	\end{minipage}
	\centerline{(c) Phase field $\phi$ for $K=0.8$ at $t=15, 20, 25, 30$}
	\vskip 3mm	
	\begin{minipage}[t]{0.24\linewidth}
		\centerline{\includegraphics[scale=0.28]{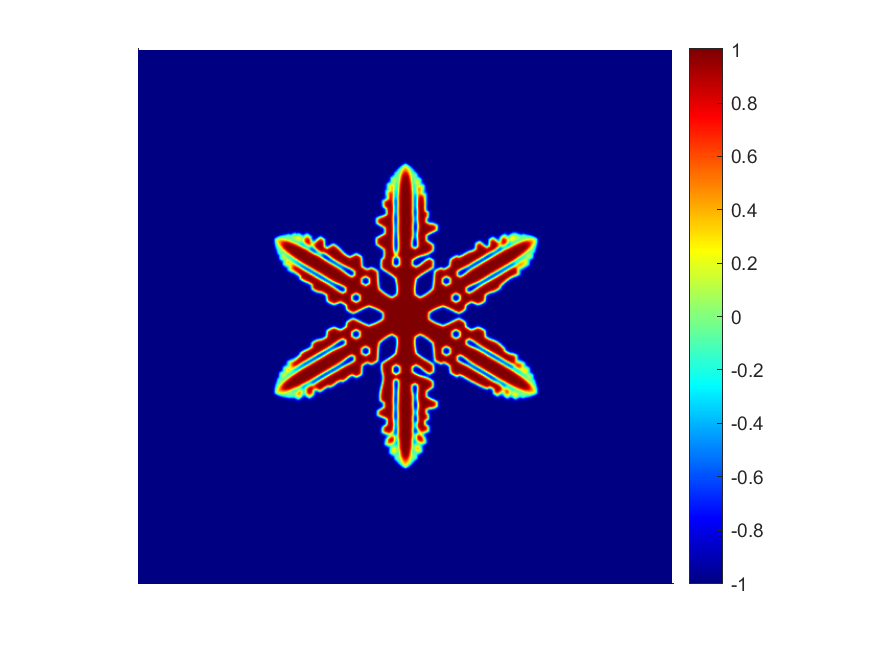}}
	\end{minipage}
	\begin{minipage}[t]{0.24\linewidth}
		\centerline{\includegraphics[scale=0.28]{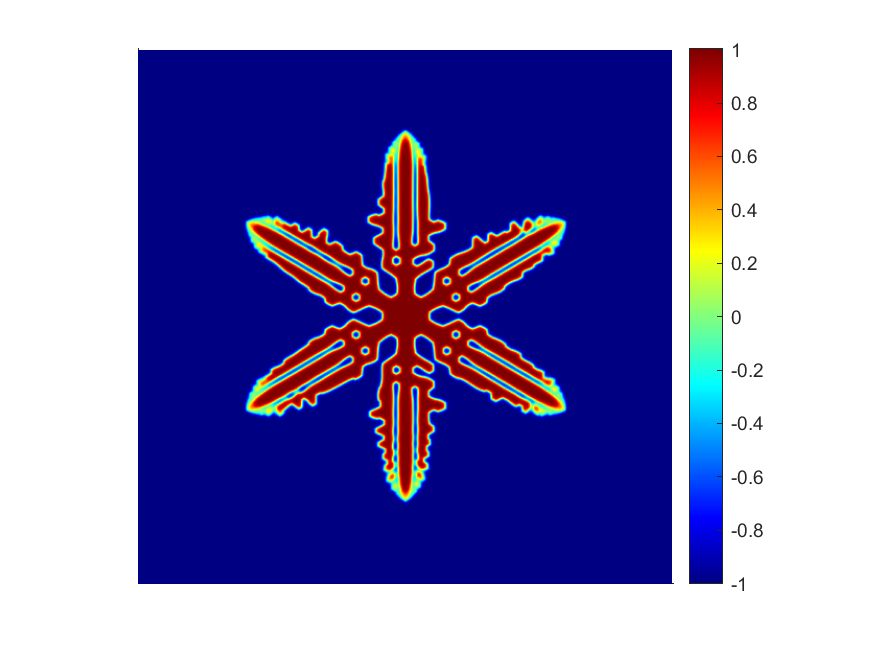}}
	\end{minipage}
	\begin{minipage}[t]{0.24\linewidth}
		\centerline{\includegraphics[scale=0.28]{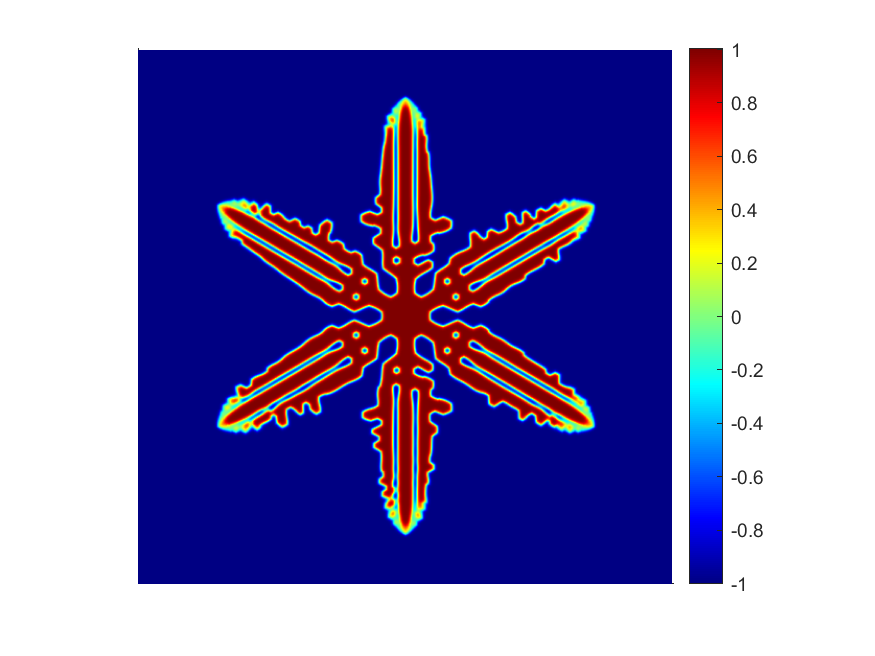}}
	\end{minipage}
	\begin{minipage}[t]{0.24\linewidth}
		\centerline{\includegraphics[scale=0.28]{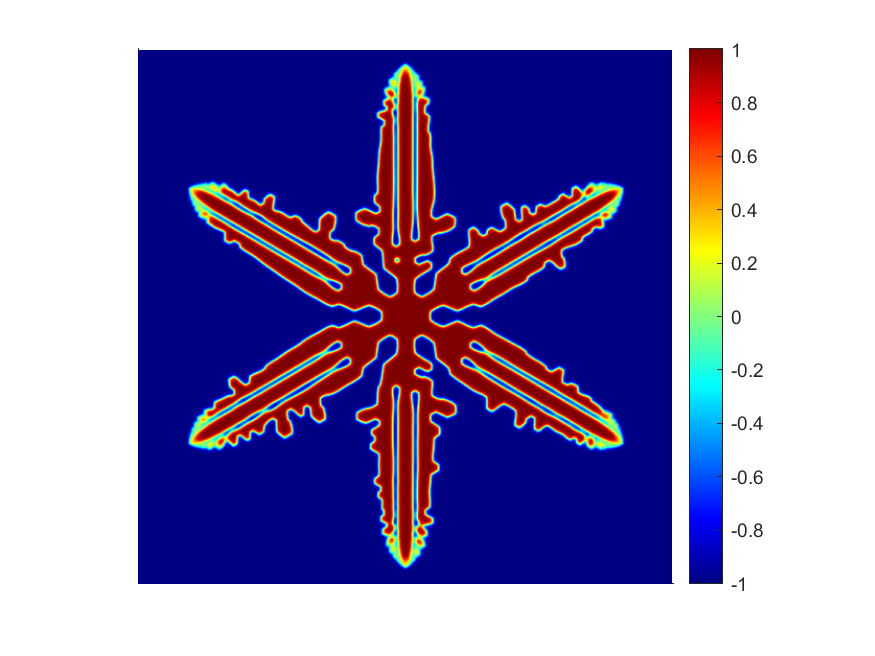}}
	\end{minipage}
	\centerline{(d) Phase field $\phi$ for $K=0.9$ at $t=20, 25, 30, 35$}
	\vskip 3mm	
	\begin{minipage}[t]{0.24\linewidth}
		\centerline{\includegraphics[scale=0.28]{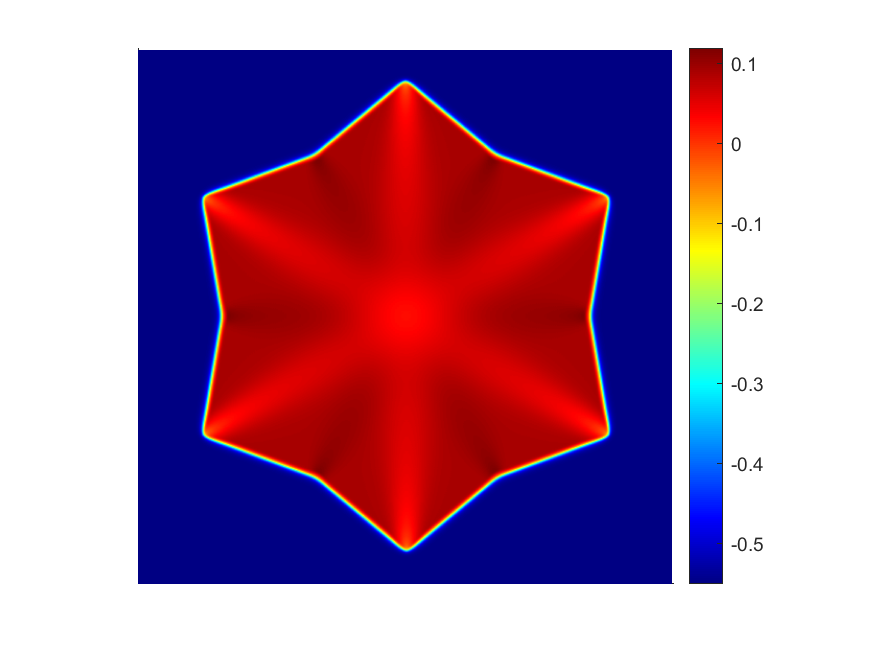}}
	\end{minipage}
	\begin{minipage}[t]{0.24\linewidth}
		\centerline{\includegraphics[scale=0.28]{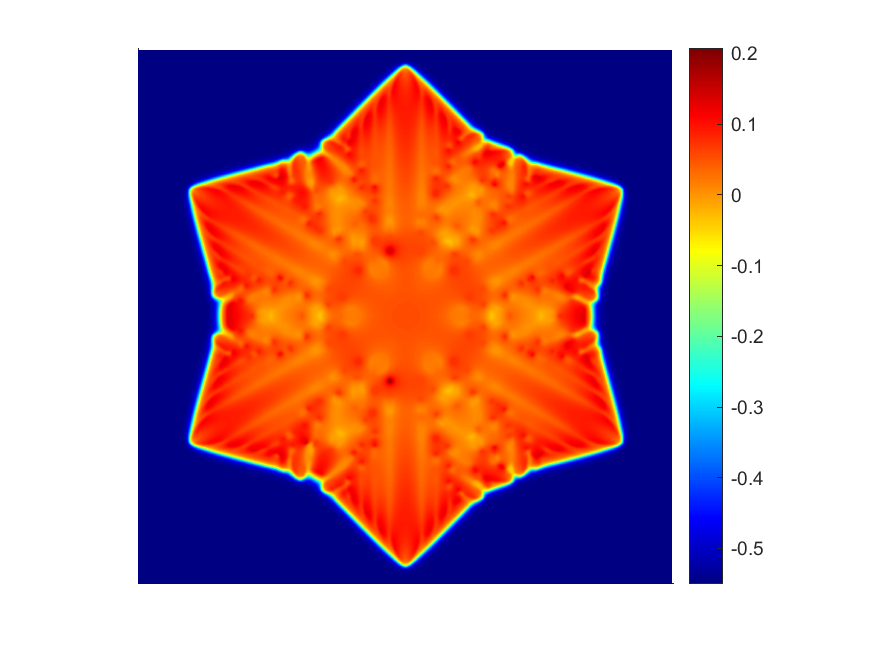}}
	\end{minipage}
	\begin{minipage}[t]{0.24\linewidth}
		\centerline{\includegraphics[scale=0.28]{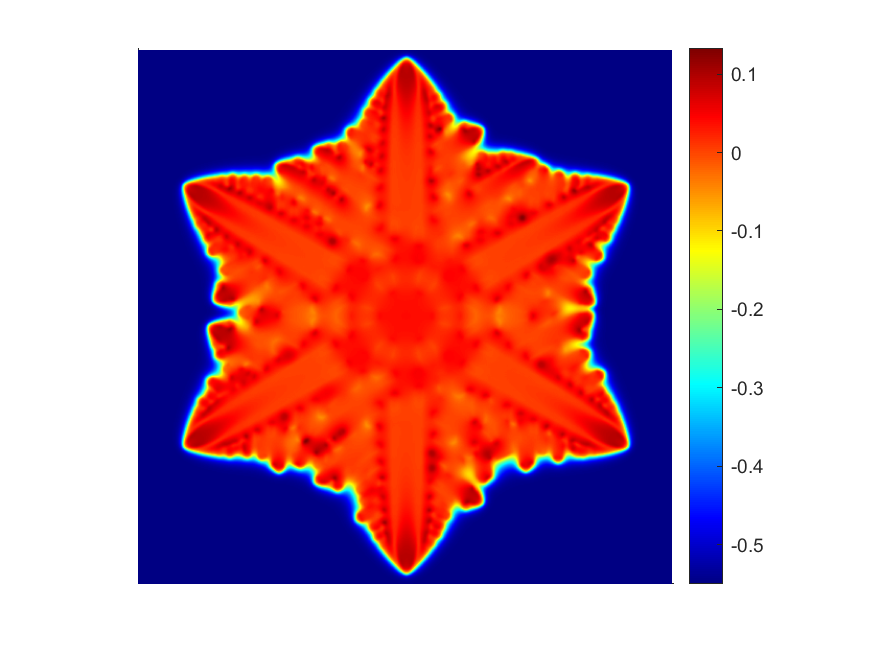}}
	\end{minipage}
	\begin{minipage}[t]{0.24\linewidth}
		\centerline{\includegraphics[scale=0.28]{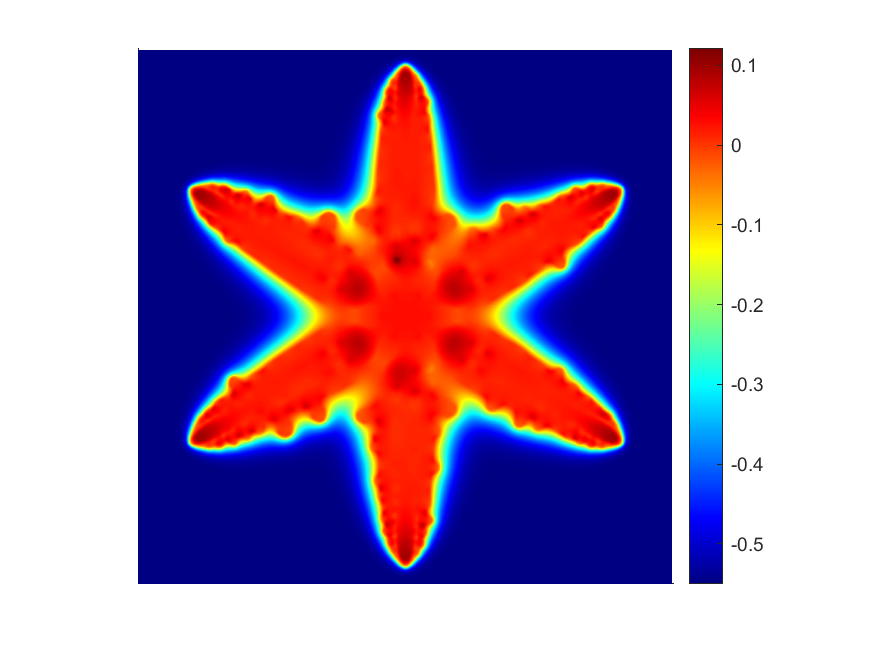}}
	\end{minipage}
	\centerline{(e) Temperature $T$ at the last snapshot of each case}	
	\caption{ Evolution of the phase field $\phi$ and temperature $T$ for sixfold anisotropy with different values of $K$.}
	\label{fig_example4_K}
\end{figure*}

\begin{figure*}[htbp]
	\begin{minipage}[t]{0.49\linewidth}
		\centerline{\includegraphics[scale=0.5]{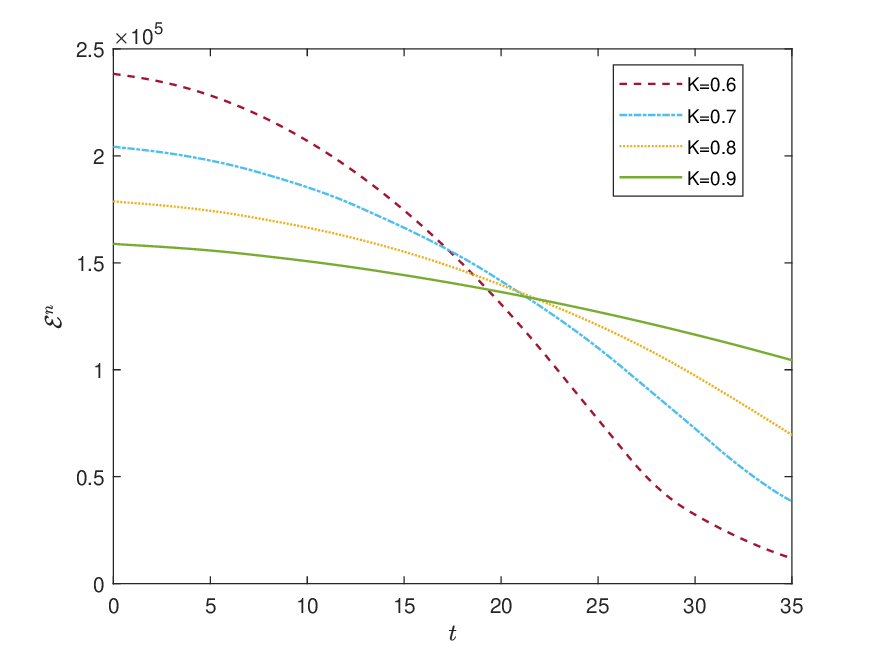}}
		\centerline{(a) Modified energy $\mathcal{E}_{n}$}
	\end{minipage}
	\begin{minipage}[t]{0.49\linewidth}
		\centerline{\includegraphics[scale=0.5]{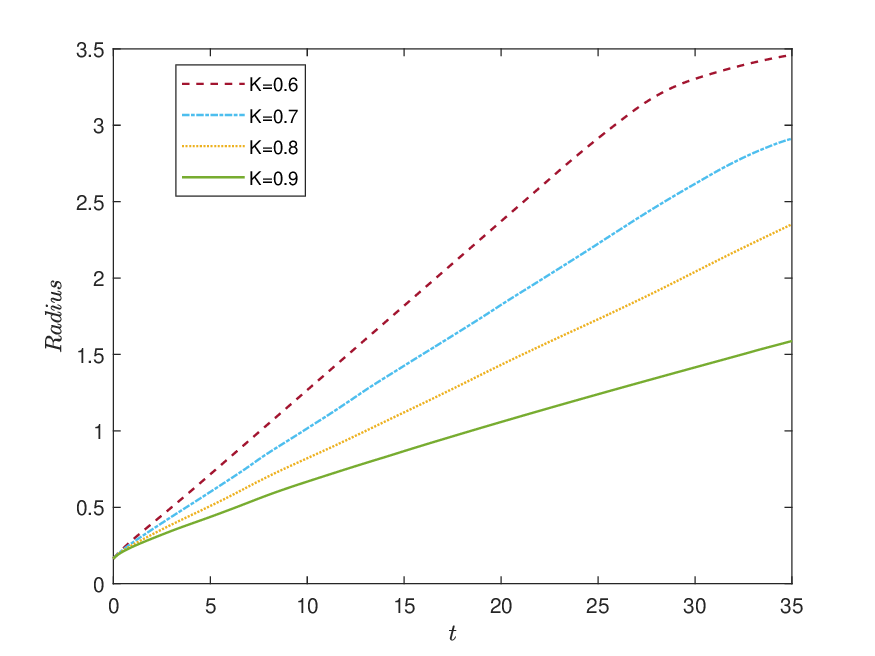}}
		\centerline{(b) Equivalent crystal radius}
	\end{minipage}
	\caption{Evolution of the modified energy $\mathcal E_n$ and equivalent crystal radius for two-dimensional sixfold anisotropy with different values of the latent heat coefficient $K$.
	}\label{fig_example4}
\end{figure*}

\begin{figure*}[htbp]
	\begin{minipage}[t]{0.24\linewidth}
		\centerline{\includegraphics[scale=0.28]{phi20K8.eps}}
	\end{minipage}
	\begin{minipage}[t]{0.24\linewidth}
		\centerline{\includegraphics[scale=0.28]{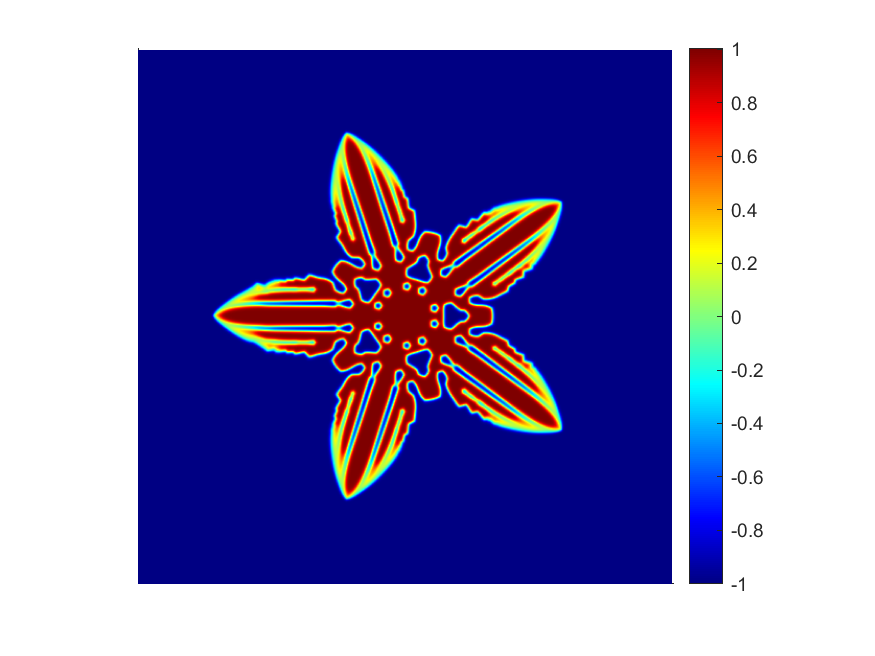}}
	\end{minipage}
	\begin{minipage}[t]{0.24\linewidth}
		\centerline{\includegraphics[scale=0.28]{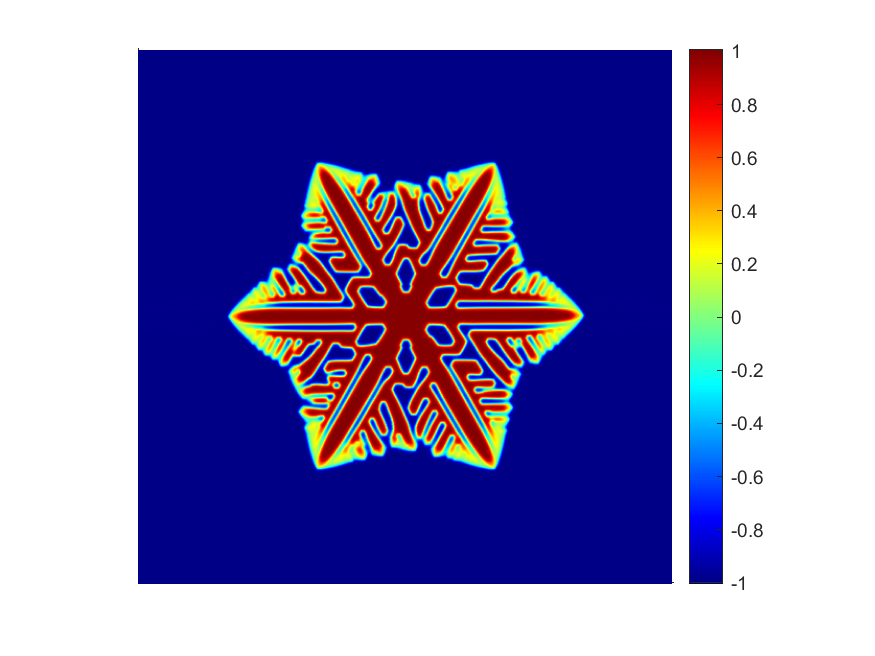}}
	\end{minipage}
	\begin{minipage}[t]{0.24\linewidth}
		\centerline{\includegraphics[scale=0.28]{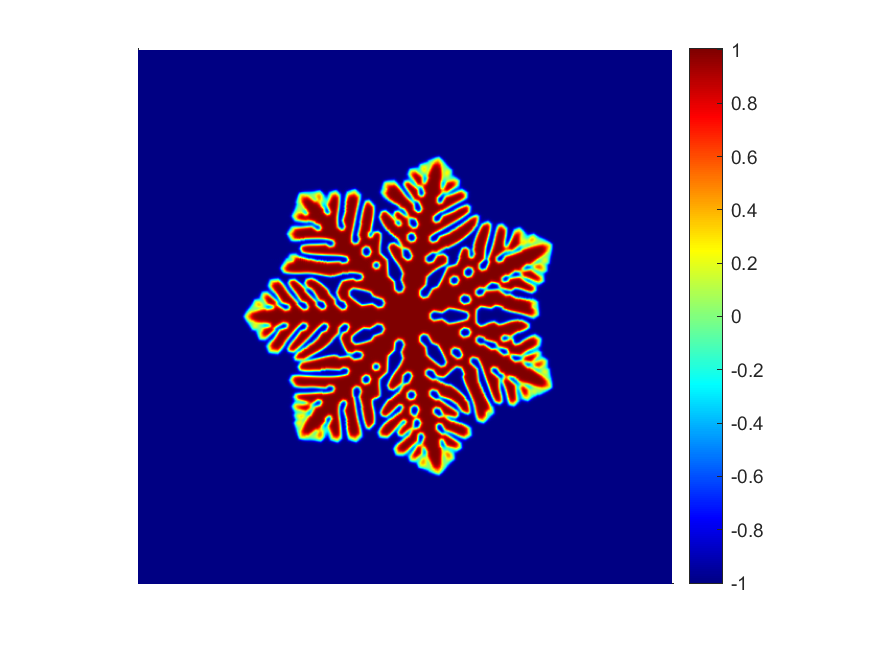}}
	\end{minipage}
	\centerline{(a) Phase field $\phi$}
	\vskip 3mm
	\begin{minipage}[t]{0.24\linewidth}
		\centerline{\includegraphics[scale=0.28]{u20K8.eps}}
	\end{minipage}
	\begin{minipage}[t]{0.24\linewidth}
		\centerline{\includegraphics[scale=0.28]{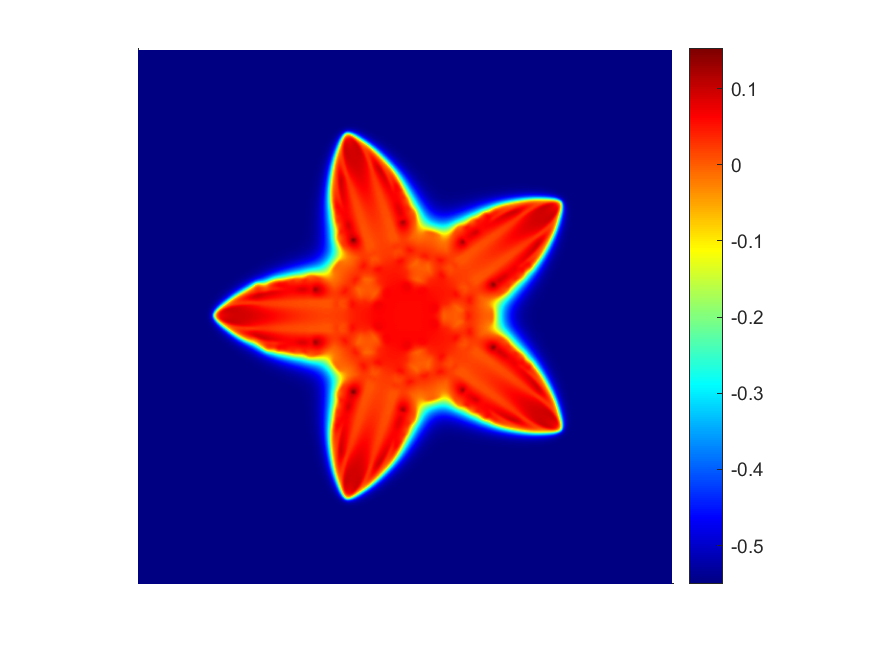}}
	\end{minipage}
	\begin{minipage}[t]{0.24\linewidth}
		\centerline{\includegraphics[scale=0.28]{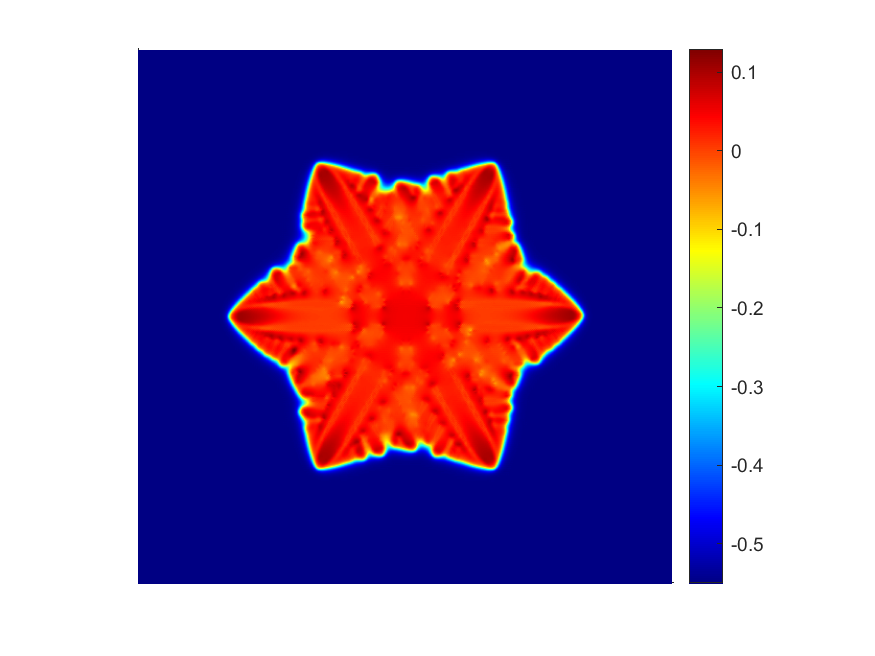}}
	\end{minipage}
	\begin{minipage}[t]{0.24\linewidth}
		\centerline{\includegraphics[scale=0.28]{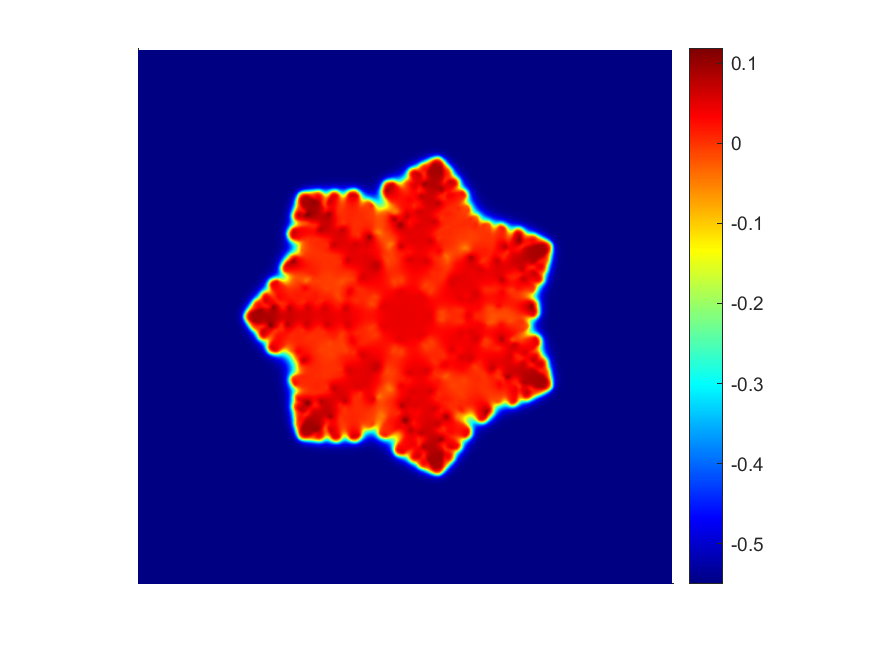}}
	\end{minipage}
	\centerline{(b) Temperature $T$}
	\caption{Dendritic morphologies for different anisotropy orders $m$ at $t=20$. From left to right, $m=4,5,6$, and $7$.
	}\label{fig_example4_m}
\end{figure*}

Finally, we consider the growth and interaction of three crystal nuclei under sixfold anisotropy. The initial conditions are given by
\begin{equation*}\label{ex5_initial}
\phi^0(\x)=\dps\sum_{i=1}^{3}\tanh\left(\dfrac{r_0-|\x-\x_i|^2}{\epsilon_0}\right)+2,\quad T^0(\x)= 
\begin{cases} 
0,\qquad \quad \phi^0(\x)>0,\\ 
-0.55, \quad \mbox{otherwise}, 
\end{cases}
\end{equation*}
where $\x_1=(\frac{3}{4}\pi,\frac{3}{4}\pi)$, $\x_2=(\frac{4}{5}\pi,\frac{13}{10}\pi)$, $\x_3=(\frac{13}{10}\pi,\pi)$, with $r_0=2.7\times10^{-2}$ and $\epsilon_0=5.4\times10^{-3}$. The parameters are chosen as $K=0.7,\ \epsilon_{4}=0.04$, and $\theta_0=90^{\circ}$, while the remaining parameters are taken from \eqref{ex3_params}.

Figure~\ref{fig_example5} presents snapshots of the phase field $\phi$ and temperature $T$.  The three initially separated nuclei gradually develop into sixfold dendritic structures with secondary branches. As the dendrites extend toward one another, neighboring branches come into contact and eventually merge, producing a more complex morphology. The temperature distributions exhibit patterns similar to those of the phase field. The numerical results are consistent with those reported in \cite{ohno2020quantitative,yang2020efficient,guo2024efficient}.

\begin{figure*}[htbp]
	\begin{minipage}[t]{0.24\linewidth}
		\centerline{\includegraphics[scale=0.28]{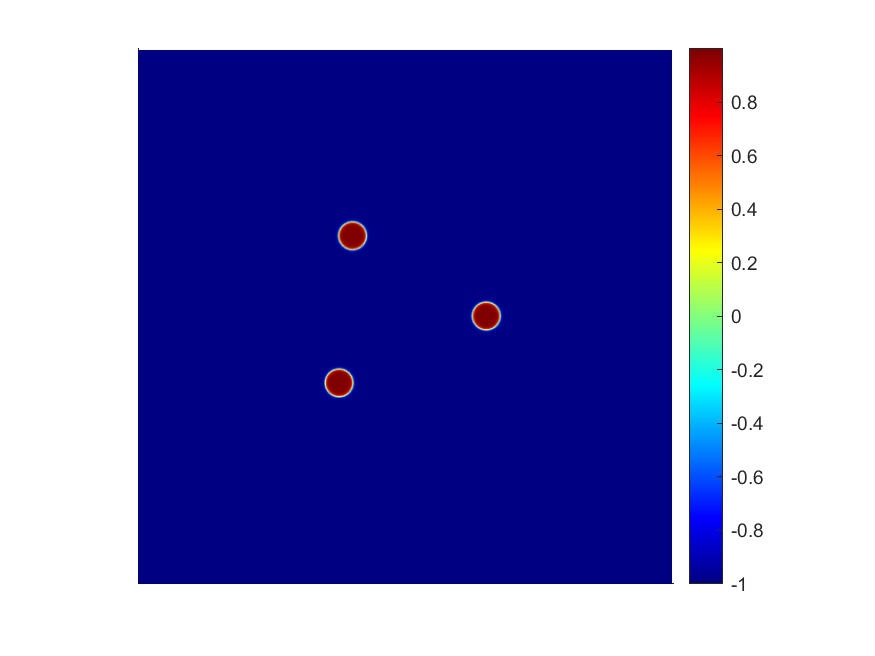}}
	\end{minipage}
	\begin{minipage}[t]{0.24\linewidth}
		\centerline{\includegraphics[scale=0.28]{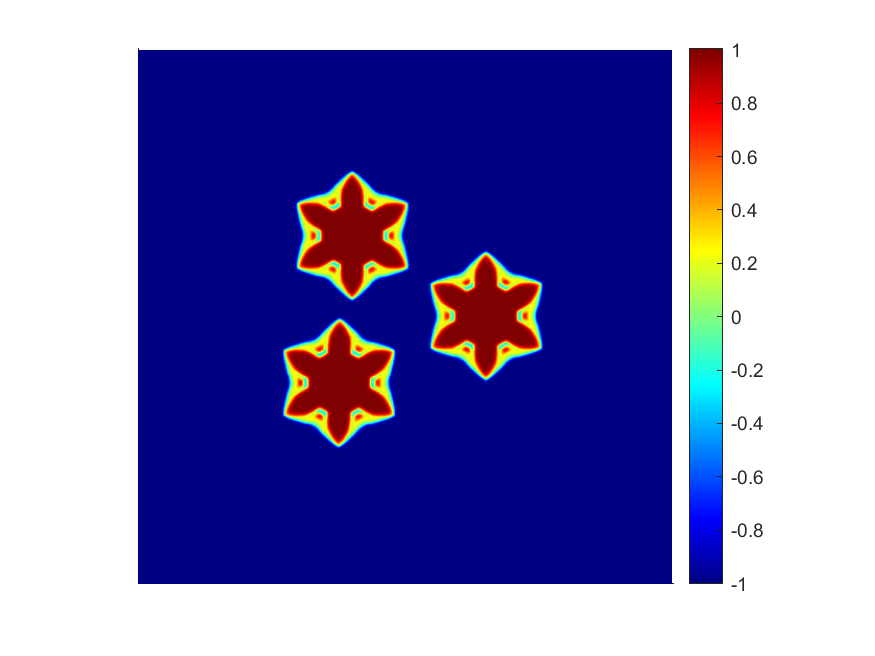}}
	\end{minipage}
	\begin{minipage}[t]{0.24\linewidth}
		\centerline{\includegraphics[scale=0.28]{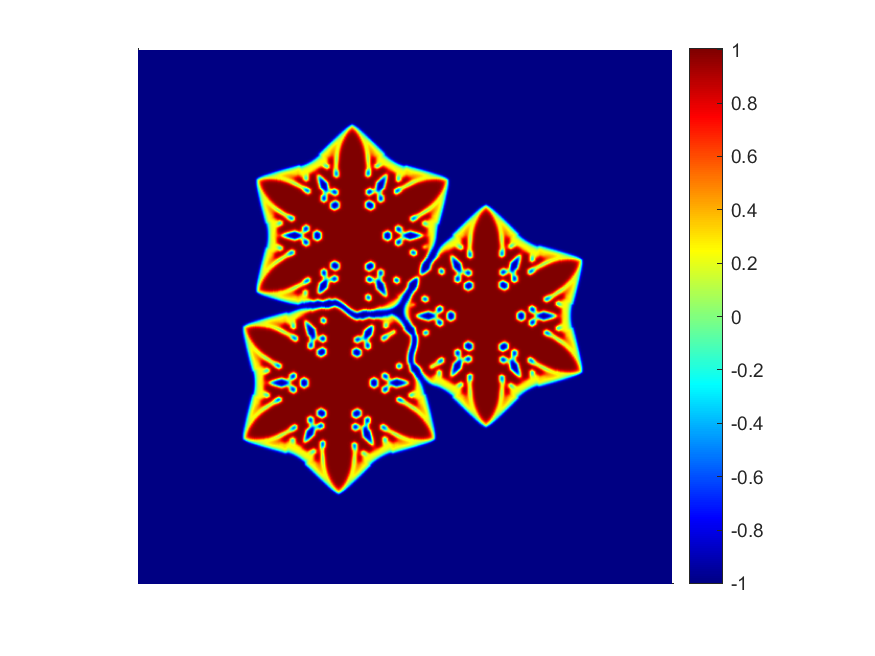}}
	\end{minipage}
	\begin{minipage}[t]{0.24\linewidth}
		\centerline{\includegraphics[scale=0.28]{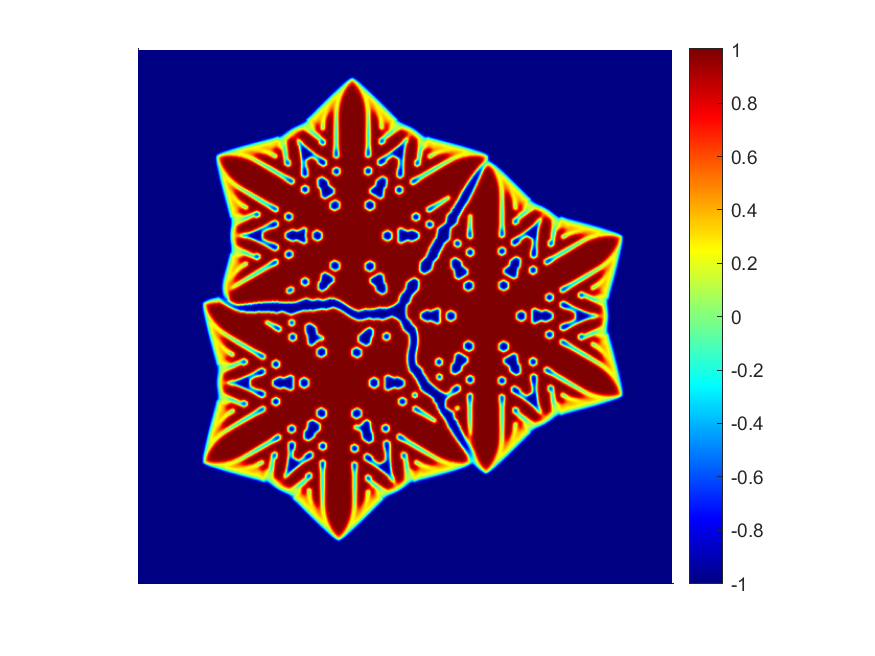}}
	\end{minipage}
	\centerline{(a) Phase field $\phi$}
	\vskip 3mm
	\begin{minipage}[t]{0.24\linewidth}
		\centerline{\includegraphics[scale=0.28]{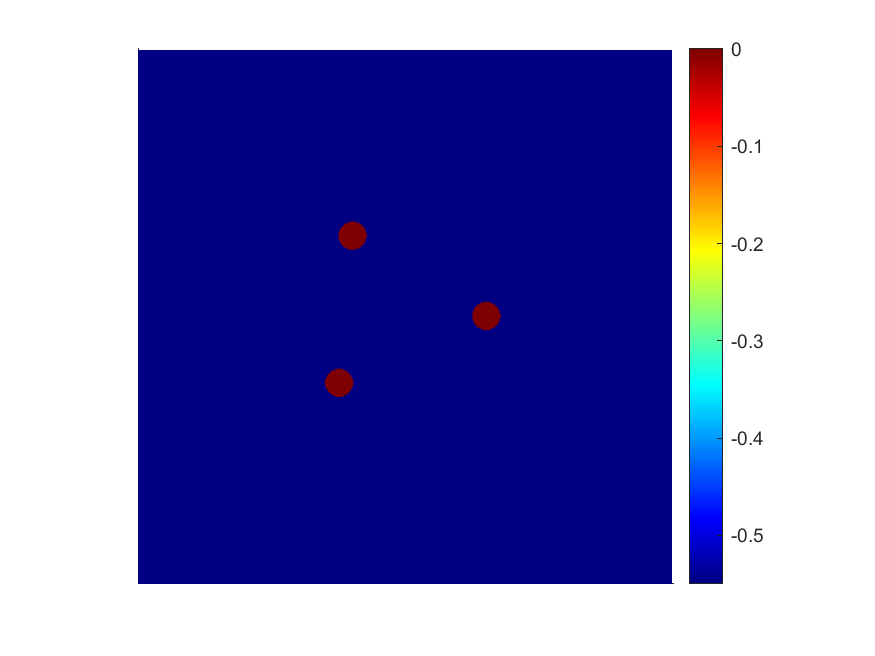}}
	\end{minipage}
	\begin{minipage}[t]{0.24\linewidth}
		\centerline{\includegraphics[scale=0.28]{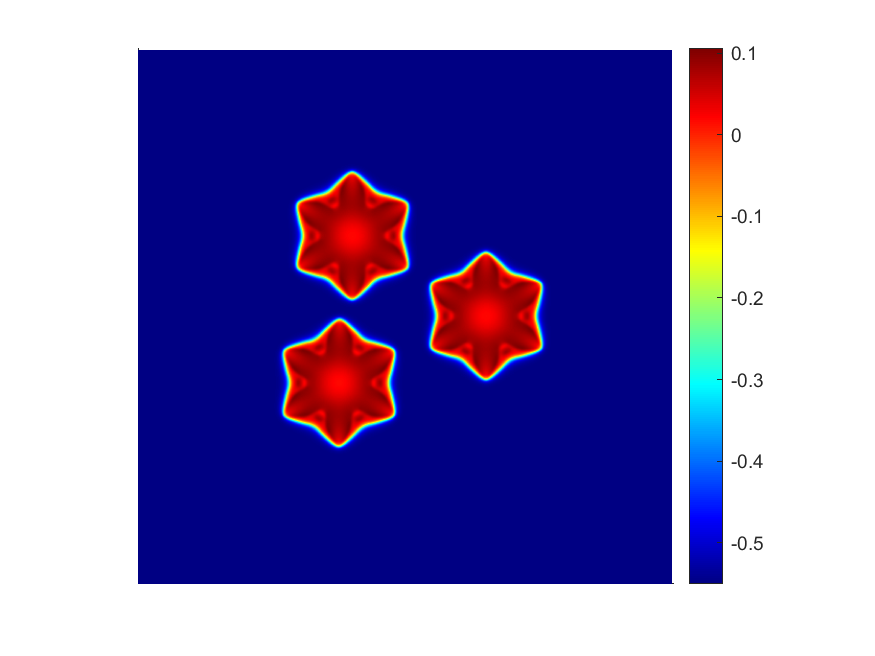}}
	\end{minipage}
	\begin{minipage}[t]{0.24\linewidth}
		\centerline{\includegraphics[scale=0.28]{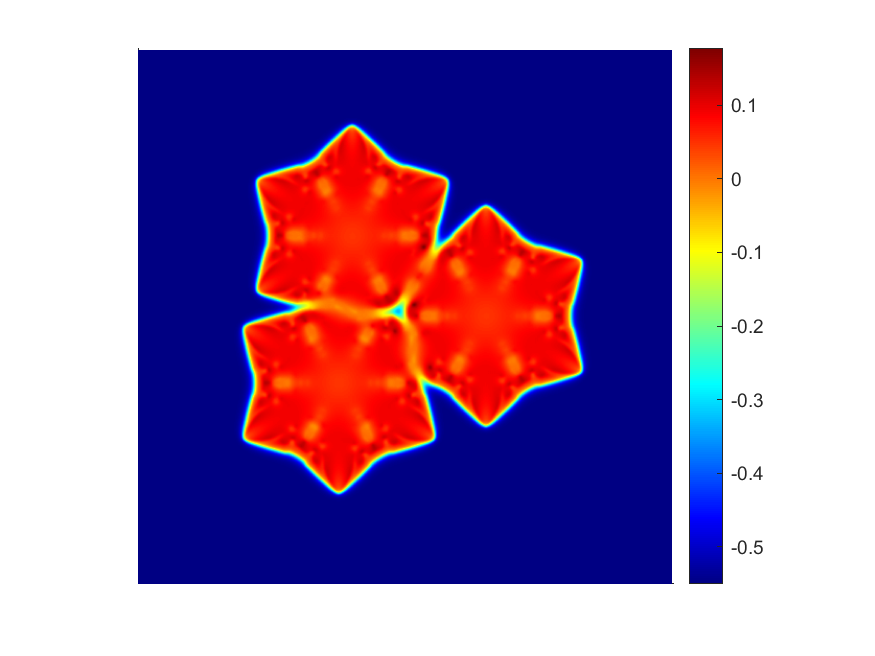}}
	\end{minipage}
	\begin{minipage}[t]{0.24\linewidth}
		\centerline{\includegraphics[scale=0.28]{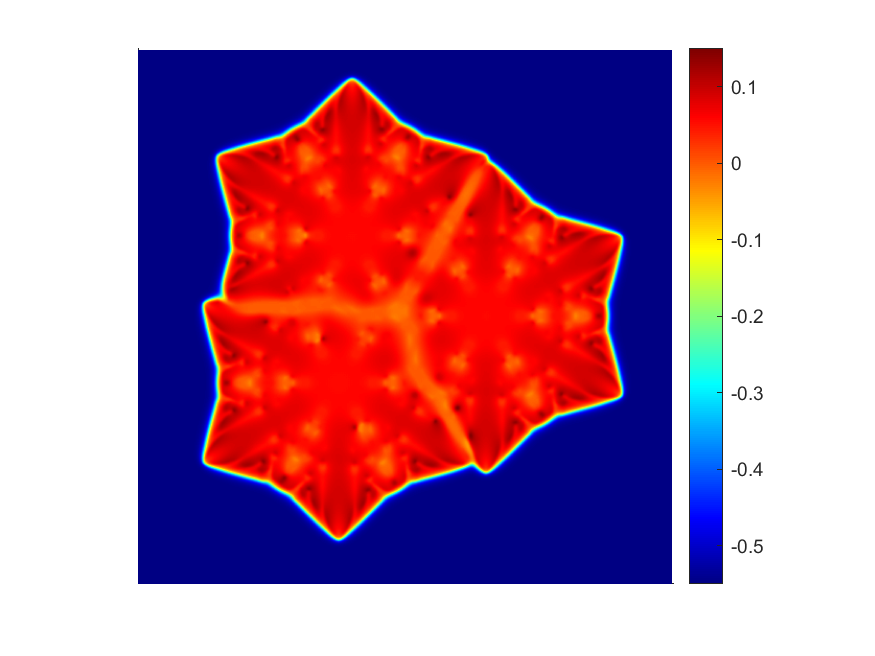}}
	\end{minipage}
	\centerline{(b) Temperature $T$}
	\caption{Evolution of the phase field $\phi$ and temperature $T$ for three crystal nuclei under sixfold anisotropy with $K=0.7$ at $t=0,5,10,15$.
	}\label{fig_example5}
\end{figure*}

\subsection{3D dendritic crystal growth}\label{sec44}
To examine the performance of the proposed scheme in three dimensions, we consider dendritic crystal growth with fourfold anisotropy on the computational domain $\Omega=(0,2\pi)^3$. The $4$-stage Gauss method is used with $\dt=0.01$, together with $128$ Fourier modes in each spatial direction. The initial conditions are given by
\begin{equation*}\label{ex6_initial}
\phi^0(\x)=\tanh\left(\dfrac{r_0-|\x-\x_0|^2}{\epsilon_0}\right),\quad T^0(\x)= \begin{cases} 
	0,\qquad \quad \phi^0(\x)>0,\\ 
	-0.55, \quad \mbox{otherwise}, 
\end{cases}
\end{equation*}
where $\x_0=(\pi,\pi,\pi)$, $r_0=2.7\times10^{-2}$, and $\epsilon_0=5.4\times10^{-3}$. The other parameters are set as
\begin{equation*}\label{ex6_params}
\begin{gathered}
C_1=0.6,\quad C_2=10,\quad C_0=5\times10^5,\quad \tau=10^3,\quad \varepsilon=1.5\times10^{-2},\\
\epsilon_4=0.1,\quad \lambda=4\times10^2,\quad D=2.5\times10^{-3},\quad \theta_0=0^\circ.
\end{gathered}
\end{equation*}

In Figure~\ref{fig_example6_K}, we present the evolution of the crystal isosurface $\phi=0$ for $K=0.5$, $1.0$, $1.5$, and $2.0$. The initially spherical nucleus gradually evolves into a dendritic structure with six dominant growth directions. A comparison among the four cases shows that increasing $K$ produces thinner branches and sharper tips, while reducing the crystal growth rate. This dependence on $K$ is further illustrated by the cross-sections at $z=\pi$ in Figure~\ref{fig_example6_2d}. Figure~\ref{fig_example6} compares the energy dissipation and crystal growth for different values of $K$. In three dimensions, the crystal size is measured by the radius of a sphere having the same solid-phase volume, defined by $\int_{\Omega} \frac{1+\phi}{2} \,d\x$. For all the tested values of $K$, the modified energy decreases monotonically and the equivalent crystal radius increases with time. Smaller values of $K$ correspond to a faster energy decay and a higher crystal growth rate. This behavior is consistent with that observed in the two-dimensional fourfold anisotropy case.

\begin{figure*}[htbp]
	\begin{minipage}[t]{0.24\linewidth}
		\centerline{\includegraphics[scale=0.28]{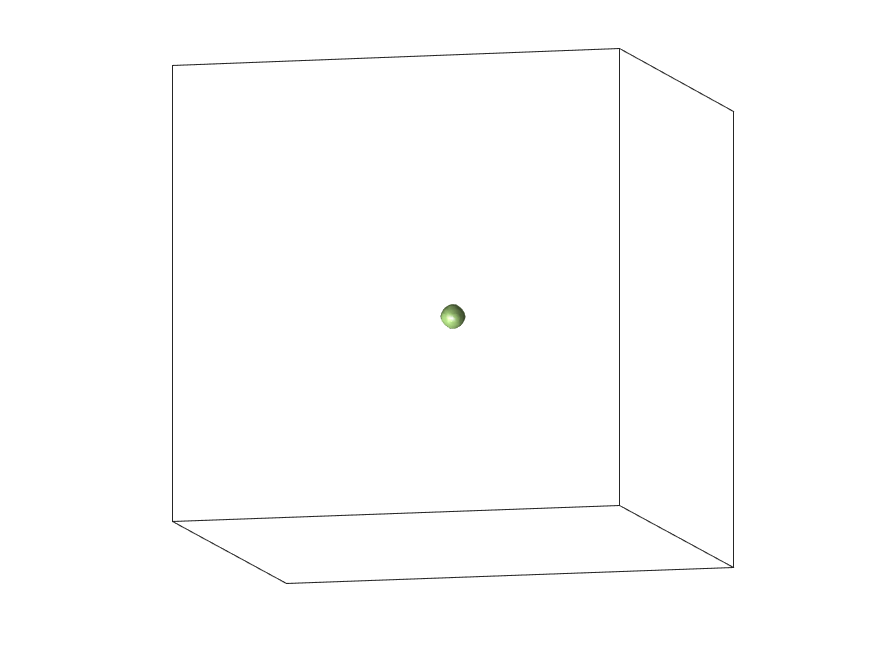}}
	\end{minipage}
	\begin{minipage}[t]{0.24\linewidth}
		\centerline{\includegraphics[scale=0.28]{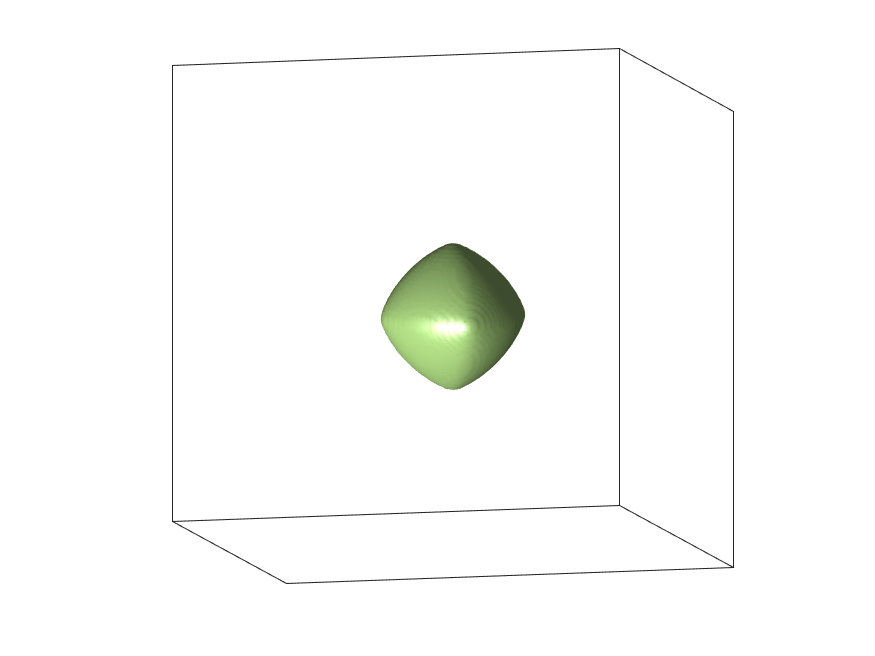}}
	\end{minipage}
	\begin{minipage}[t]{0.24\linewidth}
		\centerline{\includegraphics[scale=0.28]{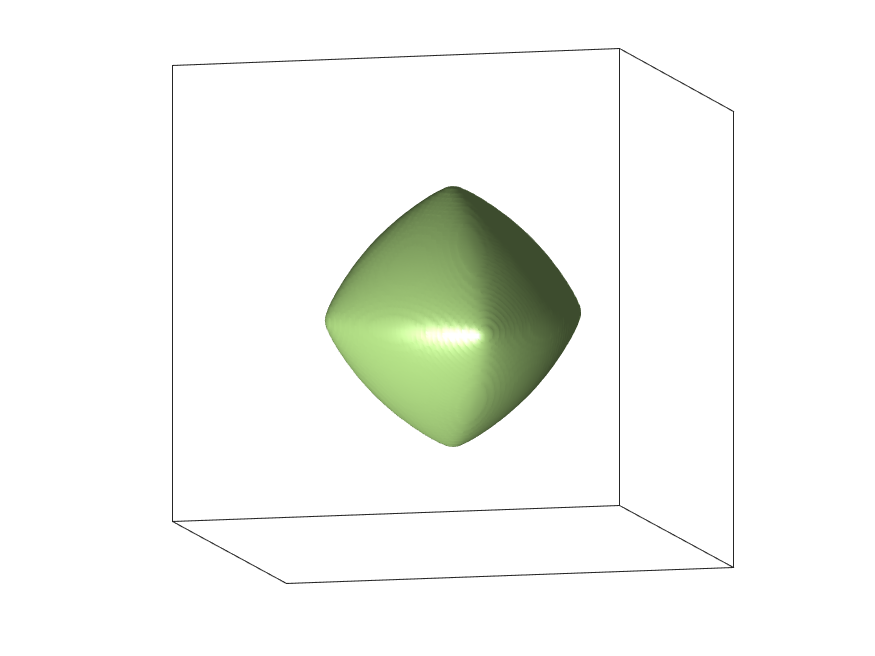}}
	\end{minipage}
	\begin{minipage}[t]{0.24\linewidth}
		\centerline{\includegraphics[scale=0.28]{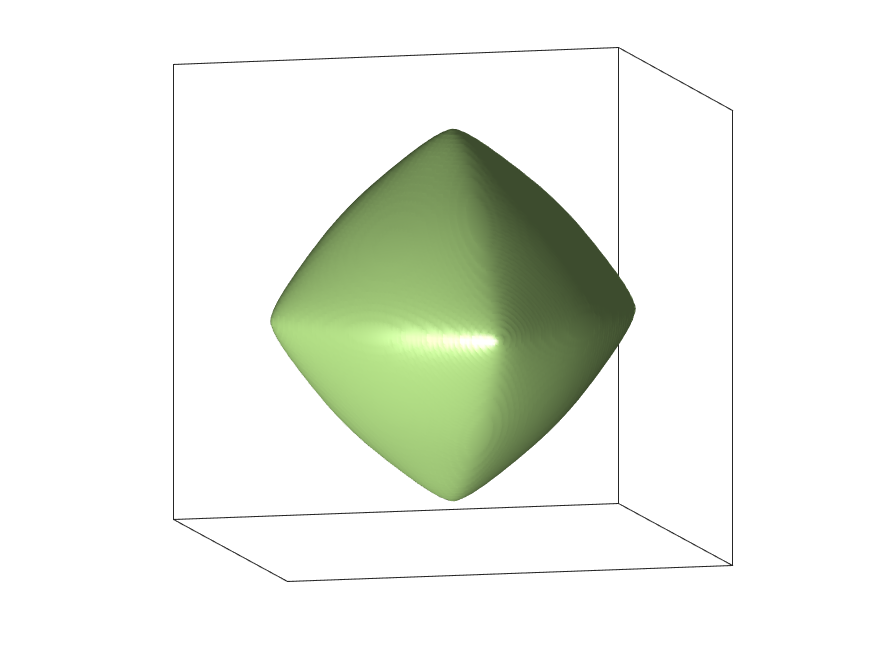}}
	\end{minipage}
	\centerline{(a) Crystal isosurface $\phi = 0$ at $t=0, 5, 10, 15$}
	\vskip 3mm
	\begin{minipage}[t]{0.24\linewidth}
		\centerline{\includegraphics[scale=0.28]{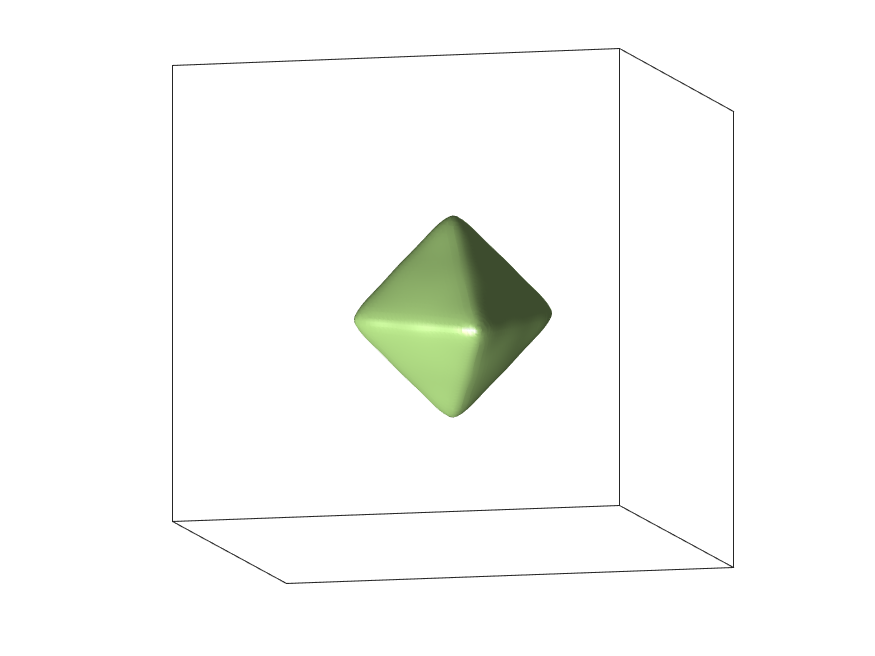}}
	\end{minipage}
	\begin{minipage}[t]{0.24\linewidth}
		\centerline{\includegraphics[scale=0.28]{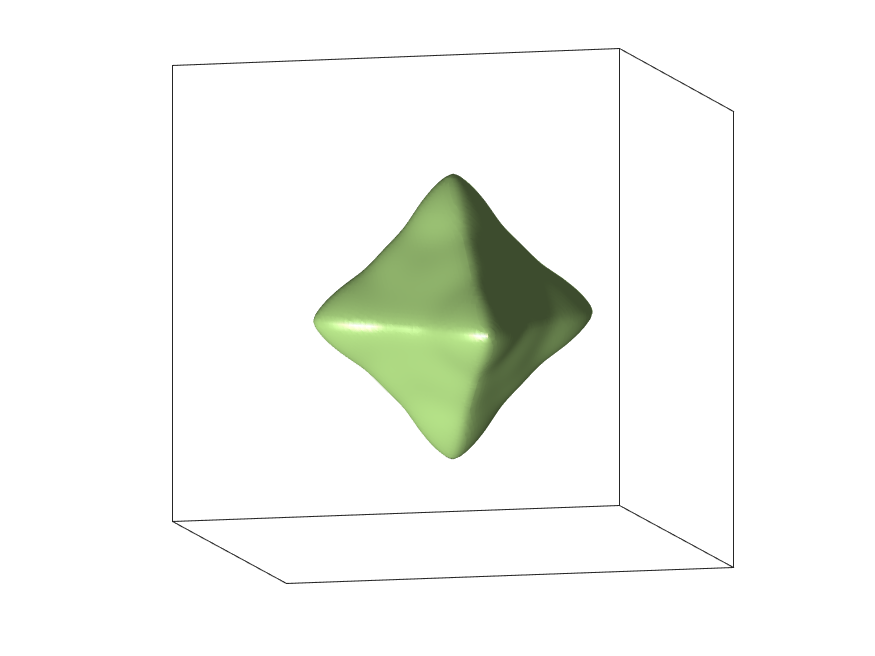}}
	\end{minipage}
	\begin{minipage}[t]{0.24\linewidth}
		\centerline{\includegraphics[scale=0.28]{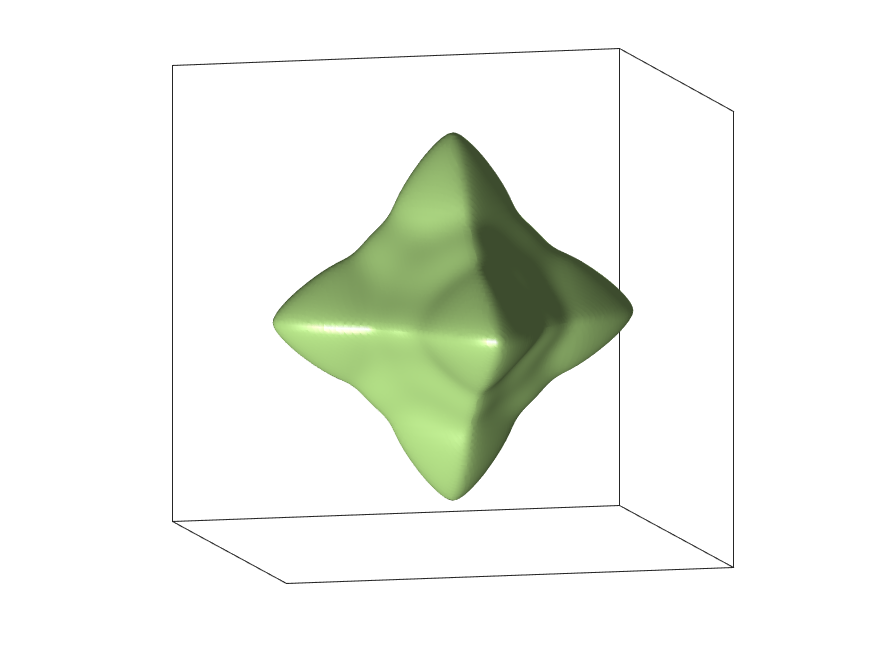}}
	\end{minipage}
	\begin{minipage}[t]{0.24\linewidth}
		\centerline{\includegraphics[scale=0.28]{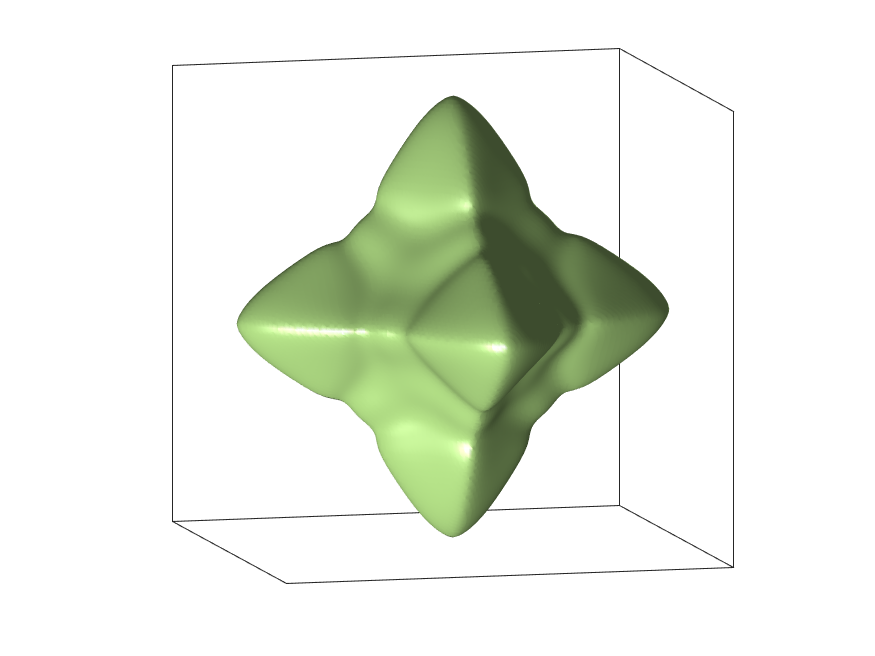}}
	\end{minipage}
	\centerline{(b) Crystal isosurface $\phi = 0$ at $t=10, 15, 20, 25$}
	\vskip 3mm
	\begin{minipage}[t]{0.24\linewidth}
		\centerline{\includegraphics[scale=0.28]{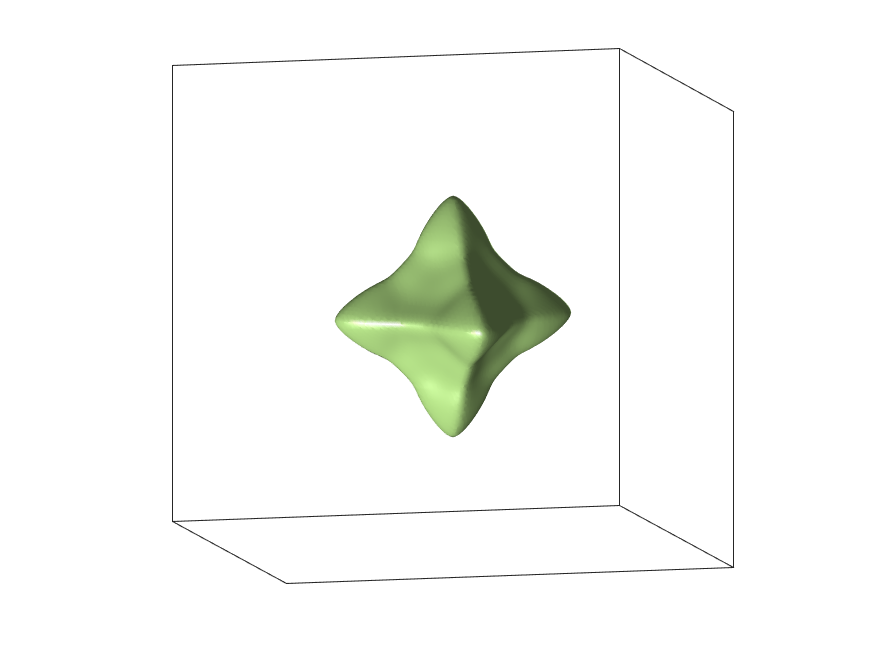}}
	\end{minipage}
	\begin{minipage}[t]{0.24\linewidth}
		\centerline{\includegraphics[scale=0.28]{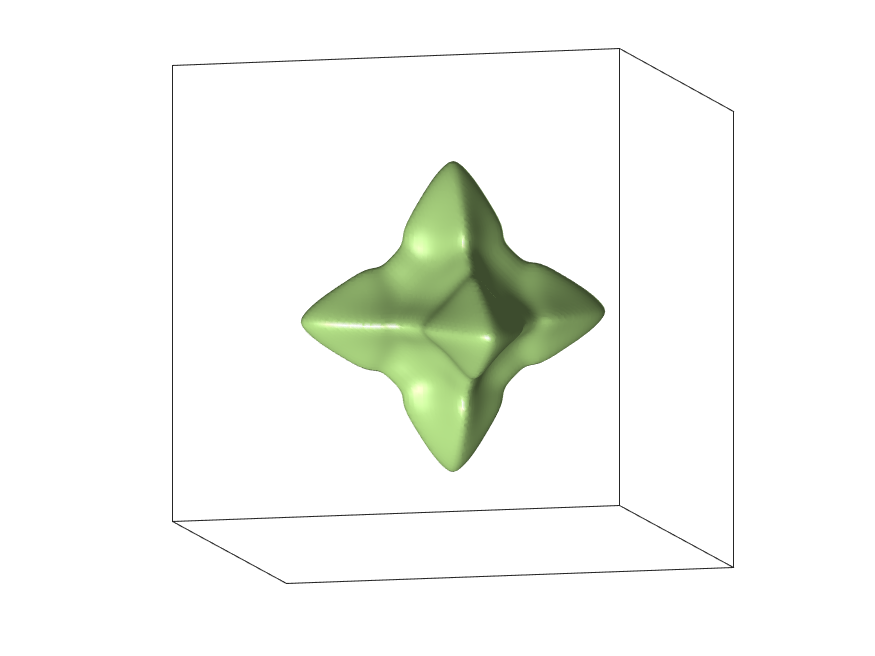}}
	\end{minipage}
	\begin{minipage}[t]{0.24\linewidth}
		\centerline{\includegraphics[scale=0.28]{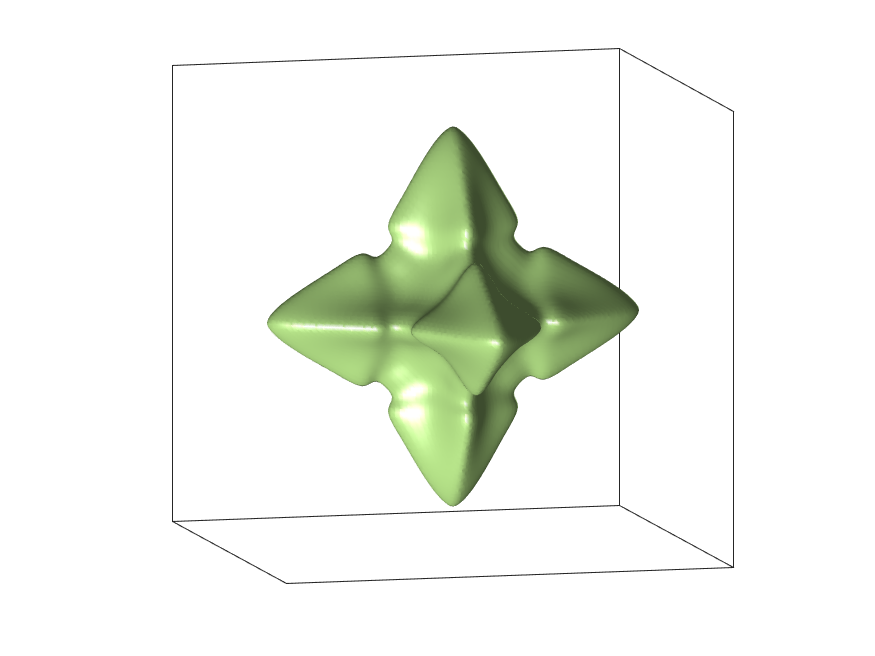}}
	\end{minipage}
	\begin{minipage}[t]{0.24\linewidth}
		\centerline{\includegraphics[scale=0.28]{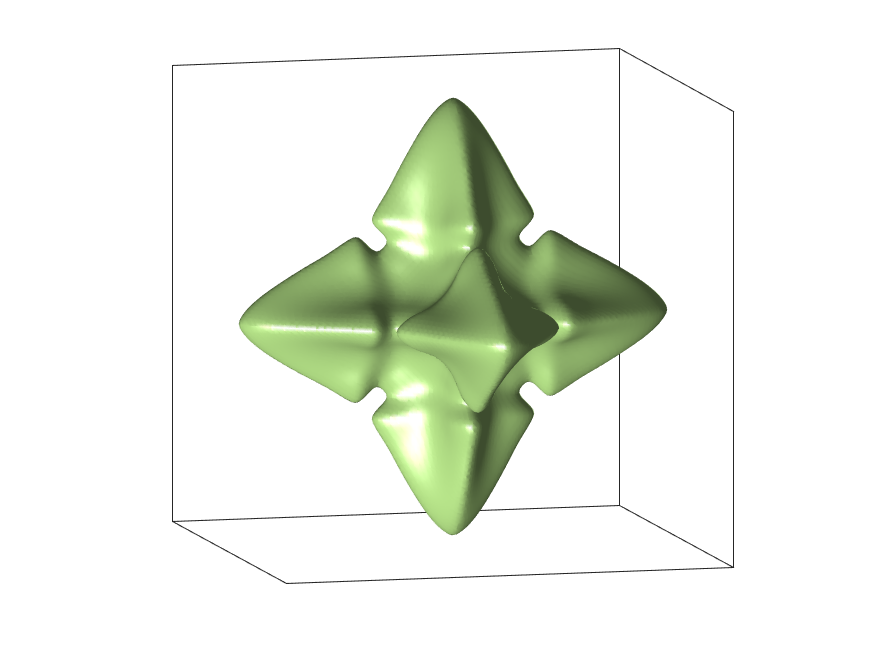}}
	\end{minipage}
	\centerline{(c) Crystal isosurface $\phi = 0$ at $t=15, 20, 25, 30$}
	\vskip 3mm
	\begin{minipage}[t]{0.24\linewidth}
		\centerline{\includegraphics[scale=0.28]{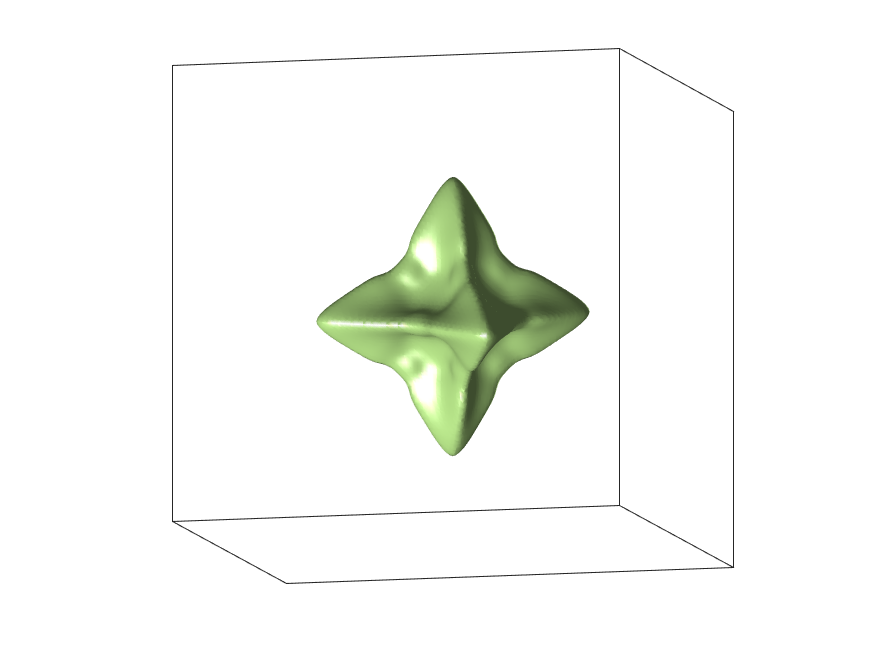}}
	\end{minipage}
	\begin{minipage}[t]{0.24\linewidth}
		\centerline{\includegraphics[scale=0.28]{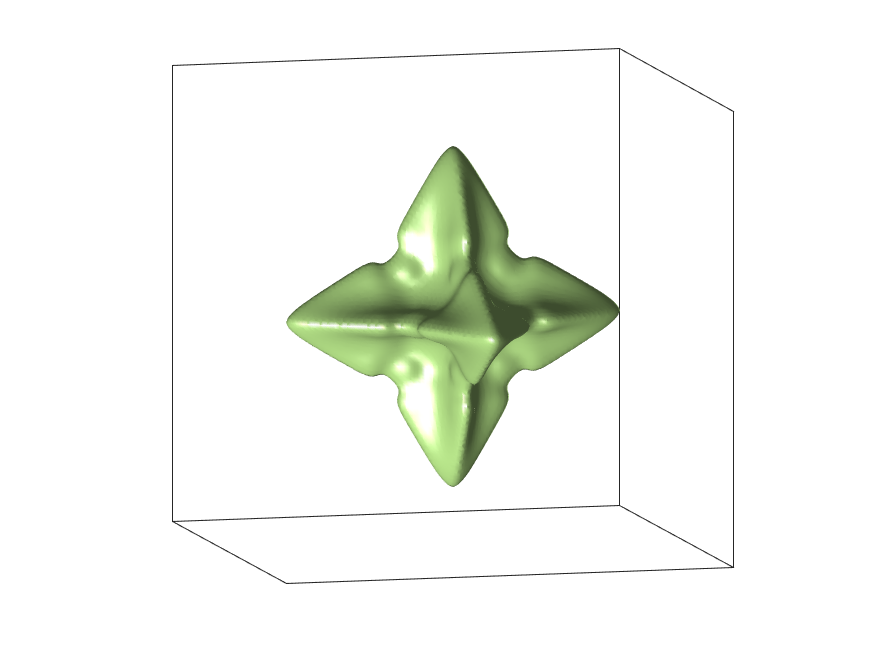}}
	\end{minipage}
	\begin{minipage}[t]{0.24\linewidth}
		\centerline{\includegraphics[scale=0.28]{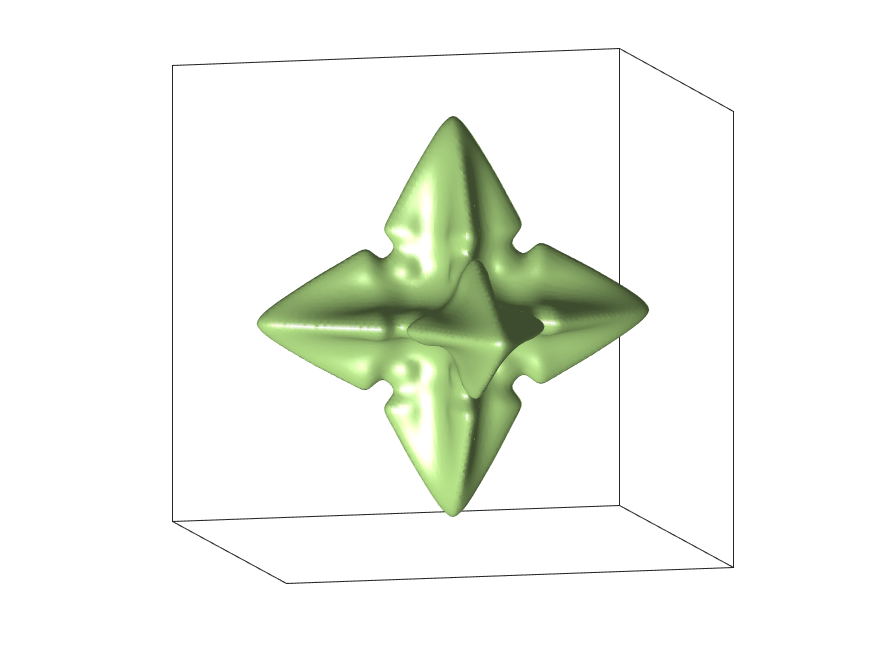}}
	\end{minipage}
	\begin{minipage}[t]{0.24\linewidth}
		\centerline{\includegraphics[scale=0.28]{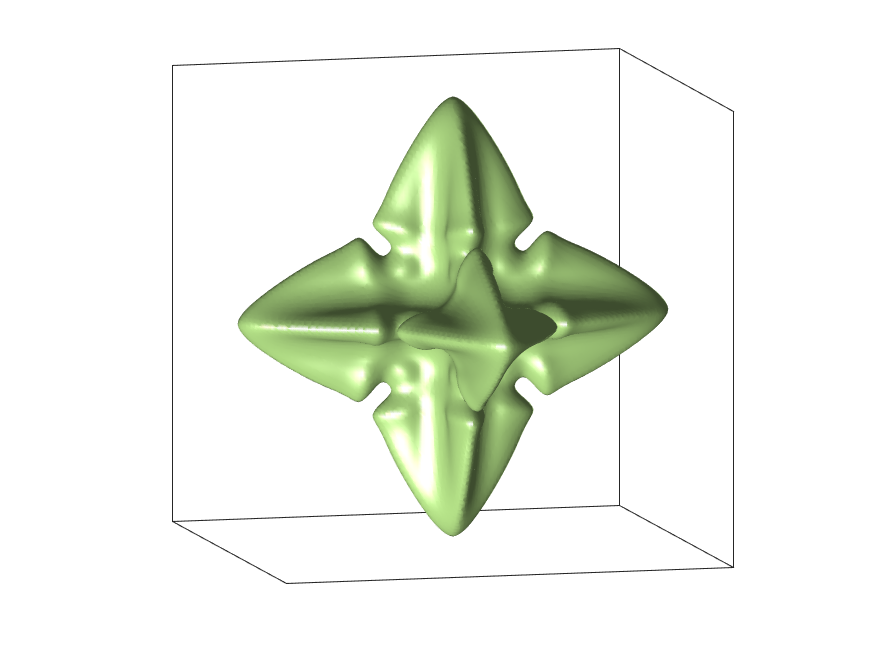}}
	\end{minipage}
	\centerline{(d) Crystal isosurface $\phi = 0$ at $t=20, 25, 30, 35$}
	\caption{Evolution of the crystal isosurface $\phi=0$ for fourfold anisotropy with different values of $K$. From top to bottom, $K=0.5$, $1.0$, $1.5$, and $2.0$.
	}\label{fig_example6_K}
\end{figure*}

\begin{figure*}[htbp]
	\begin{minipage}[t]{0.24\linewidth}
			\centerline{\includegraphics[scale=0.28]{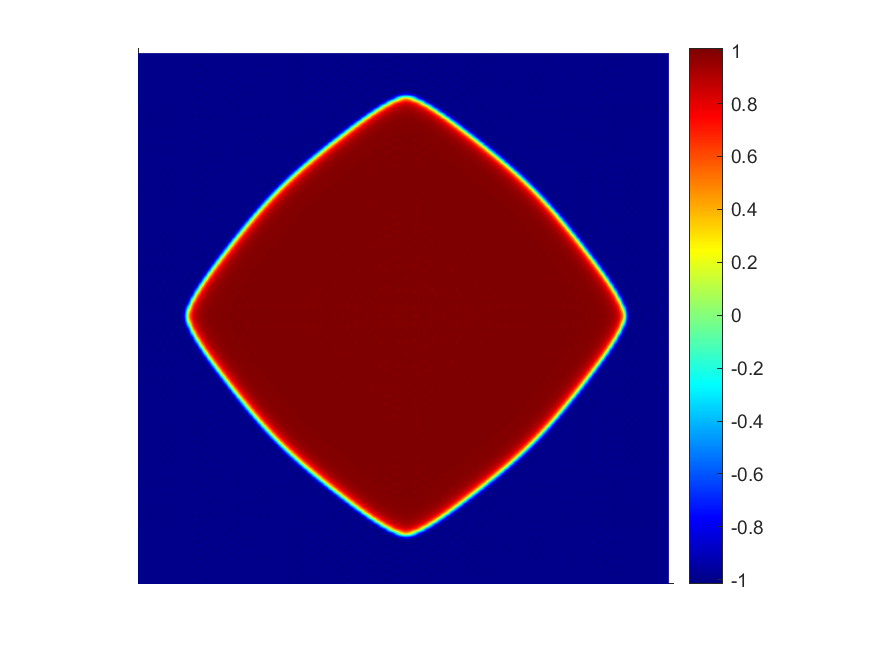}}
		\end{minipage}
	\begin{minipage}[t]{0.24\linewidth}
			\centerline{\includegraphics[scale=0.28]{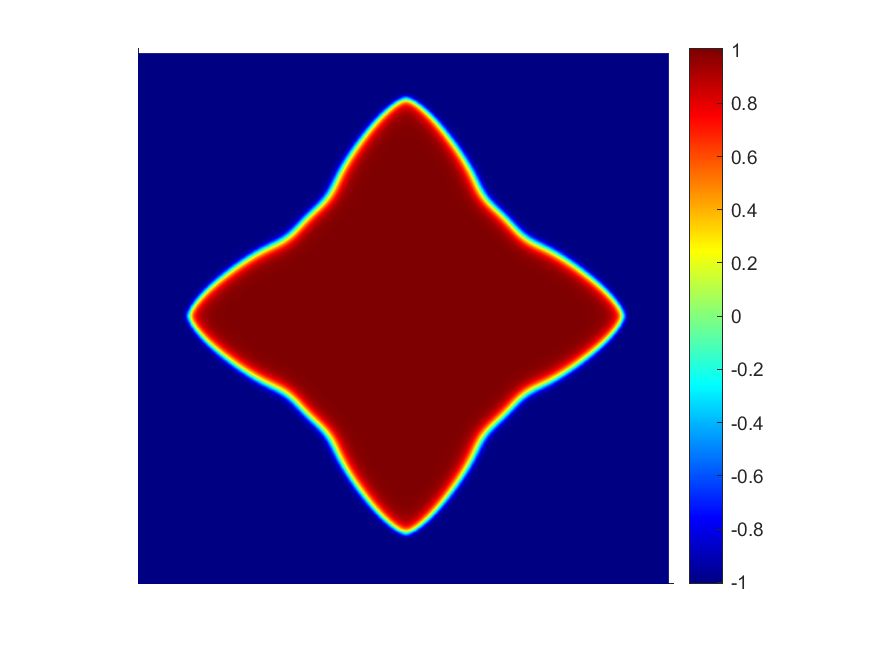}}
		\end{minipage}
	\begin{minipage}[t]{0.24\linewidth}
			\centerline{\includegraphics[scale=0.28]{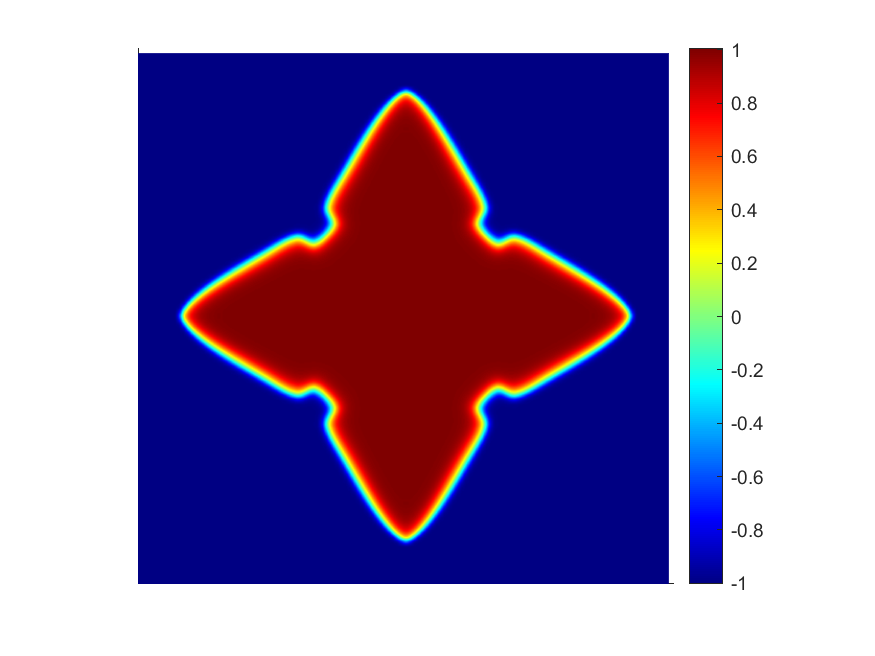}}
		\end{minipage}
	\begin{minipage}[t]{0.24\linewidth}
			\centerline{\includegraphics[scale=0.28]{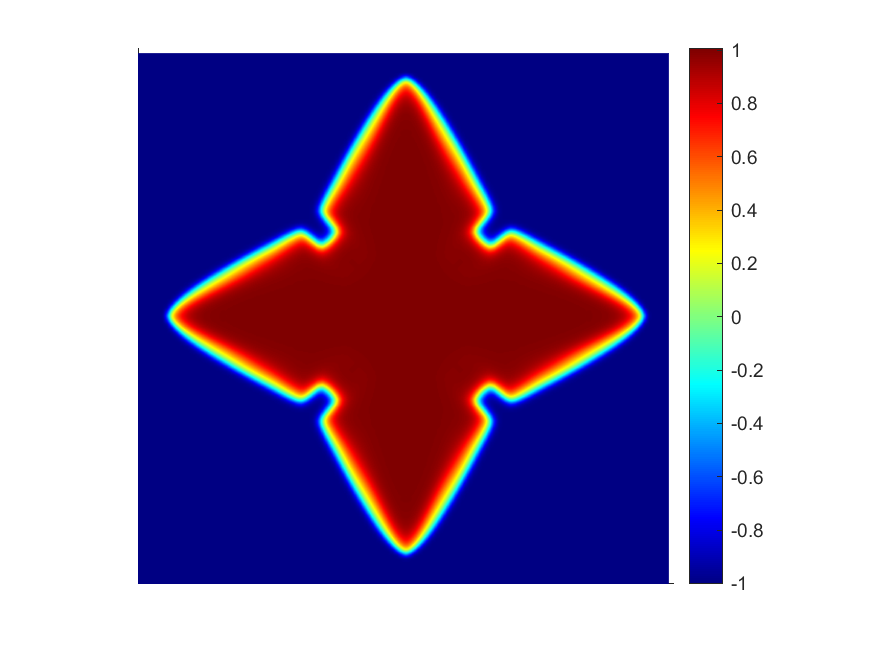}}
		\end{minipage}
	\centerline{(a) Phase field $\phi(\cdot,\cdot,\pi)$}
	\vskip 3mm
	\begin{minipage}[t]{0.24\linewidth}
			\centerline{\includegraphics[scale=0.28]{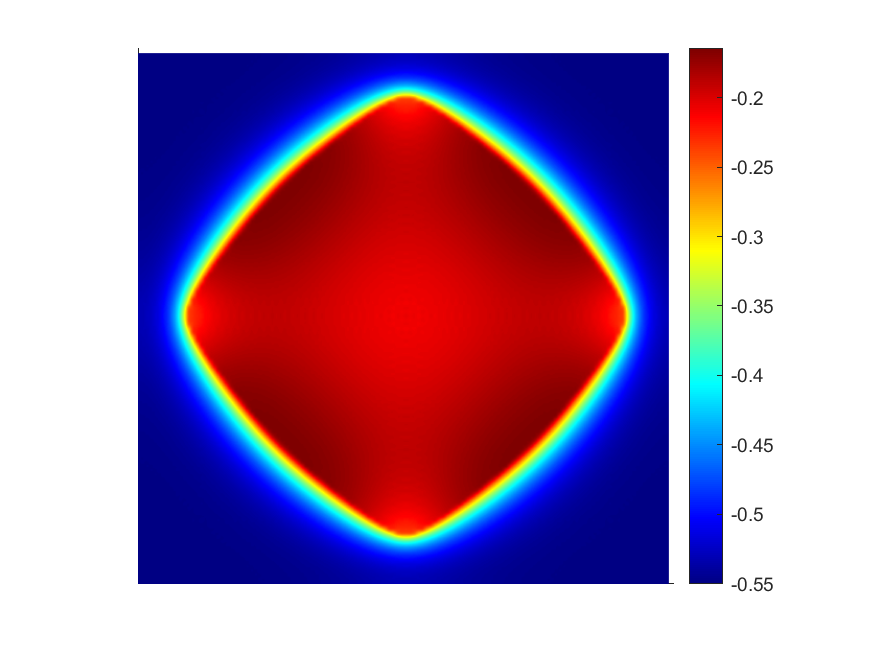}}
		\end{minipage}
	\begin{minipage}[t]{0.24\linewidth}
			\centerline{\includegraphics[scale=0.28]{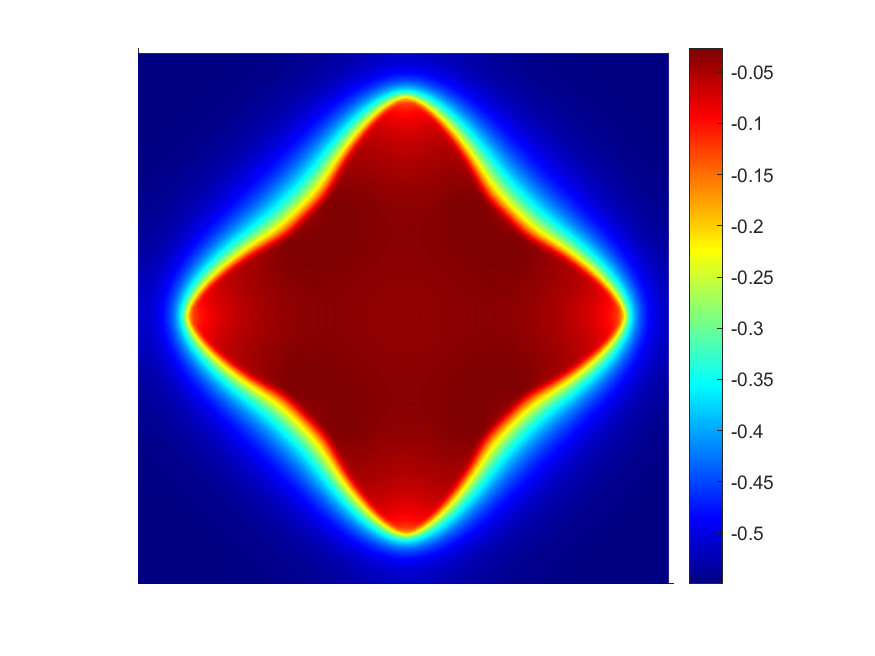}}
		\end{minipage}
	\begin{minipage}[t]{0.24\linewidth}
			\centerline{\includegraphics[scale=0.28]{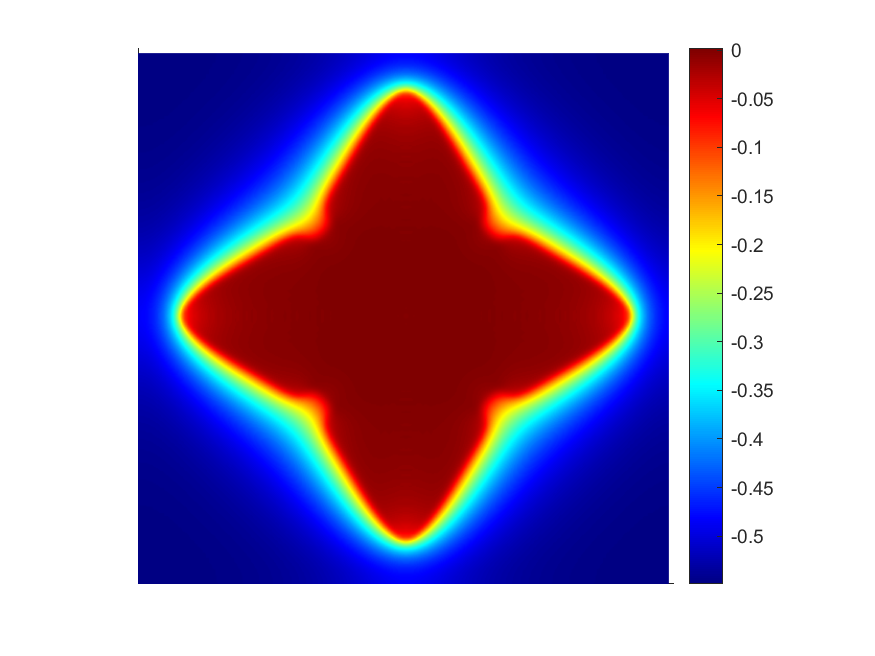}}
		\end{minipage}
	\begin{minipage}[t]{0.24\linewidth}
			\centerline{\includegraphics[scale=0.28]{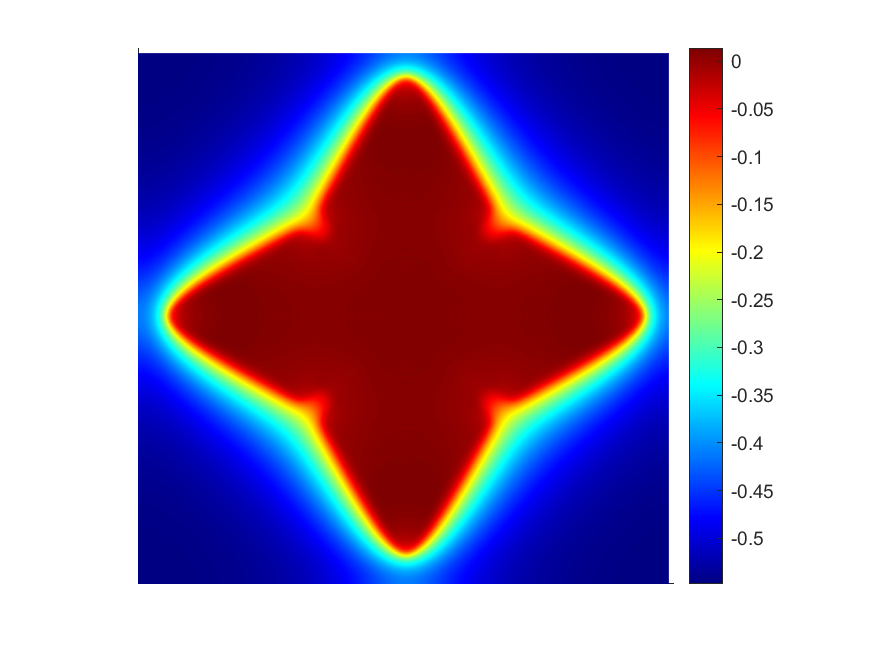}}
		\end{minipage}
	\centerline{(b) Temperature $T(\cdot,\cdot,\pi)$}
	\caption{Cross-sectional phase-field and temperature distributions at $z=\pi$ for different values of $K$. From left to right, $K=0.5$, $1.0$, $1.5$, and $2.0$, with the corresponding snapshots taken at $t=15$, $20$, $25$, and $30$, respectively.
	}\label{fig_example6_2d}
\end{figure*}

\begin{figure*}[htbp]
	\begin{minipage}[t]{0.49\linewidth}
		\centerline{\includegraphics[scale=0.5]{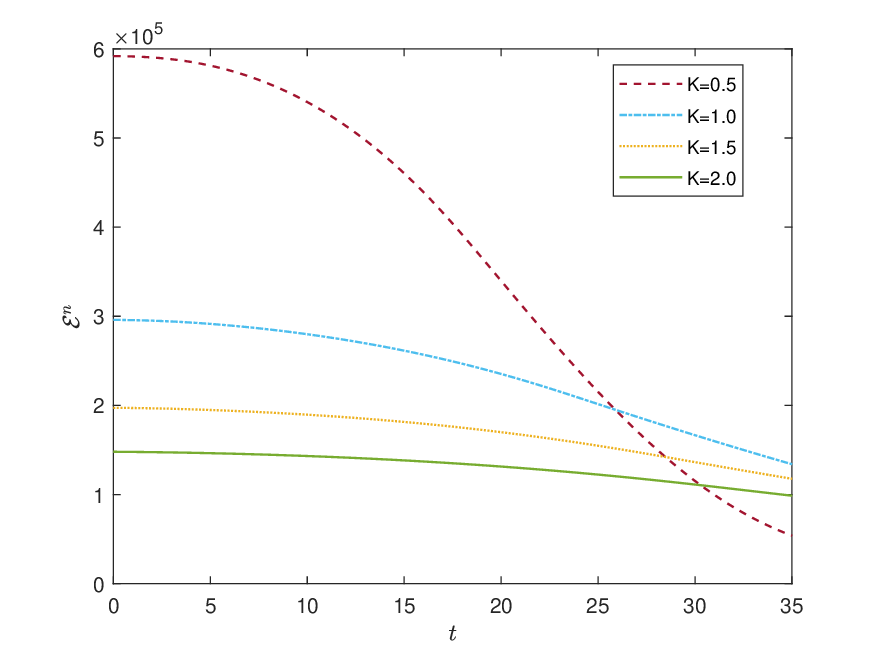}}
		\centerline{(a) Modified energy $\mathcal{E}_{n}$}
	\end{minipage}
	\begin{minipage}[t]{0.49\linewidth}
		\centerline{\includegraphics[scale=0.5]{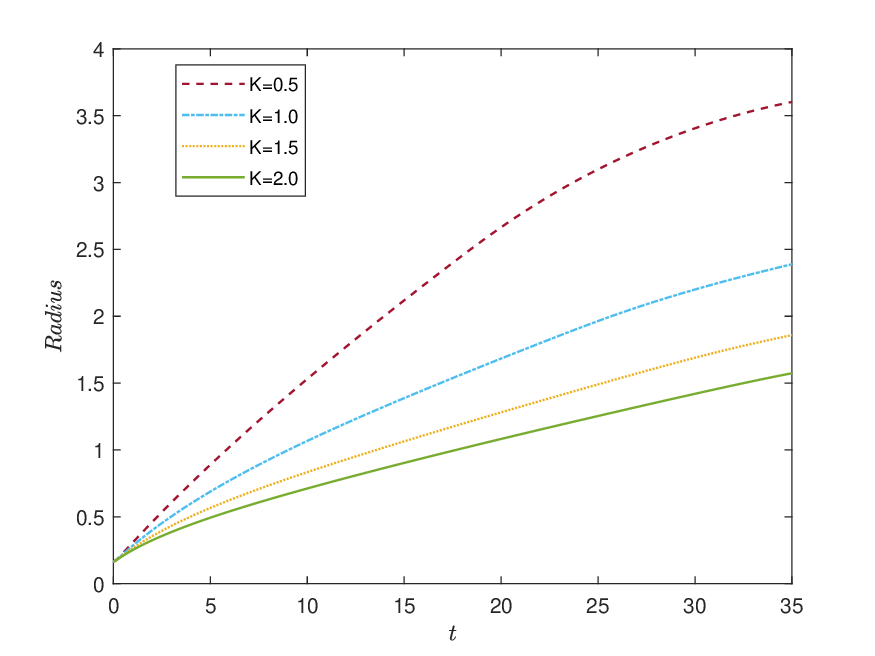}}
		\centerline{(b) Equivalent crystal radius}
	\end{minipage}
	\caption{Evolution of the modified energy $\mathcal E_n$ and equivalent crystal radius for three-dimensional fourfold anisotropy with different values of the latent heat coefficient $K$.
	}\label{fig_example6}
\end{figure*}

\section{Conclusion}\label{sec5}
This work develops arbitrary-order time discretizations for the anisotropic phase-field dendritic crystal growth model by combining a time dependent auxiliary variable with algebraically stable Runge-Kutta methods. The resulting schemes satisfy an unconditional discrete energy dissipation law. The main computational cost arises from the coupled elliptic system associated with the stages. A matrix diagonalization technique is used to transform the system into independent elliptic equations with constant coefficients. The auxiliary variable is then determined from a small $q\times q$ algebraic system. For a fixed Fourier-Galerkin space, we also establish a rigorous temporal error analysis and prove $q$th-order convergence for the scheme based on a $q$-stage Gauss method, providing a theoretical justification for the temporal accuracy observed in the numerical experiments.

The numerical results are consistent with the convergence analysis and the discrete energy law, and demonstrate the computational advantage of the decoupled implementation. In particular, the CPU-time comparisons show that the proposed implementation achieves comparable accuracy at substantially lower computational cost than the original coupled scheme. Two- and three-dimensional simulations were also used to examine the influence of anisotropy, latent heat, orientation angle, and initial nuclei on dendritic growth. The simulations reproduce characteristic fourfold and sixfold dendritic patterns, as well as the interaction of multiple nuclei and three-dimensional crystal growth.
	
Finally, we mention two directions for future work. The present convergence analysis is carried out for a fixed Fourier-Galerkin space, and a fully discrete error analysis remains to be investigated. It would also be of interest to extend the present framework to flow-coupled phase-field models for dendritic solidification \cite{beckermann1999modeling,tong2001phase}.

\appendix

\section*{Appendix. Butcher tableaux}\label{Butcher tableaux}
For completeness, we list the Butcher tableaux of the Gauss methods used in this paper. The corresponding coefficients can be found in  \cite{butcher1964implicit,ehle1969pade,tang2022arbitrarily,wang2023stable,wang2025class}.

\textbf{Tableau 1:} 2-stage Gauss method.
\beq
\begin{array}{cc|c}
	\dfrac{1}{4} & \dfrac{1}{4}-\dfrac{\sqrt{3}}{6} & \dfrac{1}{2}-\dfrac{\sqrt{3}}{6} \\[10pt]
	\dfrac{1}{4}+\dfrac{\sqrt{3}}{6} & \dfrac{1}{4} & \dfrac{1}{2}+\dfrac{\sqrt{3}}{6} \\[10pt]
	\hline
	\rule{0pt}{20pt} \dfrac{1}{2} & \rule{0pt}{20pt}  \dfrac{1}{2} &
\end{array}.
\eeq

\textbf{Tableau 2:} 3-stage Gauss method.
\beq
\begin{array}{ccc|c} 
	\dfrac{5}{36} & \dfrac{2}{9}-\dfrac{\sqrt{15}}{15} & \dfrac{5}{36}-\dfrac{\sqrt{15}}{30} & \dfrac{1}{2}-\dfrac{\sqrt{15}}{10} \\[10pt]
	\dfrac{5}{36}+\dfrac{\sqrt{15}}{24} & \dfrac{2}{9} & \dfrac{5}{36}-\dfrac{\sqrt{15}}{24} & \dfrac{1}{2} \\[10pt]
	\dfrac{5}{36}+\dfrac{\sqrt{15}}{30} & \dfrac{2}{9}+\dfrac{\sqrt{15}}{15} & \dfrac{5}{36} & \dfrac{1}{2}+\dfrac{\sqrt{15}}{10} \\[10pt]
	\hline
	\rule{0pt}{20pt}\dfrac{5}{18} & \rule{0pt}{20pt}\dfrac{4}{9} & \rule{0pt}{20pt}\dfrac{5}{18} &
\end{array}.
\eeq

\textbf{Tableau 3:} 4-stage Gauss method.
\beq
\begin{array}{cccc|c}
	\omega_1 & \omega_1'-\omega_3+\omega_4' & \omega_1'-\omega_3-\omega_4' & \omega_1-\omega_5 & \dfrac{1}{2}-\omega_2 \\[10pt]
	\omega_1-\omega_3'+\omega_4 & \omega_1' & \omega_1'-\omega_5' & \omega_1-\omega_3{ }'-\omega_4 & \dfrac{1}{2}-\omega_2' \\[10pt]
	\omega_1+\omega_3'+\omega_4 & \omega_1'+\omega_5' & \omega_1' & \omega_1+\omega_3'-\omega_4 & \dfrac{1}{2}+\omega_2' \\[10pt]
	\omega_1+\omega_5 & \omega_1'+\omega_3+\omega_4' & \omega_1'+\omega_3-\omega_4' & \omega_{1} & \dfrac{1}{2}+\omega_2 \\[8pt]
	\hline 	
	\rule{0pt}{12pt}2\omega_1 & \rule{0pt}{12pt}2 \omega_1' & \rule{0pt}{12pt}2 \omega_1' & \rule{0pt}{12pt}2 \omega_1 &
\end{array},
\eeq
where 
\begin{equation*}
\begin{aligned}
	\omega_1&=\tfrac18-\tfrac{\sqrt{30}}{144},\quad
	\omega_1'=\tfrac18+\tfrac{\sqrt{30}}{144},\quad
	\omega_2=\tfrac12\sqrt{\tfrac{15+2\sqrt{30}}{35}},\quad
	\omega_2'=\tfrac12\sqrt{\tfrac{15-2\sqrt{30}}{35}},\\
	\omega_3&=\omega_2\left(\tfrac16+\tfrac{\sqrt{30}}{24}\right),\quad
	\omega_3'=\omega_2'\left(\tfrac16-\tfrac{\sqrt{30}}{24}\right),\quad
	\omega_4=\omega_2\left(\tfrac1{21}+\tfrac{5\sqrt{30}}{168}\right),\\
	\omega_4'&=\omega_2'\left(\tfrac1{21}-\tfrac{5\sqrt{30}}{168}\right),\quad
	\omega_5=\omega_2-2\omega_3,\quad
	\omega_5'=\omega_2'-2\omega_3'.
\end{aligned}
\end{equation*}

\vspace{2mm}
\textbf{Tableau 4:} 5-stage Gauss method.
\beq
\begin{array}{ccccc|c}
	\eta_1 & \eta_1'-\eta_3+\eta_4' & \eta_8-\eta_5 & \eta_1'-\eta_3-\eta_4' & \eta_1-\eta_6 & \dfrac{1}{2}-\eta_2 \\[10pt]
	\eta_1-\eta_3'+\eta_4 & \eta_1' & \eta_8-\eta_5' & \eta_1'-\eta_6' & \eta_1-\eta_3'-\eta_4 & \dfrac{1}{2}-\eta_2' \\[10pt]
	\eta_1+\eta_7 & \eta_1'+\eta_7' & \eta_8 & \eta_1'-\eta_7' & \eta_1-\eta_7 & \dfrac{1}{2} \\[10pt]
	\eta_1+\eta_3'+\eta_4 & \eta_1'+\eta_6' & \eta_8+\eta_5' & \eta_1' & \eta_1+\eta_3'-\eta_4 & \dfrac{1}{2}+\eta_2' \\[10pt]
	\eta_1+\eta_6 & \eta_1'+\eta_3+\eta_4' & \eta_8+\eta_5 & \eta_1'+\eta_3-\eta_4' & \eta_1 & \dfrac{1}{2}+\eta_2 \\[8pt]
	\hline 	
	\rule{0pt}{12pt}2\eta_1 & \rule{0pt}{12pt}2\eta_1' & \rule{0pt}{12pt}2\eta_8 & \rule{0pt}{12pt}2\eta_1' & \rule{0pt}{12pt}2\eta_1 &
\end{array},
\eeq
where 
\begin{equation*}
\begin{aligned}
	\eta_1&=\tfrac{322-13\sqrt{70}}{3600},\quad
	\eta_1'=\tfrac{322+13\sqrt{70}}{3600},\quad
	\eta_2=\tfrac12\sqrt{\tfrac{35+2\sqrt{70}}{63}},\quad
	\eta_2'=\tfrac12\sqrt{\tfrac{35-2\sqrt{70}}{63}},\\
	\eta_3&=\eta_2\left(\tfrac{452+59\sqrt{70}}{3240}\right),\quad
	\eta_3'=\eta_2'\left(\tfrac{452-59\sqrt{70}}{3240}\right),\quad
	\eta_4=\eta_2\left(\tfrac{64+11\sqrt{70}}{1080}\right),\quad
	\eta_4'=\eta_2'\left(\tfrac{64-11\sqrt{70}}{1080}\right),\\
	\eta_5&=8\eta_2\left(\tfrac{23-\sqrt{70}}{405}\right),\quad
	\eta_5'=8\eta_2'\left(\tfrac{23+\sqrt{70}}{405}\right),\quad
	\eta_6=\eta_2-2\eta_3-\eta_5,\quad
	\eta_6'=\eta_2'-2\eta_3'-\eta_5',\\
	\eta_7&=\eta_2\left(\tfrac{308-23\sqrt{70}}{960}\right),\quad
	\eta_7'=\eta_2'\left(\tfrac{308+23\sqrt{70}}{960}\right),\quad
	\eta_8=\tfrac{32}{225}.
\end{aligned}
\end{equation*}

\bibliographystyle{unsrt}
\bibliography{reference}
\end{document}